\documentclass[article,onefignum,onetabnum]{siamonline220329}
\usepackage{_settings}

\title{The linear sampling method for data generated by small random scatterers\thanks{Submitted to the editors DATE.
\funding{This work was partially supported by Agence de l'Innovation de D\'efense (AID) via Centre Interdisciplinaire d'\'{E}tudes pour la D\'{e}fense et la S\'{e}curit\'{e} (CIEDS, project PRODIPO).}}}

\headers{LSM for data generated by small random scatterers}{J. Garnier, H. Haddar, and H. Montanelli}

\ifpdf
\hypersetup{
  pdftitle={The linear sampling method for data generated by small random scatterers},
  pdfauthor={J. Garnier, H. Haddar, and H. Montanelli}
}
\fi

\author{Josselin Garnier\thanks{CMAP, \'{E}cole Polytechnique, IP Paris, 91120 Palaiseau, France.}
\and Houssem Haddar\thanks{Inria, UMA, ENSTA Paris, IP Paris, 91120 Palaiseau, France.}
\and Hadrien Montanelli\footnotemark[3]}

\begin{document}

\maketitle

\begin{abstract}
We present an extension of the linear sampling method for solving the sound-soft inverse scattering problem in two dimensions with data generated by randomly distributed small scatterers. The theoretical justification of our novel sampling method is based on a rigorous asymptotic model, a modified Helmholtz--Kirchhoff identity, and our previous work on the linear sampling method for random sources. Our numerical implementation incorporates boundary elements, Singular Value Decomposition, Tikhonov regularization, and Morozov's discrepancy principle. We showcase the robustness and accuracy of our algorithms with a series of numerical experiments.
\end{abstract}



\section{Introduction}

Inverse scattering problems arise across a multitude of fields, ranging from medical imaging and non-destructive testing to radar technology and seismology. At their core, these problems involve the task of deducing the characteristics of an object or medium from the scattered signals it generates. These problems involve several challenges, including nonlinearity, which is particularly pronounced near resonance, making linearization inapplicable; these are also severely ill-posed, raising questions about uniqueness and stability, and necessitating the inclusion of regularization. Additionally, for iterative procedures, reconstruction time can be lengthy and prior knowledge is required for initialization. In this context, fast, data-driven algorithms, such as the linear sampling method (LSM), can be useful in reducing computational costs \cite{colton2003, colton1996, colton1997}, and initializing more sophisticated algorithms \cite{bao2007}. For an exploration of the history and evolution of the LSM, readers can refer to the 2018 \textit{SIAM Review} article by Colton and Kress \cite{colton2018}. Further mathematical insights can be found in dedicated books \cite{cakoni2014, cakoni2016, colton2019}; recent developments include the generalized LSM, in both full- and limited-aperture measurements \cite{audibert2014, audibert2017}.

Two acquisition configurations are commonly used in inverse scattering problems---active and passive. Active imaging involves sending waves through \textit{controlled} sources and recording the medium's response through \textit{controlled} sensors. Passive imaging, on the other hand, employs controlled sensors but relies on \textit{random}, \textit{uncontrolled} sources (such as microseisms and ocean swells in seismology). In this setup, it is the cross-correlations between the recorded signals that convey information about the medium \cite{garnier2016, gouedard2008}. Passive imaging is a rapidly growing research topic because it enables the imaging of areas where the use of active sources is not possible due to, e.g., safety or environmental reasons. Additionally, even in scenarios where active sources could be employed, utilizing passive sources helps reduce operational costs and enhances stealth in defense applications. Other applications encompass crystal tomography and seismic interferometry for volcano monitoring \cite{koulakov2014, shapiro2005}. Passive imaging has also demonstrated success in other fields, including structural health monitoring \cite{duroux2010, sabra2011}, oceanography \cite{godin2010, siderius2010, woolfe2015}, and medical elastography \cite{gallot2011}. We illustrate active and passive imaging in the context of subsurface imaging in \cref{fig:imaging}.

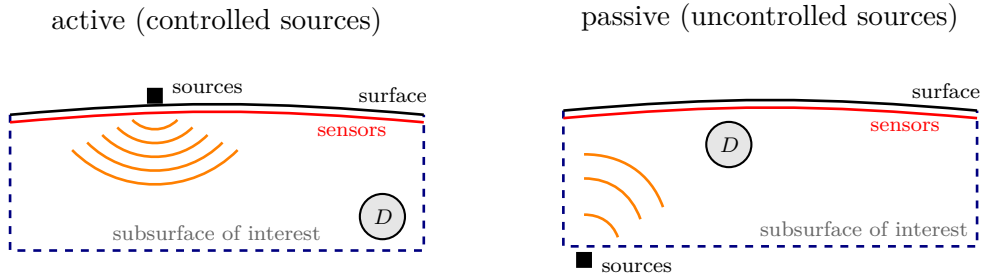
\begin{figure}
\hspace{0.75cm}
\vspace*{-2em}
{\raisebox{4mm}{
\begin{tikzpicture}[scale=1.0,every node/.style={font=\scriptsize}]

  \pgfmathsetmacro{\wid}{5.50} 
  \pgfmathsetmacro{\hei}{1.80}
  
  \coordinate (r1) at (0., 0.);  
  \coordinate (r2) at (\wid, 0.);
  \coordinate (r3) at (\wid, \hei); 
  \coordinate (r4) at (0., \hei); 
  \coordinate (r5) at (\wid, 0.95*\hei); 
  \coordinate (r6) at (0., 0.95*\hei); 
  \coordinate (rlame1) at (\wid, 0.90*\hei); 
  \coordinate (rlame2) at (0., 0.90*\hei);
  \coordinate (z) at (0.75*\wid, 0.5*\hei);
  \coordinate (D) at (0.9*\wid, 0.25*\hei);

  \node[text width=4cm,anchor=south east,white!40!black,yshift=-0mm,xshift=0cm] at (r2) {subsurface of interest};
  \draw[line width=1,color=blue!50!black,dashed] (r4) to (r1) to (r2) to (r3);  
  \draw[line width=1] (r3) to[out=175, in=05] (r4);  
  \draw[line width=1,color=red] ([yshift=-1mm] r3) to[out=175, in=05] ([yshift=-1mm] r4);  
  \pgfmathsetmacro{\srcx}{0.35*\wid}
  \pgfmathsetmacro{\srcy}{1.14*\hei}  
  \draw[mark=square*,mark size=2,mark options={color=black,line width=2}] plot[] coordinates{(\srcx, \srcy)};
  \node[text width=2cm,anchor=north west,red,yshift=-0.20mm,xshift=-1.55cm] at (r3) {sensors};
  \node[text width=2cm,anchor=north west,black,yshift=5mm,xshift=-1.0cm] at (r3) {surface};
  \node[text width=2cm,anchor=south west,black,yshift=-1mm,xshift=1mm] at (\srcx, \srcy) {sources};
  
  \draw[black, thick, fill=black!10] (D) circle (0.3);
  \node[anchor=west,black,yshift=0mm,xshift=-2.5mm] at (D) {$D$};

  \pgfmathsetmacro{\wsize}{0.3} \pgfmathsetmacro{\wdepth}{0.32}
  \coordinate (w1) at (\srcx-\wsize, \srcy - \wdepth);
  \coordinate (w2) at (\srcx+\wsize, \srcy - \wdepth);
  \draw[orange,line width=1.](w1) to[out=-45, in=225] (w2) ;
  \pgfmathsetmacro{\wsize}{0.5} \pgfmathsetmacro{\wdepth}{0.32+0.50*0.20}
  \coordinate (w1) at (\srcx-\wsize, \srcy - \wdepth);
  \coordinate (w2) at (\srcx+\wsize, \srcy - \wdepth);
  \draw[orange,line width=1.](w1) to[out=-45, in=225] (w2) ;
  \pgfmathsetmacro{\wsize}{0.70} \pgfmathsetmacro{\wdepth}{0.32+1.00*0.20}
  \coordinate (w1) at (\srcx-\wsize, \srcy - \wdepth);
  \coordinate (w2) at (\srcx+\wsize, \srcy - \wdepth);
  \draw[orange,line width=1.](w1) to[out=-45, in=225] (w2) ;
  \pgfmathsetmacro{\wsize}{0.90} \pgfmathsetmacro{\wdepth}{0.32+1.50*0.20}
  \coordinate (w1) at (\srcx-\wsize, \srcy - \wdepth);
  \coordinate (w2) at (\srcx+\wsize, \srcy - \wdepth);
  \draw[orange,line width=1.](w1) to[out=-45, in=225] (w2) ;
  \pgfmathsetmacro{\wsize}{1.10} \pgfmathsetmacro{\wdepth}{0.32+2.00*0.20}
  \coordinate (w1) at (\srcx-\wsize, \srcy - \wdepth);
  \coordinate (w2) at (\srcx+\wsize, \srcy - \wdepth);
  \draw[orange,line width=1.](w1) to[out=-45, in=225] (w2);

  \node[anchor=south] at (0.5*\wid, 1.5*\hei) {{\normalsize active (controlled sources)}};
  
\end{tikzpicture}
}}
\hspace{0.25cm} 
\begin{tikzpicture}[scale=1.0,every node/.style={font=\scriptsize}]

  \pgfmathsetmacro{\wid}{5.50} 
  \pgfmathsetmacro{\hei}{1.80}
  
  \coordinate (r1) at (0., 0.);  
  \coordinate (r2) at (\wid, 0.);
  \coordinate (r3) at (\wid, \hei); 
  \coordinate (r4) at (0., \hei); 
  \coordinate (r5) at (\wid, 0.95*\hei); 
  \coordinate (r6) at (0., 0.95*\hei); 
  \coordinate (rlame1) at (\wid, 0.90*\hei); 
  \coordinate (rlame2) at (0., 0.90*\hei);
  \coordinate (z) at (0.75*\wid, 0.5*\hei);
  \coordinate (D) at (0.4*\wid, 0.75*\hei);
	  
  \node[text width=3cm,anchor=south east,white!40!black,yshift=-0mm,xshift=0.25cm] at (r2) {subsurface of interest};
  \draw[line width=1,color=blue!50!black,dashed] (r4) to (r1) to (r2) to (r3);  
  \draw[line width=1] (r3) to[out=175, in=05] (r4);  
  \draw[line width=1,color=red] ([yshift=-1mm] r3) to[out=175, in=05] ([yshift=-1mm] r4);  
  \pgfmathsetmacro{\srcx}{0.05*\wid}
  \pgfmathsetmacro{\srcy}{-0.1*\hei}  
  \draw[mark=square*,mark size=2,mark options={color=black,line width=2}] plot[] coordinates{(\srcx, \srcy)};
  \node[text width=2cm,anchor=north west,red,yshift=-0.20mm,xshift=-1.55cm] at (r3) {sensors};
  \node[text width=2cm,anchor=north west,black,yshift=5mm,xshift=-1.0cm] at (r3) {surface};
  \node[text width=2cm,anchor=south west,black,yshift=-3mm,xshift=1mm] at (\srcx, \srcy) {sources};
  
  \draw[black, thick, fill=black!10] (D) circle (0.3);
  \node[anchor=west,black,yshift=0mm,xshift=-2.5mm] at (D) {$D$};
  
  \pgfmathsetmacro{\wsize}{0.45} \pgfmathsetmacro{\wdepth}{0.30}
  \coordinate (w1) at (\srcx+\wsize, \srcy + \wdepth);
  \coordinate (w2) at (\srcx, \srcy + 2*\wdepth);
  \draw[orange,line width=1.](w1) to[out=110, in=0] (w2) ;
  \pgfmathsetmacro{\wsize}{0.75} \pgfmathsetmacro{\wdepth}{0.54}
  \coordinate (w1) at (\srcx+\wsize, \srcy + \wdepth);
  \coordinate (w2) at (\srcx, \srcy + 2*\wdepth);
  \draw[orange,line width=1.](w1) to[out=110, in=0] (w2) ;
  \pgfmathsetmacro{\wsize}{1.05} \pgfmathsetmacro{\wdepth}{0.70}
  \coordinate (w1) at (\srcx+\wsize, \srcy + \wdepth);
  \coordinate (w2) at (\srcx, \srcy + 2*\wdepth);
  \draw[orange,line width=1.](w1) to[out=110, in=0] (w2);

  \node[anchor=south] at (0.5*\wid, 1.5*\hei) {{\normalsize passive (uncontrolled sources)}};
  
\end{tikzpicture}
\caption{In subsurface imaging, the objective is to capture an image of a specific subsurface area to identify a defect, denoted as $D$---this could involve, for instance, locating shallow or deep geothermal reservoirs. In active imaging (on the left), both the sources and sensors are controlled. Vibroseismic trucks on the surface can generate signals, and sensors can be strategically placed and moved as desired. In passive imaging (on the right), controlled sensors are still utilized, but the signal originates from uncontrolled, random sources within the subsurface, such as microseisms and ocean swells.}
\label{fig:imaging}
\end{figure}

In a previous paper \cite{montanelli2023}, we introduced an extension of the LSM to address the sound-soft inverse scattering problem involving random sources. This current study builds upon that foundation, focusing on a scenario where a single \textit{controlled} source, operating at a specified wavelength $\lambda$, illuminates a small \textit{random} scatterer and an object of comparable size to $\lambda$---our goal is to reconstruct the shape of the latter. The reflection of the wave field transmitted by the point source on the small scatterer serves as a random source, aligning with the context explored in our earlier work. The motivation for this research stems from the necessity, when imaging an unknown obstacle $D$, to illuminate it with a diverse set of incoming waves. For deterministic, controlled point sources in two dimensions, one considers a family of sources located at points $\bs{z}_m$ generating the incident fields
\begin{align}\label{eq:pt-src}
\phi(\bs{x},\bs{z}_m) = \frac{i}{4}H_0^{(1)}(k\vert\bs{x}-\bs{z}_m\vert),
\end{align} 
and measures the resulting scattered fields $u^s(\bs{x}_j,\bs{z}_m)$ at points $\bs{x}_j$ to populate the near-field matrix $N$ with entries $N_{jm} = u^s(\bs{x}_j,\bs{z}_m)$. In the case of random, uncontrolled sources at positions $\bs{z}_\ell$, as demonstrated in our prior work \cite{montanelli2023}, the relevant matrix is the cross-correlation matrix $C$. Its entries are given by:
\begin{align}\label{eq:cross-cor-mat}
C_{jm} = \frac{2ik\vert\Sigma\vert}{L}\sum_{\ell=1}^L\overline{u(\bs{x}_j,\bs{z}_\ell)}u(\bs{x}_m,\bs{z}_\ell) - \left[\phi(\bs{x}_j,\bs{x}_m) - \overline{\phi(\bs{x}_j,\bs{x}_m)}\right],
\end{align}
where $u(\bs{x},\bs{z})=\phi(\bs{x},\bs{z})+u^s(\bs{x},\bs{z})$ is the total field for the incident wave $\phi(\bs{x},\bs{z})$, and $\vert\Sigma\vert$ is the area of the surface $\Sigma$ on which the random sources are distributed. (Note that it is not necessary to know or estimate the positions $\bs{z}_\ell$ to assemble the matrix $C$.) 

Suppose now that we only have a single \textit{controlled} source located at a given point $\bs{z}$, as opposed to several positions $\bs{z}_m$. This is insufficient for reconstructing the obstacle's shape with the LSM (the near-field matrix $N$ would have rank one). To address this limitation, we propose introducing a random medium between the source and the obstacle. We illustrate this concept, in acoustic scattering, by considering a single small random scatterer between the source and the obstacle---a seemingly simple yet powerful model of a random medium that allows us to apply the LSM in a novel manner, with a modified version of \cref{eq:cross-cor-mat}. The extension to elastic waves in subsurface imaging will be the subject of future work.

Our method is supported by various key components, such as a rigorous asymptotic model (see \cref{sec:asymptotics}) and a \textit{modified} Helmholtz--Kirchhoff identity (discussed in \cref{sec:mod-HK-identity}). The paper explores numerical implementations, utilizing boundary elements alongside Singular Value Decomposition (SVD), Tikhonov regularization, and Morozov's discrepancy principle (detailed in \cref{sec:numerics}). A set of numerical experiments is presented in the same section, offering valuable insights into the practical applications of our research.

\section{Asymptotic model}\label{sec:asymptotics}

Mathematically, the configuration described in the introduction unfolds as follows. Consider the incident field \cref{eq:pt-src} generated by a point source located at $\bs{z}_\epsilon=\lambda\epsilon^{-q}e^{i\theta_z}$ for some scalars $q>0$ and $\theta_z\in[0,2\pi]$. (Complex variables will be employed to determine the coordinates of points in the plane.) Here, $\epsilon>0$ is a small dimensionless parameter, $k>0$ is the wavenumber and $\lambda=2\pi/k$ is the wavelength, which we assume to be independent of $\epsilon$. Let $D$ be an obstacle of size proportional to $\lambda$ and independent of $\epsilon$, i.e., there exists a constant $C>0$ such that its radius verifies $\rho(D)=\frac{1}{2}\sup_{\bs{x},\bs{y}\in D}\vert\bs{x}-\bs{y}\vert=C\lambda$. Without loss of generality, we assume that the geometric center of the obstacle $D$ is at the origin. We also consider a small disk $D_\epsilon=D_\epsilon(\bs{y}_\epsilon)$ of radius $\rho(D_\epsilon)=\lambda\epsilon$ centered at $\bs{y}_\epsilon = \lambda\epsilon^{-p}e^{i\theta_y}$ for some scalars $0<p<q$ and $\theta_y\in[0,2\pi]$.\footnote{Throughout the paper, we will assume that $k^2$ is not a Dirichlet eigenvalue of $-\Delta$ in both $D$ and $D_\epsilon$. For the latter, this assumption is always verified for small enough $\epsilon$.} Finally, let $B\subset\R^2\setminus\overline{D\cup D_\epsilon}$ be a compact set whose size and distance to $D$ are proportional to $\lambda$ and independent of $\epsilon$, i.e., $\rho(B) \propto \lambda$ and $d(B,D)=\inf_{\bs{x}\in B,\,\bs{y}\in D}\vert\bs{x}-\bs{y}\vert \propto \lambda$. Measurements will be taken inside the volume $B$, which is consistent with \cite{montanelli2023}. We also assume that $\partial D$ and $\partial B$ are smooth enough to allow the forming of Dirichlet and Neumann traces and the application of partial integration formulas (Lipschitz continuity is a sufficient condition). To summarize, we make the following assumptions (see also \cref{fig:w^s}):\footnote{We recall that $f(\epsilon)=\OO(\epsilon)$ means that there exists a constant $C>0$, independent of $\epsilon$, such that for small enough $\epsilon$, $\vert f(\epsilon)\vert \leq C \epsilon$.}
\begin{align}
& \bs{y}_\epsilon = \lambda\epsilon^{-p}e^{i\theta_y}, \quad \bs{z}_\epsilon = \lambda\epsilon^{-q}e^{i\theta_z}, \quad 0 < p < q, \\
& \rho(D),\,\rho(B),\,d(B, D) = \OO(1), \quad \rho(D_\epsilon) = \OO(\epsilon). \nonumber
\end{align}

We examine the scattering of the incident field $\phi(\cdot,\bs{z}_\epsilon)$ by $D$ and $D_\epsilon$, which generates the scattered field $w^s_\epsilon$. More precisely, let $w^s_\epsilon(\cdot,\bs{y}_\epsilon,\bs{z}_\epsilon)\in H^1_{\mrm{loc}}(\R^2\setminus\{D\cup D_\epsilon\})$ be the solution to the sound-soft scattering problem
\begin{align}\label{eq:w^s}
\left\{
\begin{array}{ll}
\Delta w^s_\epsilon(\cdot,\bs{y}_\epsilon,\bs{z}_\epsilon) + k^2 w^s_\epsilon(\cdot,\bs{y}_\epsilon,\bs{z}_\epsilon) = 0 \quad \text{in $\R^2\setminus\{\overline{D\cup D_\epsilon}\}$}, \\[0.4em]
w^s_\epsilon(\cdot,\bs{y}_\epsilon,\bs{z}_\epsilon) = -\phi(\cdot,\bs{z}_\epsilon) \quad \text{on $\partial D\cup\partial D_\epsilon$}, \\[0.4em]
\text{$w^s_\epsilon(\cdot,\bs{y}_\epsilon,\bs{z}_\epsilon)$ is radiating}.
\end{array}
\right.
\end{align}
Note that the (Sommerfeld) radiation condition in \cref{eq:w^s} reads
\begin{align}\label{eq:Sommerfeld}
\lim_{\vert\bs{x}\vert\to\infty} \sqrt{\vert\bs{x}\vert}\left(\frac{\bs{x}}{\vert\bs{x}\vert}\nabla_{\bs{x}} w^s_\epsilon(\bs{x},\bs{y}_\epsilon,\bs{z}_\epsilon) - ikw^s_\epsilon(\bs{x},\bs{y}_\epsilon,\bs{z}_\epsilon)\right) = 0, \quad \text{(uniformly in $\bs{x}/\vert\bs{x}\vert$)}.
\end{align}
The condition \cref{eq:Sommerfeld} ensures that the solution represents an outgoing wave.

\begin{figure}
\centering
\begin{tikzpicture}[scale=0.9]

\def\R{3.5};
\def\s{0.55};
\def\t{8};
\def\x{-0.9238795325112867};
\def\y{0.38268343236508984};
\def\zz{3}

\draw[thick, fill=black!10] plot[smooth cycle] coordinates {(-2*\R/\t, -1.5*\R/\t) (-2*\R/\t, 1.5*\R/\t) (2*\R/\t, \R/\t) (2*\R/\t, -\R/\t)};
\node[anchor=north] at (0, \R/\t) {$D(\bs{0})$};

\draw[black, thick, fill=black!10] plot[smooth cycle] coordinates {(4*\R/\t, 4*\R/\t) (6.5*\R/\t, 0*\R/\t) (4*\R/\t, -4*\R/\t) (3.5*\R/\t, -\R/\t) (3.5*\R/\t, \R/\t)};
\node[anchor=north] at (4.75*\R/\t, \R/\t) {$B$};

\node[draw, fill, circle, scale=\s] (A) at (-\zz*\R, 0.00*\R) {};
\node[anchor=east] at (-\zz*\R*1.025, 0.00*\R) {$\phi(\cdot,\bs{z}_\epsilon)$};
\node (B) at (-0.25*\R, 0.00*\R) {};
\draw[thick, black, -{Stealth[bend]}] (A) to [bend right=30] (B);
\node[anchor=north] at (-0.50*\zz*\R, -0.50*\R) {$u^s$};

\node[draw, black, fill=black!10, circle, thick, scale=1.25*\s] (C) at (1.25*\R*\x, 1.25*\R*\y) {};
\node[anchor=south] at (1.25*\R*\x, 1.25*\R*\y*1.05) {$D_\epsilon(\bs{y}_\epsilon)$};
\draw[thick, black, -{Stealth[bend]}] (A) to [bend left=30] (C);
\node[anchor=south] at (-2.50*\R, 0.45*\R) {$v^i_\epsilon$};

\draw[thick, black, -{Stealth[bend]}] (C) to [bend left=10] (B);
\node[anchor=south] at (-0.6*\R, 0.30*\R) {$v^s_\epsilon$};

\end{tikzpicture}
\caption{The incident field $\phi(\cdot,\bs{z}_\epsilon)$ is generated by a point source located at $\bs{z}_\epsilon=\lambda\epsilon^{-q}e^{i\theta_z}$. It is scattered by a small disk $D_\epsilon$ of radius $\OO(\epsilon)$ centered at $\bs{y}_\epsilon=\lambda\epsilon^{-p}e^{i\theta_y}$ and an obstacle $D$ of radius $\OO(1)$ centered at the origin. Measurements are taken near the obstacle $D$ in a volume $B$, whose size is also $\OO(1)$. The resulting scattered field, $w^s_\epsilon$, can be approximated with an error $\OO(\vert\log\epsilon\vert^{-1}\epsilon^p\epsilon^{q/2})$ in the $H^1(B)$-norm by the three-term sum $u^s+v^i_\epsilon+v^s_\epsilon$. As $\epsilon\to0$, the radius of $D_\epsilon$ goes to $0$, and $\bs{y}_\epsilon$ and $\bs{z}_\epsilon$ shoot to infinity. The scattering sequences, along with the amplitudes of the scattered fields, are also shown in \cref{tab:asymptotics}.}
\label{fig:w^s}
\end{figure}
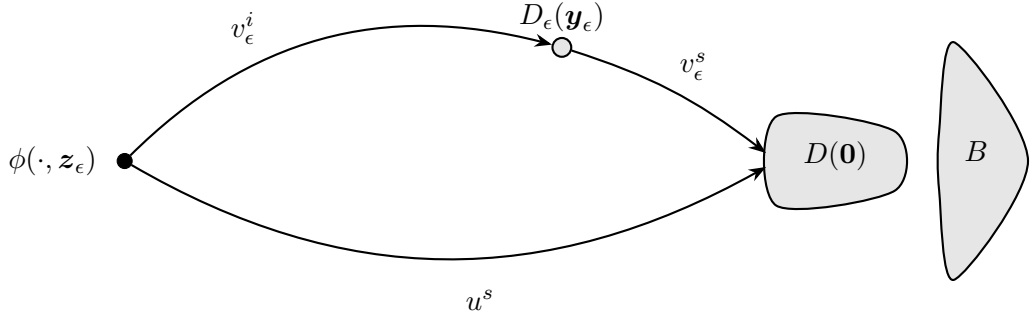

We shall show that $w^s_\epsilon(\cdot,\bs{y}_\epsilon,\bs{z}_\epsilon)$ can be approximated as the sum of three terms,
\begin{align}
w_\epsilon^s(\cdot,\bs{y}_\epsilon,\bs{z}_\epsilon) \approx u^s(\cdot,\bs{z}_\epsilon) + v_\epsilon^i(\cdot,\bs{y}_\epsilon,\bs{z}_\epsilon) + v_\epsilon^s(\cdot,\bs{y}_\epsilon,\bs{z}_\epsilon)
\end{align}
with an error $\OO(\vert\log\epsilon\vert^{-1}\epsilon^p\epsilon^{q/2})$ in the $H^1(B)$-norm; the scattered fields in the expansion above, $u^s(\cdot,\bs{z}_\epsilon)\in H^1_{\mrm{loc}}(\R^2\setminus D)$, $v_\epsilon^i(\cdot,\bs{y}_\epsilon,\bs{z}_\epsilon)\in H^1_{\mrm{loc}}(\R^2\setminus D_\epsilon)$, and $v_\epsilon^s(\cdot,\bs{y}_\epsilon,\bs{z}_\epsilon)\in H^1_{\mrm{loc}}(\R^2\setminus D)$, are the solutions to the sound-soft scattering problems
\begin{align}\label{eq:u^s}
\left\{
\begin{array}{ll}
\Delta u^s(\cdot,\bs{z}_\epsilon) + k^2 u^s(\cdot,\bs{z}_\epsilon) = 0 \quad \text{in $\R^2\setminus\overline{D}$}, \\[0.4em]
u^s(\cdot,\bs{z}_\epsilon) = -\phi(\cdot,\bs{z}_\epsilon) \quad \text{on $\partial D$}, \\[0.4em]
\text{$u^s(\cdot,\bs{z}_\epsilon)$ is radiating},
\end{array}
\right.
\end{align}
\begin{align}\label{eq:v^i}
\left\{
\begin{array}{ll}
\Delta v_\epsilon^i(\cdot,\bs{y}_\epsilon,\bs{z}_\epsilon) + k^2 v_\epsilon^i(\cdot,\bs{y}_\epsilon,\bs{z}_\epsilon) = 0 \quad \text{in $\R^2\setminus\overline{D_\epsilon}$}, \\[0.4em]
v_\epsilon^i(\cdot,\bs{y}_\epsilon,\bs{z}_\epsilon) = -\phi(\cdot,\bs{z}_\epsilon) \quad \text{on $\partial D_\epsilon$}, \\[0.4em]
\text{$v_\epsilon^i(\cdot,\bs{y}_\epsilon,\bs{z}_\epsilon)$ is radiating},
\end{array}
\right.
\end{align}
and
\begin{align}\label{eq:v^s}
\left\{
\begin{array}{ll}
\Delta v_\epsilon^s(\cdot,\bs{y}_\epsilon,\bs{z}_\epsilon) + k^2 v_\epsilon^s(\cdot,\bs{y}_\epsilon,\bs{z}_\epsilon) = 0 \quad \text{in $\R^2\setminus\overline{D}$}, \\[0.4em]
v_\epsilon^s(\cdot,\bs{y}_\epsilon,\bs{z}_\epsilon) = -v^i_\epsilon(\cdot,\bs{y}_\epsilon,\bs{z}_\epsilon) \quad \text{on $\partial D$}, \\[0.4em]
\text{$v_\epsilon^s(\cdot,\bs{y}_\epsilon,\bs{z}_\epsilon)$ is radiating}.
\end{array}
\right.
\end{align}
We illustrate the scattering problems \cref{eq:w^s} and \cref{eq:u^s}--\cref{eq:v^s} in \cref{fig:w^s}.

We start by proving estimates on $u^s(\cdot,\bs{z}_\epsilon)$, which corresponds to the wave field transmitted by the point source $\phi(\cdot,\bs{z}_\epsilon)$ scattered by $D$ alone. 

\begin{lemma}[Estimates on $u^s$]\label{lem:u^s}
There exists a constant $C>0$, independent of $\epsilon$, such that for small enough $\epsilon$,
\begin{align}
\Vert u^s(\cdot,\bs{z}_\epsilon)\Vert_{H^1(B)} \leq C \epsilon^{q/2}
\end{align}
and
\begin{align}
\Vert u^s(\cdot,\bs{z}_\epsilon)\Vert_{H^{1/2}(\partial D_\epsilon)} \leq C \epsilon^{p/2}\epsilon^{q/2}.
\end{align}
\end{lemma}

\begin{proof}
We write $u^s(\cdot,\bs{z}_\epsilon)=S g(\cdot,\bs{z}_\epsilon)$ on $\partial D$ with the single layer operator
\begin{align}
& S:H^{-1/2}(\partial D)\to H^{1/2}(\partial D), \\
& Sg(\bs{x}) = \int_{\partial D}\phi(\bs{x},\bs{y})g(\bs{y})ds(\bs{y}), \quad \bs{x}\in \partial D. \nonumber
\end{align}
Once we solve $Sg(\cdot,\bs{z}_\epsilon)=-\phi(\cdot,\bs{z}_\epsilon)$ for the surface density $g(\cdot,\bs{z}_\epsilon)$, the solution in $B$ reads $u^s(\cdot,\bs{z}_\epsilon)=\mathscr{S}g(\cdot,\bs{z}_\epsilon)=-\mathscr{S}S^{-1}\phi(\cdot,\bs{z}_\epsilon)$ with
\begin{align}
& \mathscr{S}:H^{-1/2}(\partial D)\to H^1(B), \\
& \mathscr{S}g(\bs{x}) = \int_{\partial D}\phi(\bs{x},\bs{y})g(\bs{y})ds(\bs{y}), \quad \bs{x}\in B. \nonumber
\end{align}
Using the continuity of the operators $\mathscr{S}$ (see \cref{lem:appendix-ext}) and $S^{-1}$, and asymptotics for large arguments of Hankel functions (see \cref{lem:appendix-large-x}), we obtain, for small enough $\epsilon$,
\begin{align}
\Vert u^s(\cdot,\bs{z}_\epsilon)\Vert_{H^1(B)} \leq \Vert\mathscr{S}\Vert\,\Vert S^{-1}\Vert\,\Vert\phi(\cdot,\bs{z}_\epsilon)\Vert_{H^{1/2}(\partial D)} \leq C\epsilon^{q/2},
\end{align}
with a constant $C$ independent of $\epsilon$.\footnote{We will consistently denote all constants independent of $\epsilon$ as ``$C$,'' irrespective of potential variations between inequalities or equalities.}

For the second inequality, the solution reads $u^s(\cdot,\bs{z}_\epsilon)=-T_\epsilon^TS^{-1}\phi(\cdot,\bs{z}_\epsilon)$ on $\partial D_\epsilon$ with
\begin{align}
& T_\epsilon^T:H^{-1/2}(\partial D)\to H^{1/2}(\partial D_\epsilon), \\
& T_\epsilon^Tg(\bs{x}) = \int_{\partial D}\phi(\bs{x},\bs{y})g(\bs{y})ds(\bs{y}), \quad \bs{x}\in \partial D_\epsilon. \nonumber
\end{align}
Using the estimate for $T_\epsilon^T$ in \cref{lem:appendix-ext}, we arrive at
\begin{align}
\Vert u^s(\cdot,\bs{z}_\epsilon)\Vert_{H^{1/2}(\partial D_\epsilon)} \leq \Vert T_\epsilon^T\Vert\,\Vert S^{-1}\Vert\,\Vert\phi(\cdot,\bs{z}_\epsilon)\Vert_{H^{1/2}(\partial D)} \leq C\epsilon^{p/2}\epsilon^{q/2}.
\end{align}
\end{proof}

We interpret the estimates in \cref{lem:u^s} as follows: the small norm of $u^s(\cdot,\bs{z}_\epsilon)$ can be attributed to the wave field transmitted by the point source covering a distance $d=\OO(\epsilon^{-q})$ to reach $D$, with its amplitude decaying as $1/\sqrt{d}$. Subsequently, during the evaluation in the proximity of $D$ within $B$, there is no additional loss of signal amplitude. However, when we evaluate it on $\partial D_\epsilon$, some signal amplitude is lost, as it has to travel back from $D$ to $\partial D_\epsilon$ covering a distance of $\OO(\epsilon^{-p})$, resulting in an additional decaying term of $\OO(\epsilon^{p/2})$.

We continue with estimates on $v^i_\epsilon(\cdot,\bs{y}_\epsilon,\bs{z}_\epsilon)$, representing the wave field transmitted by the point source $\phi(\cdot,\bs{z}_\epsilon)$ scattered solely by $D_\epsilon(\bs{y}_\epsilon)$.

\begin{lemma}[Estimates on $v^i_\epsilon$]\label{lem:v^i} 
The solution to (\ref{eq:v^i}) reads
\begin{align}
v^i_\epsilon(\bs{x},\bs{y}_\epsilon,\bs{z}_\epsilon) = -\frac{i}{4}\sum_{n=-\infty}^{+\infty}\frac{H_n^{(1)}(k\vert\bs{y}_\epsilon-\bs{z}_\epsilon\vert)J_n(2\pi\epsilon)}{H_n^{(1)}(2\pi\epsilon)}e^{-in(\mu_\epsilon+\pi)} H_n^{(1)}(k\vert\bs{x}-\bs{y}_\epsilon\vert)e^{in\theta}
\end{align}
in the polar coordinate system $(\vert\bs{x}-\bs{y}_\epsilon\vert,\theta)$ centered at $\bs{y}_\epsilon$; the angle $\mu_\epsilon$ is the angle of $\bs{y}_\epsilon$ in the polar coordinate system centered at $\bs{z}_\epsilon$. Moreover, there exists a constant $C>0$, independent of $\epsilon$, such that for small enough $\epsilon$,
\begin{align}
\Vert v^i_\epsilon(\cdot,\bs{y}_\epsilon,\bs{z}_\epsilon)\Vert_{H^1(B)} \leq C\vert\log\epsilon\vert^{-1}\epsilon^{p/2}\epsilon^{q/2}
\end{align}
and 
\begin{align}
\Vert v^i_\epsilon(\cdot,\bs{y}_\epsilon,\bs{z}_\epsilon)\Vert_{H^{1/2}(\partial D)} \leq C\vert\log\epsilon\vert^{-1}\epsilon^{p/2}\epsilon^{q/2}.
\end{align}
\end{lemma}

\begin{proof}
We look for a solution of the form
\begin{align}
v^i_\epsilon(\bs{x},\bs{y}_\epsilon,\bs{z}_\epsilon) = \frac{1}{\sqrt{2\pi}}\sum_{n=-\infty}^{+\infty}\frac{c_n(\epsilon)}{H_n^{(1)}(2\pi\epsilon)}H_n^{(1)}(k\vert\bs{x}-\bs{y}_\epsilon\vert)e^{in\theta},
\end{align}
with Fourier coefficients
\begin{align}
c_n(\epsilon) = \frac{1}{2\pi}\int_0^{2\pi} v^i_\epsilon(\bs{y}_\epsilon+\lambda\epsilon e^{i\theta},\bs{y}_\epsilon,\bs{z}_\epsilon)e^{-in\theta}d\theta.
\end{align}
We must rewrite the right-hand side of the boundary condition in equation \cref{eq:v^i}. Utilizing Graf's addition theorem (see, e.g., \cite[eq.~(9.1.79)]{abramowitz1964} or \cite[eq.~(5.12.11)]{lebedev1972}), we obtain
\begin{align}
H_0^{(1)}(k\vert\bs{x}-\bs{z}_\epsilon\vert) = \sum_{n=-\infty}^{+\infty}H_n^{(1)}(k\vert\bs{y}_\epsilon-\bs{z}_\epsilon\vert)J_n(k\vert\bs{x}-\bs{y}_\epsilon\vert)e^{-in(\mu_\epsilon+\pi)}e^{in\theta}.
\end{align}
Utilizing the boundary condition yields the desired result.

Although a closed-form formula exists for $v^i_\epsilon$, estimating it within $B$ and on $\partial D$ using this formula is quite challenging due to the lack of uniform estimates for the terms of the series. Instead, we write $v^i_\epsilon(\cdot,\bs{y}_\epsilon,\bs{z}_\epsilon)=S_\epsilon g_\epsilon(\cdot,\bs{z}_\epsilon)$ on $\partial D_\epsilon$ with the single layer operator
\begin{align}
& S_\epsilon:H^{-1/2}(\partial D_\epsilon)\to H^{1/2}(\partial D_\epsilon), \\
& S_\epsilon g(\bs{x}) = \int_{\partial D_\epsilon}\phi(\bs{x},\bs{y})g(\bs{y})ds(\bs{y}), \quad \bs{x}\in \partial D_\epsilon. \nonumber
\end{align}
The solution in $B$ then reads $v^i_\epsilon(\cdot,\bs{y}_\epsilon,\bs{z}_\epsilon)=\mathscr{S}_\epsilon g_\epsilon(\cdot,\bs{z}_\epsilon)=-\mathscr{S}_\epsilon S_\epsilon^{-1}\phi(\cdot,\bs{z}_\epsilon)$ with
\begin{align}
& \mathscr{S}_\epsilon:H^{-1/2}(\partial D_\epsilon)\to H^1(B), \\
& \mathscr{S}_\epsilon g(\bs{x}) = \int_{\partial D_\epsilon}\phi(\bs{x},\bs{y})g(\bs{y})ds(\bs{y}), \quad \bs{x}\in B. \nonumber
\end{align}
Using the estimate for $\mathscr{S}_\epsilon$ (see \cref{lem:appendix-ext}) and applying \cref{thm:appendix-int} to $f(\cdot,\bs{z}_\epsilon)=\phi(\cdot,\bs{z}_\epsilon)$ with $r=q/2$, we obtain, for small enough $\epsilon$,
\begin{align}
\Vert v^i_\epsilon(\cdot,\bs{y}_\epsilon,\bs{z}_\epsilon)\Vert_{H^1(B)} \leq \Vert\mathscr{S}_\epsilon\Vert\,\Vert S_\epsilon^{-1}\phi(\cdot,\bs{z}_\epsilon)\Vert_{H^{-1/2}(\partial D_\epsilon)} \leq C\vert\log\epsilon\vert^{-1}\epsilon^{p/2}\epsilon^{q/2},
\end{align}
with a constant $C$ independent of $\epsilon$.

For the second inequality, the solution reads $v^i_\epsilon(\cdot,\bs{y}_\epsilon,\bs{z}_\epsilon)=-T_\epsilon S_\epsilon^{-1}\phi(\cdot,\bs{z}_\epsilon)$ on $\partial D$ with
\begin{align}
& T_\epsilon:H^{-1/2}(\partial D_\epsilon)\to H^{1/2}(\partial D), \\
& T_\epsilon g(\bs{x}) = \int_{\partial D_\epsilon}\phi(\bs{x},\bs{y})g(\bs{y})ds(\bs{y}), \quad \bs{x}\in \partial D. \nonumber
\end{align}
Using the estimate for $T_\epsilon$ in \cref{lem:appendix-ext}, we arrive at
\begin{align}
\Vert v^i_\epsilon(\cdot,\bs{y}_\epsilon,\bs{z}_\epsilon)\Vert_{H^{1/2}(\partial D)} \leq \Vert T_\epsilon\Vert\,\Vert S_\epsilon^{-1}\phi(\cdot,\bs{z}_\epsilon)\Vert_{H^{-1/2}(\partial D_\epsilon)} \leq C\vert\log\epsilon\vert^{-1}\epsilon^{p/2}\epsilon^{q/2}.
\end{align}
\end{proof}

From \cref{lem:v^i}, we understand that the small norm of $v^i_\epsilon(\cdot,\bs{y}_\epsilon,\bs{z}_\epsilon)$ can be explained by the decay of the wave field transmitted by the point source over a distance $d=\OO(\epsilon^{-q})$ to reach $D_\epsilon$---the wave field experiences a decay in amplitude proportional to $1/\sqrt{d}$. Subsequently, during the evaluation in the vicinity of $D$ within $B$ and on $\partial D$, some signal amplitude is lost as the wave travels from $D_\epsilon$ to $D$, covering a distance of $\OO(\epsilon^{-p})$, introducing an additional decaying term of $\OO(\epsilon^{p/2})$. Lastly, the presence of the logarithmic term is a consequence of $D_\epsilon$ having a radius of $\OO(\epsilon)$.

The final step preceding the proof of our main theorem involves deriving estimates for $v^s_\epsilon(\cdot,\bs{y}_\epsilon,\bs{z}_\epsilon)$, representing the scattering of $v^i_\epsilon(\cdot,\bs{y}_\epsilon,\bs{z}_\epsilon)$ by $D$.

\begin{lemma}[Estimates on $v^s_\epsilon$]\label{lem:v^s}
There exists a constant $C>0$, independent of $\epsilon$, such that for small enough $\epsilon$,
\begin{align}
\Vert v^s_\epsilon(\cdot,\bs{y}_\epsilon,\bs{z}_\epsilon)\Vert_{H^1(B)} \leq C \vert\log\epsilon\vert^{-1}\epsilon^{p/2}\epsilon^{q/2}
\end{align}
and
\begin{align}
\Vert v^s_\epsilon(\cdot,\bs{y}_\epsilon,\bs{z}_\epsilon)\Vert_{H^{1/2}(\partial D_\epsilon)} \leq C \vert\log\epsilon\vert^{-1}\epsilon^p\epsilon^{q/2}.
\end{align}
\end{lemma}

\begin{proof}
The first inequality follows from
\begin{align}
\Vert v^s_\epsilon(\cdot,\bs{y}_\epsilon,\bs{z}_\epsilon)\Vert_{H^1(B)} \leq C\Vert v^i_\epsilon(\cdot,\bs{y}_\epsilon,\bs{z}_\epsilon)\Vert_{H^{1/2}(\partial D)},
\end{align}
a standard result of functional analysis (see, e.g., \cite[Chap.~6]{evans2010}). For the second inequality, we write $v^s_\epsilon(\cdot,\bs{y}_\epsilon,\bs{z}_\epsilon)=-T_\epsilon^TS^{-1}v^i_\epsilon(\cdot,\bs{y}_\epsilon,\bs{z}_\epsilon)$ on $\partial D_\epsilon$, and then
\begin{align}
\Vert v^s_\epsilon(\cdot,\bs{y}_\epsilon,\bs{z}_\epsilon)\Vert_{H^{1/2}(\partial D_\epsilon)} \leq \Vert T_\epsilon^T\Vert\,\Vert S^{-1}\Vert\,\Vert v^i_\epsilon(\cdot,\bs{y}_\epsilon,\bs{z}_\epsilon)\Vert_{H^{1/2}(\partial D)} \leq C\vert\log\epsilon\vert^{-1}\epsilon^p\epsilon^{q/2}.
\end{align}
\end{proof}

The amplitude of $v^s_\epsilon(\cdot,\bs{y}_\epsilon,\bs{z}_\epsilon)$ in $B$ is the same as that of $v^i_\epsilon(\cdot,\bs{y}_\epsilon,\bs{z}_\epsilon)$. However, on $\partial D_\epsilon$, because of the distance $\OO(\epsilon^{-p})$, we get an additional decaying factor of $\OO(\epsilon^{p/2})$.

Let $e_\epsilon(\cdot,\bs{y}_\epsilon,\bs{z}_\epsilon)\in H^1_{\mrm{loc}}(\R^2\setminus\{D\cup D_\epsilon\})$ be the error
\begin{align}\label{eq:e}
e_\epsilon(\cdot,\bs{y}_\epsilon,\bs{z}_\epsilon) = w^s_\epsilon(\cdot,\bs{y}_\epsilon,\bs{z}_\epsilon) - u^s(\cdot,\bs{z}_\epsilon) - v^i_\epsilon(\cdot,\bs{y}_\epsilon,\bs{z}_\epsilon) - v^s_\epsilon(\cdot,\bs{y}_\epsilon,\bs{z}_\epsilon).
\end{align}
It solves the scattering problem
\begin{align}
\left\{
\begin{array}{ll}
\Delta e_\epsilon(\cdot,\bs{y}_\epsilon,\bs{z}_\epsilon) + k^2 e_\epsilon(\cdot,\bs{y}_\epsilon,\bs{z}_\epsilon) = 0 \quad \text{in $\R^2\setminus\{\overline{D\cup D_\epsilon}\}$}, \\[0.4em]
e_\epsilon(\cdot,\bs{y}_\epsilon,\bs{z}_\epsilon) = 0 \quad \text{on $\partial D$}, \\[0.4em]
e_\epsilon(\cdot,\bs{y}_\epsilon,\bs{z}_\epsilon) = f_\epsilon(\cdot,\bs{y}_\epsilon,\bs{z}_\epsilon) \quad \text{on $\partial D_\epsilon$}, \\[0.4em]
\text{$e_\epsilon(\cdot,\bs{y}_\epsilon,\bs{z}_\epsilon)$ is radiating},
\end{array}
\right.
\end{align}
with $f_\epsilon(\cdot,\bs{y}_\epsilon,\bs{z}_\epsilon) = -u^s(\cdot,\bs{z}_\epsilon) - v^s_\epsilon(\cdot,\bs{y}_\epsilon,\bs{z}_\epsilon)$. We are now ready to prove our main result.

\begin{theorem}[Estimates on $e_\epsilon$]\label{thm:e}
There exists a constant $C>0$, independent of $\epsilon$, such that for small enough $\epsilon$,
\begin{align}
\Vert e_\epsilon(\cdot,\bs{y}_\epsilon,\bs{z}_\epsilon)\Vert_{H^1(B)} \leq C \vert\log\epsilon\vert^{-1}\epsilon^{p}\epsilon^{q/2}.
\end{align}
\end{theorem}

\begin{proof}
Consider the solution $u_\epsilon(\cdot,\bs{y}_\epsilon,\bs{z}_\epsilon)\in H^1_{\mrm{loc}}(\R^2\setminus D_\epsilon)$ to
\begin{align}
\left\{
\begin{array}{ll}
\Delta u_\epsilon(\cdot,\bs{y}_\epsilon,\bs{z}_\epsilon) + k^2 u_\epsilon(\cdot,\bs{y}_\epsilon,\bs{z}_\epsilon) = 0 \quad \text{in $\R^2\setminus\overline{D_\epsilon}$}, \\[0.4em]
u_\epsilon(\cdot,\bs{y}_\epsilon,\bs{z}_\epsilon) = f_\epsilon(\cdot,\bs{y}_\epsilon,\bs{z}_\epsilon) \quad \text{on $\partial D_\epsilon$}, \\[0.4em]
\text{$u_\epsilon(\cdot,\bs{y}_\epsilon,\bs{z}_\epsilon)$ is radiating}.
\end{array}
\right.
\end{align}
Then, $w_\epsilon = e_\epsilon - u_\epsilon\in H^1_{\mrm{loc}}(\R^2\setminus\{D\cup D_\epsilon\})$ satisfies
\begin{align}
\left\{
\begin{array}{ll}
\Delta w_\epsilon(\cdot,\bs{y}_\epsilon,\bs{z}_\epsilon) + k^2 w_\epsilon(\cdot,\bs{y}_\epsilon,\bs{z}_\epsilon) = 0 \quad \text{in $\R^2\setminus\{\overline{D\cup D_\epsilon}\}$}, \\[0.4em]
w_\epsilon(\cdot,\bs{y}_\epsilon,\bs{z}_\epsilon) = -u_\epsilon(\cdot,\bs{y}_\epsilon,\bs{z}_\epsilon) \quad \text{on $\partial D$}, \\[0.4em]
w_\epsilon(\cdot,\bs{y}_\epsilon,\bs{z}_\epsilon) = 0 \quad \text{on $\partial D_\epsilon$}, \\[0.4em]
\text{$w_\epsilon(\cdot,\bs{y}_\epsilon,\bs{z}_\epsilon)$ is radiating}.
\end{array}
\right.
\end{align}
Utilizing \cref{thm:appendix-ext}, there exists a constant $C>0$, independent of $\epsilon$, such that for small enough $\epsilon$,
\begin{align}
\Vert w_\epsilon(\cdot,\bs{y}_\epsilon,\bs{z}_\epsilon)\Vert_{H^1(B)} \leq C \Vert u_\epsilon(\cdot,\bs{y}_\epsilon,\bs{z}_\epsilon)\Vert_{H^{1/2}(\partial D)}.
\end{align}
This yields
\begin{align}
\Vert e_\epsilon(\cdot,\bs{y}_\epsilon,\bs{z}_\epsilon)\Vert_{H^1(B)} \leq \Vert u_\epsilon(\cdot,\bs{y}_\epsilon,\bs{z}_\epsilon)\Vert_{H^1(B)} + C \Vert u_\epsilon(\cdot,\bs{y}_\epsilon,\bs{z}_\epsilon)\Vert_{H^{1/2}(\partial D)}.
\end{align}
We write $u_\epsilon(\cdot,\bs{y}_\epsilon,\bs{z}_\epsilon)=-\mathscr{S}_\epsilon S_\epsilon^{-1}f_\epsilon(\cdot,\bs{y}_\epsilon,\bs{z}_\epsilon)$ in $B$ and $u_\epsilon(\cdot,\bs{y}_\epsilon,\bs{z}_\epsilon)=-T_\epsilon S_\epsilon^{-1}f_\epsilon(\cdot,\bs{y}_\epsilon,\bs{z}_\epsilon)$ on $\partial D$ with $f_\epsilon(\cdot,\bs{y}_\epsilon,\bs{z}_\epsilon) = -u^s(\cdot,\bs{z}_\epsilon) - v^s_\epsilon(\cdot,\bs{y}_\epsilon,\bs{z}_\epsilon)$. Using the estimates for $\mathscr{S}_\epsilon$ and $T_\epsilon$ (see \cref{lem:appendix-ext}) and applying \cref{thm:appendix-int} to $f_\epsilon(\cdot,\bs{y}_\epsilon,\bs{z}_\epsilon) $ with $r=p/2+q/2$, we obtain, for small enough $\epsilon$,
\begin{align}
\Vert u_\epsilon(\cdot,\bs{y}_\epsilon,\bs{z}_\epsilon)\Vert_{H^1(B)} \leq \Vert\mathscr{S}_\epsilon\Vert\,\Vert S_\epsilon^{-1}f_\epsilon(\cdot,\bs{y}_\epsilon,\bs{z}_\epsilon)\Vert_{H^{-1/2}(\partial D_\epsilon)} \leq C\vert\log\epsilon\vert^{-1}\epsilon^p\epsilon^{q/2}
\end{align}
and 
\begin{align}
\Vert u_\epsilon(\cdot,\bs{y}_\epsilon,\bs{z}_\epsilon)\Vert_{H^{1/2}(\partial D)} \leq \Vert T_\epsilon\Vert\,\Vert S_\epsilon^{-1}f_\epsilon(\cdot,\bs{y}_\epsilon,\bs{z}_\epsilon)\Vert_{H^{-1/2}(\partial D_\epsilon)} \leq C\vert\log\epsilon\vert^{-1}\epsilon^p\epsilon^{q/2},
\end{align}
with a constant $C$ independent of $\epsilon$.
\end{proof}

We summarize our asymptotic results in \cref{tab:asymptotics}. We note, in particular, that the error, as expressed in \cref{eq:e}, is $\OO(\vert\log\epsilon\vert^{-1}\epsilon^p\epsilon^{q/2})$ in the $H^1(B)$-norm, according to \cref{thm:e}. This error comes from the scattering sequence $\phi(\cdot,\bs{z}_\epsilon) \rightarrow D \rightarrow u^s \rightarrow D_\epsilon \rightarrow e_\epsilon\rightarrow e_\epsilon|_B$, which generates a field smaller than the sequence $\phi(\cdot,\bs{z}_\epsilon) \rightarrow D_\epsilon \rightarrow v^i_\epsilon \rightarrow D \rightarrow v^s_\epsilon \rightarrow v^s_\epsilon|_B$. We also illustrate our asymptotic model in \cref{fig:asymptotics}.

\begin{table}
\caption{\textit{The amplitude of the scattered fields $u^s$, $v^i_\epsilon$, $v^s_\epsilon$ and $e_\epsilon$, described by the equations \cref{eq:u^s}--\cref{eq:v^s} and \cref{eq:e}, depends on $\epsilon$. The estimates are proved in \cref{lem:u^s}, \cref{lem:v^i}, \cref{lem:v^s}, and \cref{thm:e}.}}
\centering
\ra{1.3}
\begin{tabular}{cc}
\toprule
Scattering sequence & $H^1(B)$-norm \\
\midrule
$\phi(\cdot,\bs{z}_\epsilon) \underset{\text{incident}}{\longrightarrow} D \underset{\text{scattering}}{\longrightarrow} u^s \underset{\text{evaluation}}{\longrightarrow} u^s|_B$ & $\OO(\epsilon^{q/2})$ \\
$\phi(\cdot,\bs{z}_\epsilon) \underset{\text{incident}}{\longrightarrow} D_\epsilon \underset{\text{scattering}}{\longrightarrow} v^i_\epsilon \underset{\text{evaluation}}{\longrightarrow} v^i_\epsilon|_B$ & $\OO(\vert\log\epsilon\vert^{-1}\epsilon^{p/2}\epsilon^{q/2})$ \\
$\phi(\cdot,\bs{z}_\epsilon) \underset{\text{incident}}{\longrightarrow} D_\epsilon \underset{\text{scattering}}{\longrightarrow} v^i_\epsilon \underset{\text{incident}}{\longrightarrow} D \underset{\text{scattering}}{\longrightarrow} v^s_\epsilon \underset{\text{evaluation}}{\longrightarrow} v^s_\epsilon|_B$ & $\OO(\vert\log\epsilon\vert^{-1}\epsilon^{p/2}\epsilon^{q/2})$ \\
$\phi(\cdot,\bs{z}_\epsilon) \underset{\text{incident}}{\longrightarrow} D \underset{\text{scattering}}{\longrightarrow} u^s \underset{\text{incident}}{\longrightarrow} D_\epsilon \underset{\text{scattering}}{\longrightarrow} e_\epsilon \underset{\text{evaluation}}{\longrightarrow} e_\epsilon|_B$ & $\OO(\vert\log\epsilon\vert^{-1}\epsilon^{p}\epsilon^{q/2})$ \\
\bottomrule
\end{tabular}
\label{tab:asymptotics}
\end{table}

\begin{figure}
\centering
\def\scl{0.2}
\begin{tabular}{cc}
\includegraphics[scale=\scl]{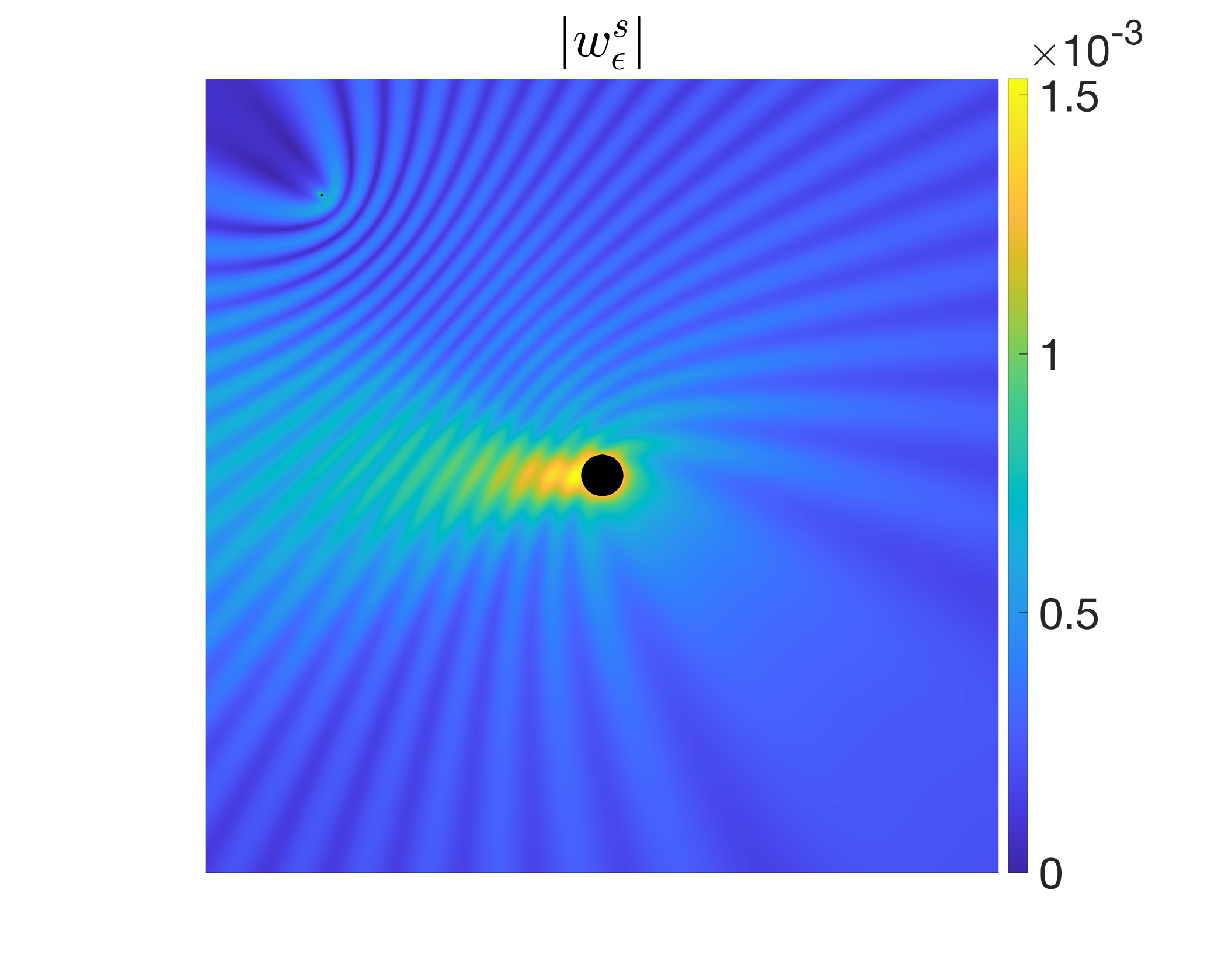} & \includegraphics[scale=\scl]{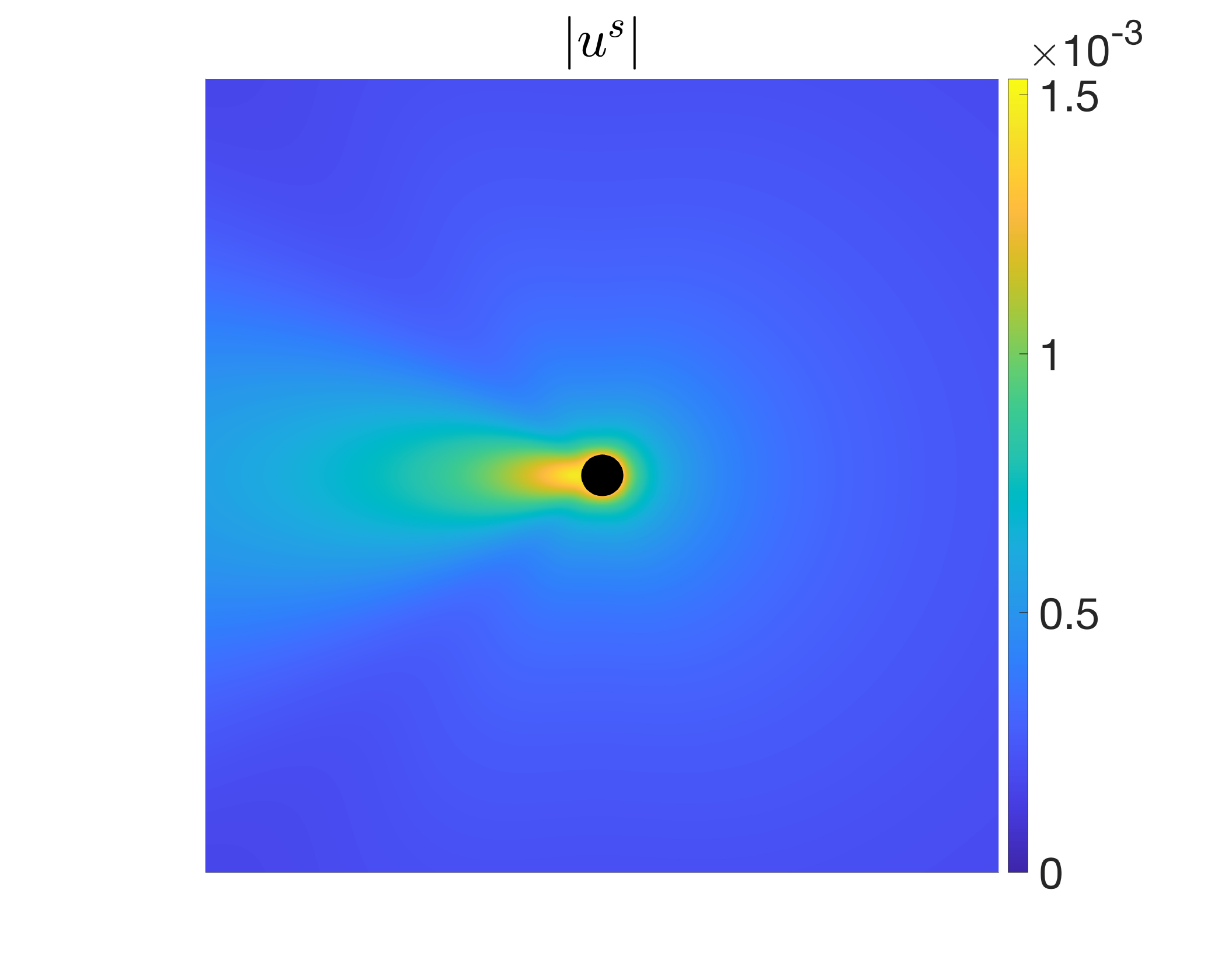} \\
\includegraphics[scale=\scl]{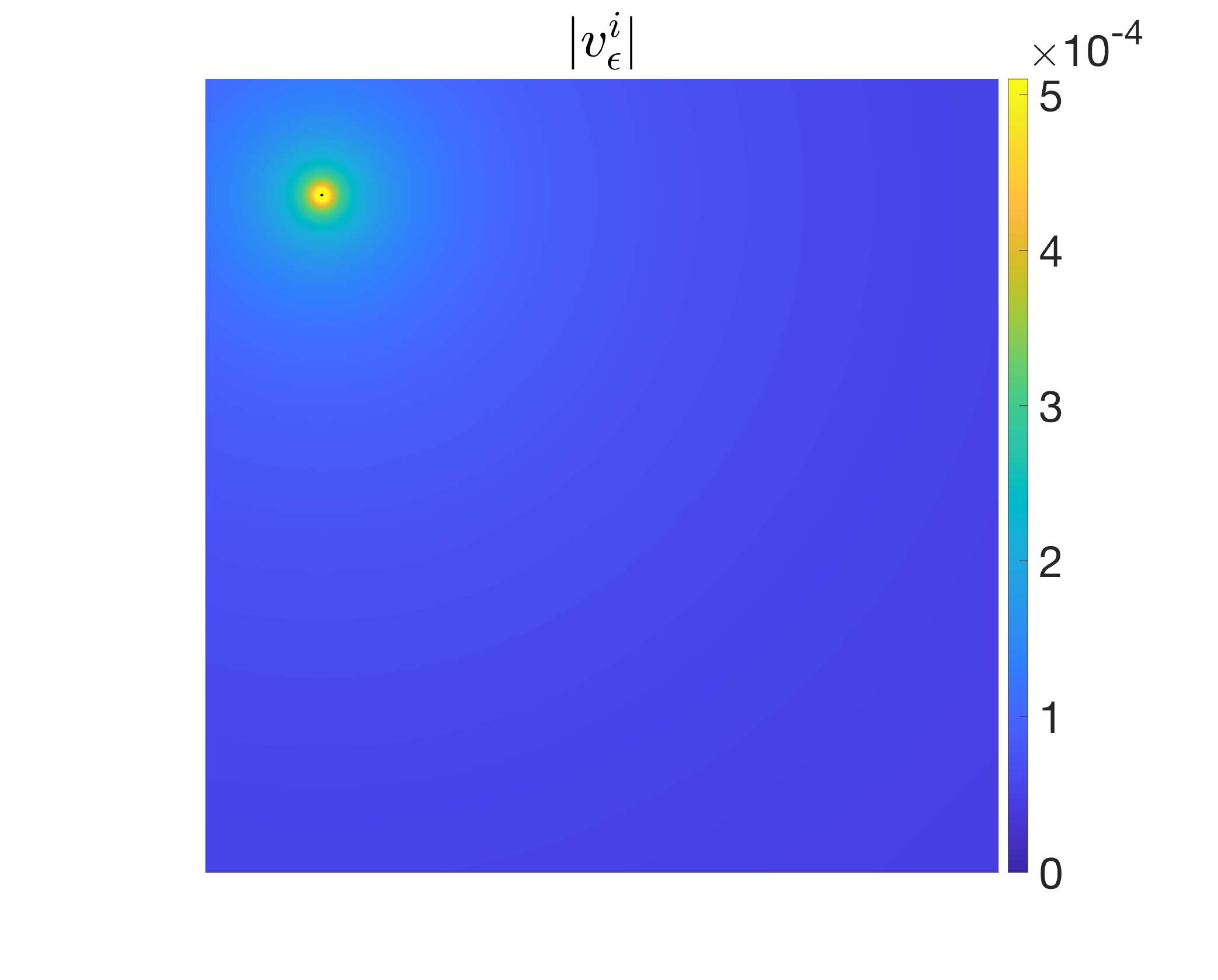} & \includegraphics[scale=\scl]{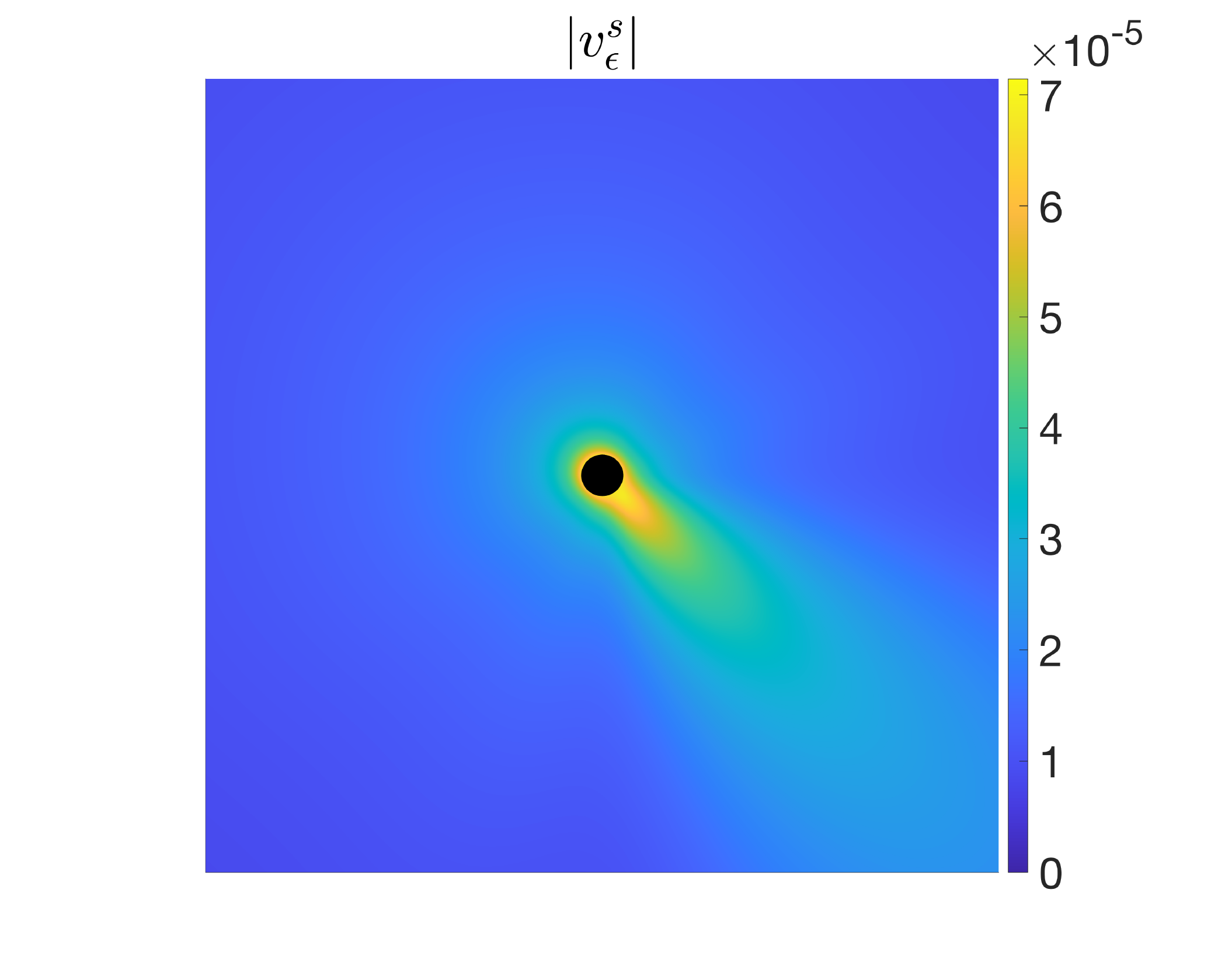}
\end{tabular}
\caption{The scattered field $w^s_\epsilon$ (top left) can be approximated with an error $\OO(\vert\log\epsilon\vert^{-1}\epsilon^p\epsilon^{q/2})$ by the sum of three terms: $u^s$ (top right), $v^i_\epsilon$ (bottom left), and $v^s_\epsilon$ (bottom right). The relevant signal for the LSM is stored in the total field $v_\epsilon=v^i_\epsilon+v^s_\epsilon$, which satisfies a modified Helmholtz--Kirchhoff identity (\cref{sec:mod-HK-identity}).}
\label{fig:asymptotics}
\end{figure}

Moving forward, we will utilize approximations for $v^i_\epsilon(\cdot,\bs{y}_\epsilon,\bs{z}_\epsilon)$ and $v^s_\epsilon(\cdot,\bs{y}_\epsilon,\bs{z}_\epsilon)$, based on the asymptotics of \cite{cassier2013}. For instance, $v^i_\epsilon(\cdot,\bs{y}_\epsilon,\bs{z}_\epsilon)$ can be approximated by the field 
\begin{align}\label{eq:v^it}
\tilde{v}^i_\epsilon(\cdot,\bs{y}_\epsilon,\bs{z}_\epsilon) = \mu_\epsilon(\bs{y}_\epsilon,\bs{z}_\epsilon)\phi(\cdot,\bs{y}_\epsilon)
\end{align}
transmitted by a point source located at $\bs{y}_\epsilon$ with amplitude
\begin{align}\label{eq:mu}
\mu_\epsilon(\bs{y}_\epsilon,\bs{z}_\epsilon)=-H_0^{(1)}(k\vert\bs{y}_\epsilon-\bs{z}_\epsilon\vert)/H_0^{(1)}(2\pi\epsilon).
\end{align}
(This is why we use the superscript ``$i$'' in $v^i_\epsilon$---it is an incident field for $D$.)
Similarly, let $\tilde{v}^s_\epsilon(\cdot,\bs{y}_\epsilon,\bs{z}_\epsilon)$ be the solution to $\cref{eq:v^s}$ with $v^i_\epsilon(\cdot,\bs{y}_\epsilon,\bs{z}_\epsilon)$ replaced by $\tilde{v}^i_\epsilon(\cdot,\bs{y}_\epsilon,\bs{z}_\epsilon)$. Then, 
\begin{align}\label{eq:v^st}
\tilde{v}^s_\epsilon(\cdot,\bs{y}_\epsilon,\bs{z}_\epsilon) = \mu_\epsilon(\bs{y}_\epsilon,\bs{z}_\epsilon)u^s(\cdot,\bs{y}_\epsilon),
\end{align} 
where $u^s(\cdot,\bs{y}_\epsilon)$ is the solution to the scattering problem \cref{eq:u^s} for the incident wave $\phi(\cdot,\bs{y}_\epsilon)$.

Let $\tilde{e}_\epsilon(\cdot,\bs{y}_\epsilon,\bs{z}_\epsilon)\in H^1_{\mrm{loc}}(\R^2\setminus\{D\cup\{\bs{y}_\epsilon\}\})$ be the error
\begin{align}\label{eq:et}
\tilde{e}_\epsilon(\cdot,\bs{y}_\epsilon,\bs{z}_\epsilon) = w^s_\epsilon(\cdot,\bs{y}_\epsilon,\bs{z}_\epsilon) - u^s(\cdot,\bs{z}_\epsilon) - \tilde{v}^i_\epsilon(\cdot,\bs{y}_\epsilon,\bs{z}_\epsilon) - \tilde{v}^s_\epsilon(\cdot,\bs{y}_\epsilon,\bs{z}_\epsilon).
\end{align}
It solves the scattering problem
\begin{align}
\left\{
\begin{array}{ll}
\Delta \tilde{e}_\epsilon(\cdot,\bs{y}_\epsilon,\bs{z}_\epsilon) + k^2 \tilde{e}_\epsilon(\cdot,\bs{y}_\epsilon,\bs{z}_\epsilon) = 0 \quad \text{in $\R^2\setminus\{\overline{D}\cup \{\bs{y}_\epsilon\}\}$}, \\[0.4em]
\tilde{e}_\epsilon(\cdot,\bs{y}_\epsilon,\bs{z}_\epsilon) = 0 \quad \text{on $\partial D$}, \\[0.4em]
\tilde{e}_\epsilon(\cdot,\bs{y}_\epsilon,\bs{z}_\epsilon) = \tilde{f}_\epsilon(\cdot,\bs{y}_\epsilon,\bs{z}_\epsilon) \quad \text{on $\partial D_\epsilon$}, \\[0.4em]
\text{$e_\epsilon(\cdot,\bs{y}_\epsilon,\bs{z}_\epsilon)$ is radiating},
\end{array}
\right.
\end{align}
with $\tilde{f}_\epsilon(\cdot,\bs{y}_\epsilon,\bs{z}_\epsilon) = \phi(\cdot,\bs{z}_\epsilon) - u^s(\cdot,\bs{z}_\epsilon) - \tilde{v}^i_\epsilon(\cdot,\bs{y}_\epsilon,\bs{z}_\epsilon) - \tilde{v}^s_\epsilon(\cdot,\bs{y}_\epsilon,\bs{z}_\epsilon)$. We obtain the following theorem from \cref{thm:e} and the formulas \cref{eq:v^it} and \cref{eq:v^st} for $\tilde{v}^i_\epsilon$ and $\tilde{v}^s_\epsilon$.

\begin{theorem}[Estimates on $\tilde{e}_\epsilon$]\label{thm:et}
There exists a constant $C>0$, independent of $\epsilon$, such that for small enough $\epsilon$,
\begin{align}
\Vert \tilde{e}_\epsilon(\cdot,\bs{y}_\epsilon,\bs{z}_\epsilon)\Vert_{H^1(B)} \leq C (\vert\log\epsilon\vert^{-1}\epsilon^{p}\epsilon^{q/2} + \epsilon^{p/2}\epsilon^{q/2}\epsilon^2).
\end{align}
\end{theorem}

\begin{proof}
We combine \cref{eq:e} with \cref{eq:et} to write
\begin{align}
\tilde{e}_\epsilon(\cdot,\bs{y}_\epsilon,\bs{z}_\epsilon) = e_\epsilon(\cdot,\bs{y}_\epsilon,\bs{z}_\epsilon) + \tilde{v}^i_\epsilon(\cdot,\bs{y}_\epsilon,\bs{z}_\epsilon) - v^i_\epsilon(\cdot,\bs{y}_\epsilon,\bs{z}_\epsilon) + \tilde{v}^s_\epsilon(\cdot,\bs{y}_\epsilon,\bs{z}_\epsilon) - v^s_\epsilon(\cdot,\bs{y}_\epsilon,\bs{z}_\epsilon).
\end{align}
Let us define $e_\epsilon^i = \tilde{v}^i_\epsilon - v^i_\epsilon$ and $e_\epsilon^s = \tilde{v}^s_\epsilon - v^s_\epsilon$. These solve the scattering problems
\begin{align}\label{eq:dv^i}
\left\{
\begin{array}{ll}
\Delta e_\epsilon^i(\cdot,\bs{y}_\epsilon,\bs{z}_\epsilon) + k^2 e_\epsilon^i(\cdot,\bs{y}_\epsilon,\bs{z}_\epsilon) = 0 \quad \text{in $\R^2\setminus\overline{D_\epsilon}$}, \\[0.4em]
e_\epsilon^i(\cdot,\bs{y}_\epsilon,\bs{z}_\epsilon) = -(\phi(\bs{y}_\epsilon,\bs{z}_\epsilon)-\phi(\cdot,\bs{z}_\epsilon)) \quad \text{on $\partial D_\epsilon$}, \\[0.4em]
\text{$e_\epsilon^i(\cdot,\bs{y}_\epsilon,\bs{z}_\epsilon)$ is radiating},
\end{array}
\right.
\end{align}
and
\begin{align}\label{eq:dv^s}
\left\{
\begin{array}{ll}
\Delta e_\epsilon^s(\cdot,\bs{y}_\epsilon,\bs{z}_\epsilon) + k^2 e_\epsilon^s(\cdot,\bs{y}_\epsilon,\bs{z}_\epsilon) = 0 \quad \text{in $\R^2\setminus\overline{D}$}, \\[0.4em]
e_\epsilon^s(\cdot,\bs{y}_\epsilon,\bs{z}_\epsilon) = -e_\epsilon^i(\cdot,\bs{y}_\epsilon,\bs{z}_\epsilon) \quad \text{on $\partial D$}, \\[0.4em]
\text{$e_\epsilon^s(\cdot,\bs{y}_\epsilon,\bs{z}_\epsilon)$ is radiating}.
\end{array}
\right.
\end{align}
We note that, using the notations of \cref{lem:v^i},
\begin{align}
e_\epsilon^i(\bs{x},\bs{y}_\epsilon,\bs{z}_\epsilon) = & \; -\frac{i}{4}\frac{H_0^{(1)}(k\vert\bs{y}_\epsilon-\bs{z}_\epsilon\vert)(1-J_0(2\pi\epsilon))}{H_0^{(1)}(2\pi\epsilon)}H_0^{(1)}(k\vert\bs{x}-\bs{y}_\epsilon\vert) \\
& \; -\frac{i}{4}\sum_{\vert n\vert \neq0}\frac{H_n^{(1)}(k\vert\bs{y}_\epsilon-\bs{z}_\epsilon\vert)J_n(2\pi\epsilon)}{H_n^{(1)}(2\pi\epsilon)}e^{-in(\mu_\epsilon+\pi)} H_n^{(1)}(k\vert\bs{x}-\bs{y}_\epsilon\vert)e^{in\theta}. \nonumber
\end{align}
We write $e_\epsilon^i(\cdot,\bs{y}_\epsilon,\bs{z}_\epsilon)=-\mathscr{S}_\epsilon S_\epsilon^{-1}f(\cdot,\bs{y}_\epsilon,\bs{z}_\epsilon)$ in $B$ and $e_\epsilon^i(\cdot,\bs{y}_\epsilon,\bs{z}_\epsilon)=-T_\epsilon S_\epsilon^{-1}f(\cdot,\bs{y}_\epsilon,\bs{z}_\epsilon)$ on $\partial D$ with $f(\cdot,\bs{y}_\epsilon,\bs{z}_\epsilon) = -(\phi(\bs{y}_\epsilon,\bs{z}_\epsilon)-\phi(\cdot,\bs{z}_\epsilon))$. Applying \cref{lem:appendix-ext} to $f(\cdot,\bs{y}_\epsilon,\bs{z}_\epsilon)$ with $r=q/2$, we obtain, for small enough $\epsilon$,
\begin{align}
\Vert e_\epsilon^i(\cdot,\bs{y}_\epsilon,\bs{z}_\epsilon)\Vert_{H^{1/2}(\partial D)} = \Vert T_\epsilon S_\epsilon^{-1}f_\epsilon(\cdot,\bs{y}_\epsilon,\bs{z}_\epsilon)\Vert_{H^{1/2}(\partial D)} \leq C\epsilon^{p/2}\epsilon^{q/2}\epsilon^2
\end{align}
and
\begin{align}
\Vert e_\epsilon^i(\cdot,\bs{y}_\epsilon,\bs{z}_\epsilon)\Vert_{H^1(B)} = \Vert\mathscr{S}_\epsilon S_\epsilon^{-1}f_\epsilon(\cdot,\bs{y}_\epsilon,\bs{z}_\epsilon)\Vert_{H^1(B)} \leq C\epsilon^{p/2}\epsilon^{q/2}\epsilon^2,
\end{align}
with a constant $C$ independent of $\epsilon$. We conclude by noting that 
\begin{align}
\Vert e_\epsilon^s(\cdot,\bs{y}_\epsilon,\bs{z}_\epsilon)\Vert_{H^1(B)}\leq C \Vert e_\epsilon^i(\cdot,\bs{y}_\epsilon,\bs{z}_\epsilon)\Vert_{H^{1/2}(\partial D)}
\end{align}
and using the triangle inequality.
\end{proof}

In the proof of \cref{thm:et}, we showed that approximating $v^i_\epsilon$ by $\tilde{v}^i_\epsilon$ yields an error $\OO(\epsilon^2)$ when $p=q=0$. This is sharper than the error $\OO(\vert\log\epsilon\vert^{-1}\epsilon)$ proved in \cite[Thm.~1]{cassier2013}.

\section{A modified Helmholtz--Kirchhoff identity}\label{sec:mod-HK-identity}

In this section, we will establish that the total field $\tilde{v}_\epsilon = \tilde{v}^i_\epsilon + \tilde{v}^s_\epsilon$ satisfies a \textit{modified} Helmholtz--Kirchhoff identity. Consequently, the information is encapsulated in $\tilde{v}_\epsilon$. However, the practical challenge lies in having access only to measurements of $w^s_\epsilon = u^s + \tilde{v}_\epsilon + \tilde{e}_\epsilon$. We first outline the process of extracting $\tilde{v}_\epsilon$.

Let $\Sigma_\epsilon$ be the circle of radius $\lambda\epsilon^{-p}$ centered at the origin and let $\langle\cdot\rangle$ denote the average with respect to $\bs{y}_\epsilon\in\Sigma_\epsilon$, e.g.,
\begin{align}
\langle w^s_\epsilon\rangle(\bs{x},\bs{z}_\epsilon) = \frac{1}{\vert\Sigma_\epsilon\vert}\int_{\Sigma_\epsilon}w^s_\epsilon(\bs{x},\bs{y}_\epsilon,\bs{z}_\epsilon)ds(\bs{y}_\epsilon).
\end{align}
Since $u^s(\cdot,\bs{z}_\epsilon)$ does not depend on $\bs{y}_\epsilon\in\Sigma_\epsilon$, we note that, for any $\bs{x}\in B$,
\begin{align}\label{eq:average}
w^s_\epsilon(\bs{x},\bs{y}_\epsilon,\bs{z}_\epsilon) - \langle w^s_\epsilon\rangle(\bs{x},\bs{z}_\epsilon) = \tilde{v}_\epsilon(\bs{x},\bs{y}_\epsilon,\bs{z}_\epsilon) - \langle \tilde{v}_\epsilon\rangle(\bs{x},\bs{z}_\epsilon) + \tilde{e}_\epsilon(\bs{x},\bs{y}_\epsilon,\bs{z}_\epsilon) - \langle\tilde{e}_\epsilon\rangle(\bs{x},\bs{z}_\epsilon).
\end{align}
We are going to show that the $\bs{y}_\epsilon$-average of $\tilde{v}_\epsilon$ is small in comparison to $\tilde{v}_\epsilon$ (\cref{thm:average}). Therefore, removing the $\bs{y}_\epsilon$-average of $w^s_\epsilon$ from $w^s_\epsilon$ is a simple way of accessing $\tilde{v}_\epsilon$, as the other three terms in the right-hand side of \cref{eq:average} are negligible.

\begin{theorem}[Norm of $\tilde{v}_\epsilon$ and $\langle\tilde{v}_\epsilon\rangle$]\label{thm:average}
There exists a constant $C>0$, independent of $\epsilon$, such that for small enough $\epsilon$,
\begin{align}
\Vert\tilde{v}_\epsilon(\cdot,\bs{y}_\epsilon,\bs{z}_\epsilon)\Vert_{H^1(B)} \leq C \vert\log\epsilon\vert^{-1}\epsilon^{p/2}\epsilon^{q/2}
\end{align}
and
\begin{align}
\Vert\langle\tilde{v}_\epsilon\rangle(\cdot,\bs{z}_\epsilon)\Vert_{H^1(B)} \leq C \vert\log\epsilon\vert^{-1}\epsilon^{p}\epsilon^{q/2}.
\end{align}
\end{theorem}

\begin{proof}
The first estimate can be obtained via those on $v^i_\epsilon(\cdot,\bs{y}_\epsilon,\bs{z}_\epsilon)$ (\cref{lem:v^i}) and $v^s_\epsilon(\cdot,\bs{y}_\epsilon,\bs{z}_\epsilon)$ (\cref{lem:v^s}), or via \cref{lem:appendix-average}. For the second estimate, we first note that the amplitude of $\langle\tilde{v}^i_\epsilon\rangle(\cdot,\bs{z}_\epsilon)$ can be estimated using \cref{lem:appendix-average}. To conclude, we observe that $\langle\tilde{v}^s_\epsilon\rangle(\cdot,\bs{z}_\epsilon)$ is a scattered field for the incident wave $\langle\tilde{v}^i_\epsilon\rangle(\cdot,\bs{z}_\epsilon)$.
\end{proof}

For any $\bs{x}'\in\R^2\setminus\overline{D}$, we recall that $u^s(\cdot,\bs{x}')$ denotes the solution to \cref{eq:u^s} for the incident wave $\phi(\cdot,\bs{x}')$. In our previous work, we utilize the fact that $u^s(\cdot,\bs{x}')$ and the associated total field $u(\bs{x},\bs{x}')=\phi(\cdot,\bs{x}')+u^s(\cdot,\bs{x}')$ verify the following Helmholtz--Kirchhoff identity
\begin{align}\label{eq:HK-identity}
u^s(\bs{x},\bs{x}') - \overline{u^s(\bs{x},\bs{x}')} = 2ik\int_{\Sigma}\overline{u(\bs{x},\bs{z})}u(\bs{x}',\bs{z})ds(\bs{z}) - \left[\phi(\bs{x},\bs{x}')-\overline{\phi(\bs{x},\bs{x}')}\right].
\end{align}
This motivated the setup of \cref{fig:passive-LSM} (left) and the introduction of the cross-correlation matrix \cref{eq:cross-cor-mat}. Indeed, in that setup, we assumed that the $J>0$ measurement points $\bs{x}_j$ are located in some bounded volume $B\subset\R^2\setminus\overline{D}$. We also assumed that there exists a surface $\Sigma$ that encloses $B$ and $D$, and that there are $L>0$ point sources $\bs{z}_\ell$ randomly distributed on $\Sigma$. These sources can transmit a unit-amplitude time-harmonic signal, one by one, so that it is possible to measure the total fields $u(\bs{x}_j,\bs{z}_\ell)$. Moreover, it is possible to compute $\phi(\bs{x}_j,\bs{x}_m)$, so we can evaluate the cross-correlation matrix \cref{eq:cross-cor-mat}. This matrix corresponds to the discretization of the right-hand side of \cref{eq:HK-identity} at points $\bs{z}_\ell$ with uniform weights and evaluated at $\bs{x}_j$ and $\bs{x}_m$, i.e.,
\begin{align}\label{eq:trap-rule-approx}
C_{jm} \approx 2ik\int_{\Sigma}\overline{u(\bs{x}_j,\bs{z})}u(\bs{x}_m,\bs{z})ds(\bs{z}) - \left[\phi(\bs{x}_j,\bs{x}_m)-\overline{\phi(\bs{x}_j,\bs{x}_m)}\right].
\end{align}
The quadrature error associated with \cref{eq:trap-rule-approx} is $\OO(1/\sqrt{L})$, in general. However, if the $\bs{z}_\ell$'s correspond to a $\beta$-perturbed trapezoidal rule, i.e., $\bs{z}_\ell = 2\pi(\ell - 1 + \beta_\ell)/L$ with $\vert\beta_\ell\vert\leq\beta<1/2$ for all $1\leq\ell\leq L$, then the error improves to $\OO(1/L^{\nu - 4\beta})$ whenever the integrand has $\nu>4\beta+1/2$ derivatives \cite[Thm.~1]{austin2017b}. (For details about computations with trigonometric interpolants, we refer the reader to \cite{montanelli2017phd, trefethen2019, montanelli2015b}.)

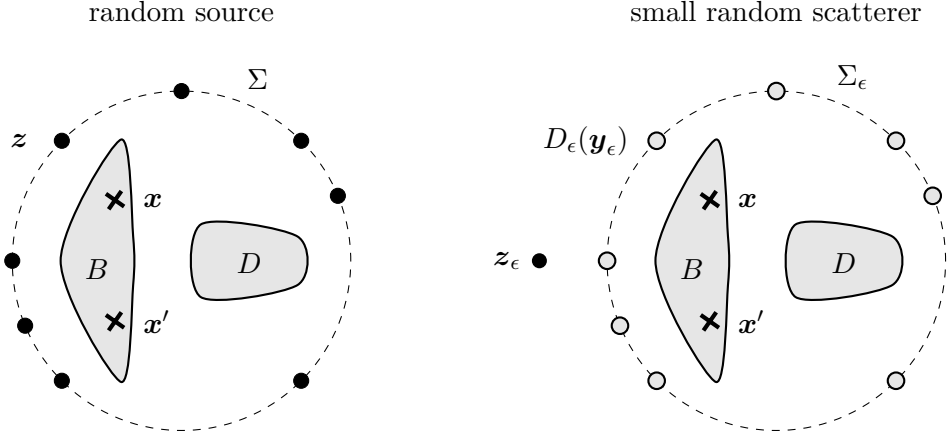
\begin{figure}
\centering
\begin{tikzpicture}[scale=0.9]

\def\R{2.5};
\def\s{0.55};
\def\t{7};
\def\x{-0.9238795325112867};
\def\y{0.38268343236508984};
\def\xx{-0.7071067811865475};
\def\yy{0.7071067811865476};
\def\z{0.7071067811865476};
\def\theta{10};

\draw[black, thick, fill=black!10, xshift=2/5*\R cm] plot[smooth cycle] coordinates {(-2*\R/\t, -1.5*\R/\t) (-2*\R/\t, 1.5*\R/\t) (2*\R/\t, \R/\t) (2*\R/\t, -\R/\t)};
\node[anchor=north] at (2/5*\R, 0.75*\R/\t) {$D$};

\draw[black, thick, fill=black!10] plot[smooth cycle] coordinates {(-2.5*\R/\t, 5*\R/\t) (-5*\R/\t, 0*\R/\t) (-2.5*\R/\t, -5*\R/\t) (-2*\R/\t, -\R/\t) (-2*\R/\t, \R/\t)};
\node[anchor=north] at (-3.5*\R/\t, 0.5*\R/\t) {$B$};

\draw[black, dashed] (0,0) circle (\R);
\node[anchor=south] at (0.45*\R, 0.95*\R) {$\Sigma$};

\node[draw, fill, circle, scale=\s] at (\R*-1, \R*0) {};
\node[draw, fill, circle, scale=\s] at (\R*0, \R*1) {};
\node[draw, fill, circle, scale=\s] at (\R*\x, -\R*\y) {};
\node[draw, fill, circle, scale=\s] at (-\R*\x, \R*\y) {};
\node[draw, fill, circle, scale=\s] at (\R*\z, \R*\z) {};
\node[draw, fill, circle, scale=\s] at (\R*\z, -\R*\z) {};
\node[draw, fill, circle, scale=\s] at (-\R*\z, \R*\z) {};
\node[anchor=east] at (-1.2*\R*\z, \R*\z) {$\bs{z}$};
\node[draw, fill, circle, scale=\s] at (-\R*\z, -\R*\z) {};

\node[cross, rotate=\theta, scale=\s] at (-2.75*\R/\t, 2.5*\R/\t) {};
\node[anchor=west] at (-2*\R/\t, 2.5*\R/\t) {$\bs{x}$};
\node[cross, rotate=\theta, scale=\s] at (-2.75*\R/\t, -2.5*\R/\t) {};
\node[anchor=west] at (-2*\R/\t, -2.5*\R/\t) {$\bs{x}'$};

\node[anchor=south] at (0, 1.35*\R) {random source};

\end{tikzpicture}
\hspace{1.5cm}
\begin{tikzpicture}[scale=0.9]

\def\R{2.5};
\def\s{0.55};
\def\t{7};
\def\x{-0.9238795325112867};
\def\y{0.38268343236508984};
\def\xx{-0.7071067811865475};
\def\yy{0.7071067811865476};
\def\z{0.7071067811865476};
\def\theta{10};

\draw[black, thick, fill=black!10, xshift=2/5*\R cm] plot[smooth cycle] coordinates {(-2*\R/\t, -1.5*\R/\t) (-2*\R/\t, 1.5*\R/\t) (2*\R/\t, \R/\t) (2*\R/\t, -\R/\t)};
\node[anchor=north] at (2/5*\R, 0.75*\R/\t) {$D$};

\draw[black, thick, fill=black!10] plot[smooth cycle] coordinates {(-2.5*\R/\t, 5*\R/\t) (-5*\R/\t, 0*\R/\t) (-2.5*\R/\t, -5*\R/\t) (-2*\R/\t, -\R/\t) (-2*\R/\t, \R/\t)};
\node[anchor=north] at (-3.5*\R/\t, 0.5*\R/\t) {$B$};

\draw[black, dashed] (0,0) circle (\R);
\node[anchor=south] at (0.45*\R, 0.95*\R) {$\Sigma_\epsilon$};

\node[fill, circle, scale=\s] at (\R*-1.4, \R*0) {};
\node[anchor=east] at (\R*-1.45, \R*0) {$\bs{z}_\epsilon$};

\node[draw, black, fill=black!10, circle, thick, scale=1.1*\s] at (\R*-1, \R*0) {};
\node[draw, black, fill=black!10, circle, thick, scale=1.1*\s] at (\R*0, \R*1) {};
\node[draw, black, fill=black!10, circle, thick, scale=1.1*\s] at (\R*\x, -\R*\y) {};
\node[draw, black, fill=black!10, circle, thick, scale=1.1*\s] at (-\R*\x, \R*\y) {};
\node[draw, black, fill=black!10, circle, thick, scale=1.1*\s] at (\R*\z, \R*\z) {};
\node[draw, black, fill=black!10, circle, thick, scale=1.1*\s] at (\R*\z, -\R*\z) {};
\node[draw, black, fill=black!10, circle, thick, scale=1.1*\s] at (-\R*\z, \R*\z) {};
\node[anchor=east] at (-1.15*\R*\z, \R*\z) {$D_\epsilon(\bs{y}_\epsilon)$};
\node[draw, black, fill=black!10, circle, thick, scale=1.1*\s] at (-\R*\z, -\R*\z) {};

\node[cross, rotate=\theta, scale=\s] at (-2.75*\R/\t, 2.5*\R/\t) {};
\node[anchor=west] at (-2*\R/\t, 2.5*\R/\t) {$\bs{x}$};
\node[cross, rotate=\theta, scale=\s] at (-2.75*\R/\t, -2.5*\R/\t) {};
\node[anchor=west] at (-2*\R/\t, -2.5*\R/\t) {$\bs{x}'$};

\node[anchor=south] at (0, 1.35*\R) {small random scatterer};

\end{tikzpicture}
\caption{In passive imaging with a random source (left), the defect $D$ is illuminated by an uncontrolled, random source located at points $\bs{z}$, which we assume to be distributed on a surface $\Sigma$ that encloses the defect. The medium's response is recorded at points $\bs{x}$ and $\bs{x}'$, contained in some measurement volume $B$. In passive imaging with a small random scatterer (right), the defect is illuminated by a single controlled point source located at $\bs{z}_\epsilon$. The incident field is scattered by a small random scatterer $D_\epsilon(\bs{y}_\epsilon)$ whose center $\bs{y}_\epsilon$ is located on $\Sigma_\epsilon$, creating a secondary, random source. The medium's response is recorded in some volume $B$.}
\label{fig:passive-LSM}
\end{figure}

We show, now, that the total field $\tilde{v}_\epsilon = \tilde{v}^i_\epsilon + \tilde{v}^s_\epsilon$ satisfies a \textit{modified} Helmholtz--Kirchhoff identity, which justifies the setup of \cref{fig:passive-LSM} (right).

\begin{theorem}[Modified Helmholtz--Kirchhoff identity]\label{thm:mod-HK-identity}
For any point $\bs{x}$ and $\bs{x}'$ in $B$, and small enough $\epsilon$,
\begin{align}\label{eq:mod-HK-identity}
u^s(\bs{x},\bs{x}') - \overline{u^s(\bs{x},\bs{x}')} = & \; 2ik\sigma_\epsilon\int_{\Sigma_\epsilon}\overline{\tilde{v}_\epsilon(\bs{x},\bs{y}_\epsilon,\bs{z}_\epsilon)}\tilde{v}_\epsilon(\bs{x}',\bs{y}_\epsilon,\bs{z}_\epsilon)ds(\bs{y}_\epsilon)\left[1 + \OO(\epsilon^{q-p})\right] \\ 
& \; - [\phi(\bs{x},\bs{x}') - \overline{\phi(\bs{x},\bs{x}')}], \nonumber
\end{align}
where the scaling factor $\sigma_\epsilon$ reads
\begin{align}
\sigma_\epsilon = \pi^2\vert H_0^{(1)}(2\pi\epsilon)\vert^2\epsilon^{-q}.
\end{align}
\end{theorem}

\begin{proof}
Recall that:
\begin{align}
\tilde{v}^i_\epsilon(\bs{x},\bs{y}_\epsilon,\bs{z}_\epsilon) = \mu_\epsilon(\bs{y}_\epsilon,\bs{z}_\epsilon)\phi(\bs{x},\bs{y}_\epsilon) \quad \text{and} \quad
\tilde{v}^s_\epsilon(\bs{x},\bs{y}_\epsilon,\bs{z}_\epsilon) = \mu_\epsilon(\bs{y}_\epsilon,\bs{z}_\epsilon)u^s(\bs{x},\bs{y}_\epsilon).
\end{align}
Therefore,
\begin{align}
u^s(\bs{x},\bs{y}_\epsilon) = \mu_\epsilon(\bs{y_\epsilon},\bs{z}_\epsilon)^{-1}\tilde{v}^s_\epsilon(\bs{x},\bs{y}_\epsilon,\bs{z}_\epsilon)
\end{align}
satisfies the standard Helmholtz--Kirchhoff identity
\begin{align}
u^s(\bs{x},\bs{x}') - \overline{u^s(\bs{x},\bs{x}')} = 2ik\int_{\Sigma_\epsilon}\overline{u(\bs{x},\bs{y}_\epsilon)}u(\bs{x}',\bs{y}_\epsilon)ds(\bs{y}_\epsilon) - \left[\phi(\bs{x},\bs{x}') - \overline{\phi(\bs{x},\bs{x}')}\right],
\end{align}
with total field
\begin{align}
u(\bs{x},\bs{y}_\epsilon) = \phi(\bs{x},\bs{y}_\epsilon) + u^s(\bs{x},\bs{y}_\epsilon) = \mu_\epsilon(\bs{y}_\epsilon,\bs{z}_\epsilon)^{-1}\tilde{v}_\epsilon(\bs{x},\bs{y}_\epsilon,\bs{z}_\epsilon).
\end{align}
We can rewrite the integral as
\begin{align}
2ik\int_{\Sigma_\epsilon}\vert\mu_\epsilon(\bs{y}_\epsilon,\bs{z}_\epsilon)^{-1}\vert^2\,\overline{\tilde{v}_\epsilon(\bs{x},\bs{y}_\epsilon,\bs{z}_\epsilon)}\tilde{v}_\epsilon(\bs{x}',\bs{y}_\epsilon,\bs{z}_\epsilon)ds(\bs{y}_\epsilon).
\end{align}
We conclude by using \cref{lem:appendix-scal}.
\end{proof}

Our modified Helmholtz--Kirchhoff identity in \cref{thm:mod-HK-identity} justifies the introduction of the following \textit{modified} cross-correlation matrix,
\begin{align}\label{eq:mod-cross-cor-mat}
\widetilde{C}_{jm} = \frac{2ik\vert\Sigma_\epsilon\vert\sigma_\epsilon}{L}\sum_{\ell=1}^L\overline{\tilde{v}_\epsilon(\bs{x}_j,\bs{y}_\epsilon^\ell,\bs{z}_\epsilon)}\tilde{v}_\epsilon(\bs{x}_m,\bs{y}_\epsilon^\ell,\bs{z}_\epsilon) - \left[\phi(\bs{x}_j,\bs{x}_m) - \overline{\phi(\bs{x}_j,\bs{x}_m)}\right],
\end{align}
whose accuracy and robustness we demonstrate next in a series of numerical experiments.

\section{Numerical experiments}\label{sec:numerics}

The solution to the inverse acoustic scattering problem consists of two steps. First, the modified cross-correlation matrix \cref{eq:mod-cross-cor-mat} is filled out in the data acquisition step (direct problem). This entails solving \cref{eq:w^s} for $L$ different positions $\bs{y}_\epsilon^\ell$ of the small scatterer and evaluating the solution at $J$ points $\bs{x}_j$ (for a given $\bs{z}_\epsilon$). This yields a $J\times L$ near-field matrix $N$ with entries $N_{j\ell}=w^s_\epsilon(\bs{x}_j, \bs{y}_\epsilon^\ell, \bs{z}_\epsilon)$. We then remove its column-average, generating a $J\times L$ matrix $\widetilde{N}$ with entries $\widetilde{N}_{j\ell} \approx \tilde{v}_\epsilon(\bs{x}_j, \bs{y}_\epsilon^\ell, \bs{z}_\epsilon)$. This allows us to assemble the $J\times J$ matrix \cref{eq:mod-cross-cor-mat}. Second, we probe the medium by solving the system $\widetilde{C}g_{\bs{s}} = \phi_{\bs{s}}$, for several sampling points $\bs{s}\in\R^2$, in the data processing step (inverse problem). The right-hand side reads $(\phi_{\bs{s}})_j = \phi(\bs{x}_j, \bs{s})$, $1\leq j\leq J$. The boundary $\partial D$ of the unknown defect $D$ coincides with those points $\bs{s}$ for which $\Vert g_{\bs{s}}\Vert_2$ is large \cite[Thm.~4.3]{montanelli2023}.

\paragraph{Solving the direct problem} Our MATLAB implementation leverages the capabilities of \href{https://github.com/matthieuaussal/gypsilab}{gypsilab}, an open-source toolbox designed for efficient boundary element computations \cite{alouges2018}. The approach adopted employs the combined boundary integral formulation for the exterior Dirichlet problem \cref{eq:w^s}, allowing for the computation of weakly and strongly singular and near-singular integrals through the methods delineated in \cite{montanelli2022, montanelli2024a}. It is preferable to employ a combined integral approach, as it is coercive for large wavenumbers \cite{spence2015}. We utilize up to 100 points to discretize the boundary and aim for four-digit accuracy.

\paragraph{Solving the inverse problem} To simulate noisy measurements, we add some random noise with amplitude $\delta$ to the near-field matrix $N$, before constructing $\widetilde{C}$. This yields a noisy matrix $\widetilde{C}_\delta$. To solve $\widetilde{C}_\delta g_{\bs{s}} = \phi_{\bs{s}}$, we compute the SVD of the matrix $\widetilde{C}_\delta$, $\widetilde{C}_\delta = \widetilde{U}_\delta \widetilde{S}_\delta \widetilde{V}_\delta^*$, and apply Tikhonov regularization with parameter $\alpha>0$. To choose $\alpha$, we use Morozov's discrepancy principle. For details, we refer the reader to \cite[sect.~5]{montanelli2023}.

\paragraph{Full-aperture measurements} We consider an ellipse and a kite of size $\lambda/2$ centered at $-2\lambda-2\lambda i$ and $2\lambda + 2\lambda i$ for the wavenumber $k=2\pi$ (wavelength $\lambda =1$). The ellipse has axes $a=1.5$ and $b=1$, while the kite is that of \cite[sect.~3.6]{colton2019}. We take $\epsilon=10^{-2}$, $p=1$, $q=2$, and $\theta_z=\pi$ for the asymptotic model, which yields $\bs{z}_\epsilon=-10000$. (We will keep all parameters listed thus far unchanged throughout all experiments.) For the LSM, we take $J=120$ equispaced sensors on the circle of radius $5\lambda=5$, 
\begin{align}\label{eq:sensors}
\bs{x}_j = 5e^{i\theta_x^j},  \quad \theta_x^j = \frac{2\pi}{J}(j - 1), \quad  1\leq j\leq J,
\end{align}
and $L=150$ different positions of a single small scatterer on the circle of radius $\lambda\epsilon^{-p}=100$,
\begin{align}\label{eq:scatterer-beta}
\bs{y}_\epsilon^\ell = 100e^{i\theta_y^\ell}, \quad \theta_y^\ell = \frac{2\pi}{L}(\ell - 1 + \beta_\ell), \quad 1\leq \ell\leq L,
\end{align}
where the $\beta_\ell$'s are independent and identically distributed with the uniform distribution over $(0,0.1)$. This is a (slightly) perturbed equispaced grid, with the corresponding quadrature in \cref{eq:mod-cross-cor-mat} serving as an accurate approximation of the integral within \cref{eq:mod-HK-identity} (see also the comments preceding \cref{thm:mod-HK-identity}). Finally, we add some multiplicative noise with amplitude $5\times10^{-3}$ to the near-field measurements,\footnote{In theory, to ensure the preservation of the signal encapsulated in $\tilde{v}_\epsilon$, the noise level must align, up to a logarithmic factor, with the approximation error $\OO(\epsilon^p\epsilon^{q/2})=10^{-4}$ described in \cref{thm:e}. However, in practice, we managed to obtain satisfactory reconstructions even with noise levels as high as $5\times10^{-3}$. This still presents a clear limitation, which we will partially address in the last experiment.} and probe the medium on a $100\times100$ uniform grid on $[-6\lambda,6\lambda]\times[-6\lambda,6\lambda]$. The results are shown in \cref{fig:elli} and \cref{fig:kite} for the ellipse and the kite. In both cases, the defect is well identified by our novel LSM, based on cross-correlations and a small random scatterer. Our method can also handle several obstacles, as shown in \cref{fig:kite-two}.

\begin{figure}
\centering
\def\scl{0.25}
\includegraphics[scale=\scl]{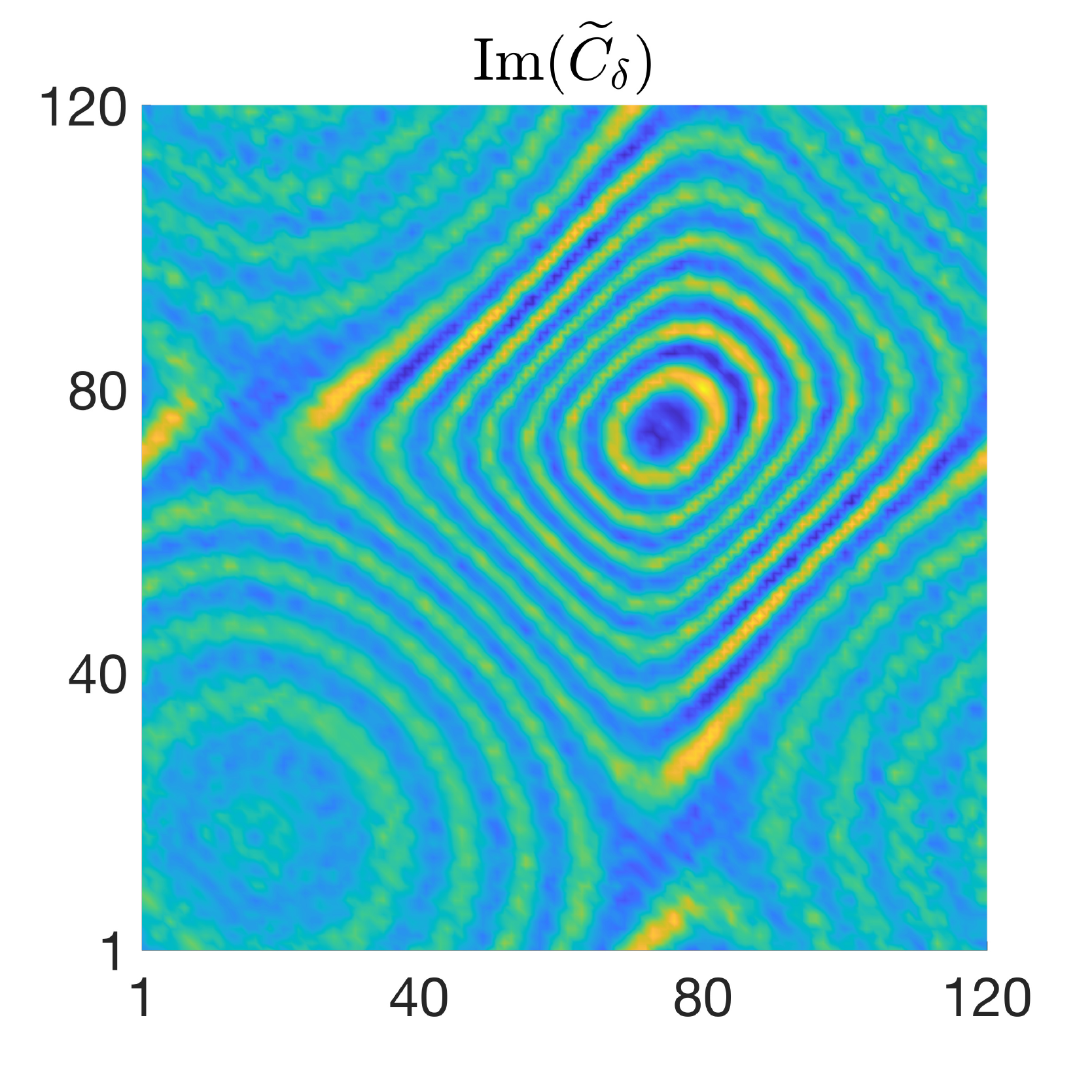}
\includegraphics[scale=\scl]{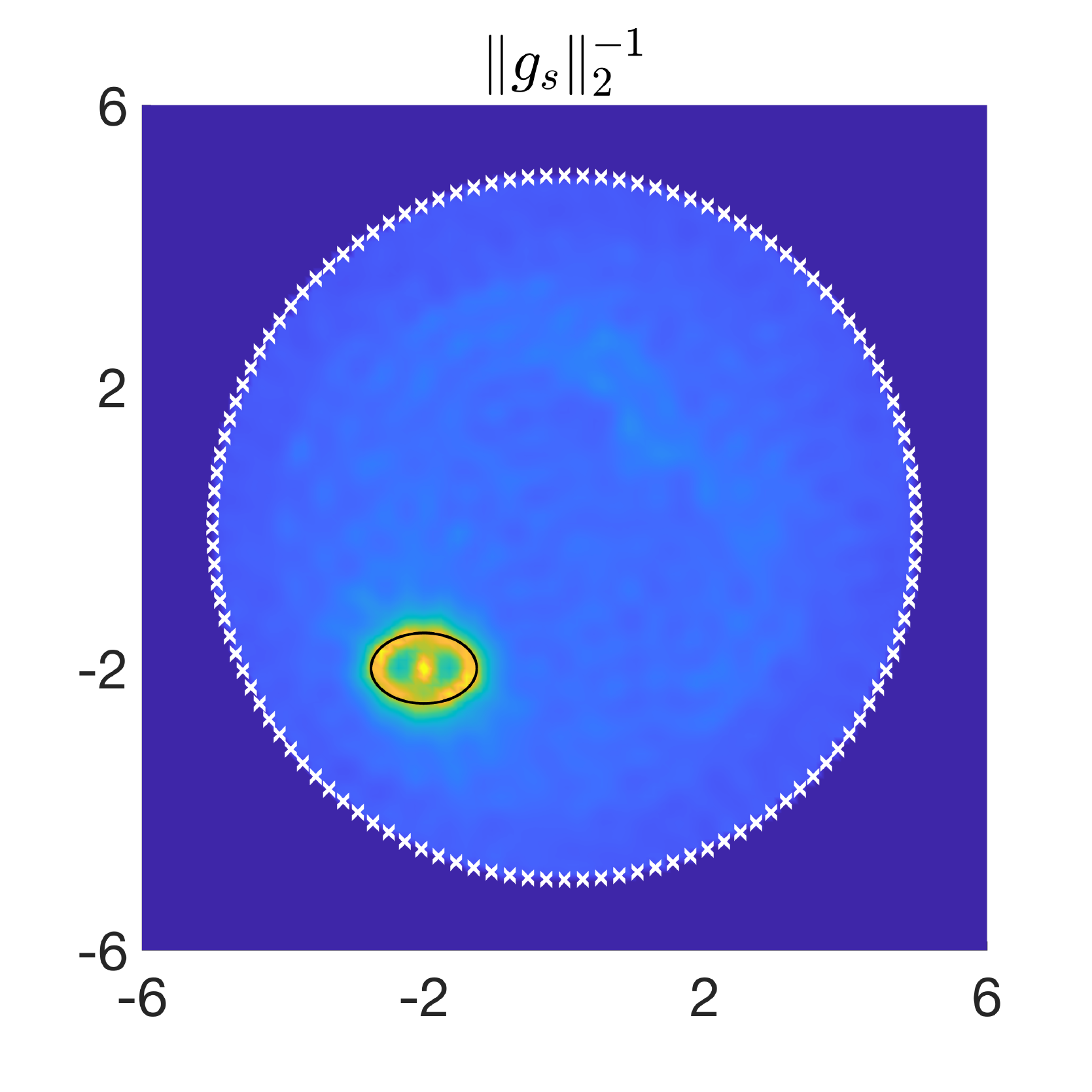}
\caption{The left picture displays the imaginary part of the modified cross-correlation matrix \cref{eq:mod-cross-cor-mat} and the right picture displays the indicator function (values outside of the disk of radius $5\lambda$ were zeroed out). The measurement points are represented by crosses, while the true boundary is depicted in a solid black line. Our sampling method, based on cross-correlations and a small random scatterer, successfully identified $D$.}
\label{fig:elli}
\end{figure}

\begin{figure}
\centering
\def\scl{0.25}
\includegraphics[scale=\scl]{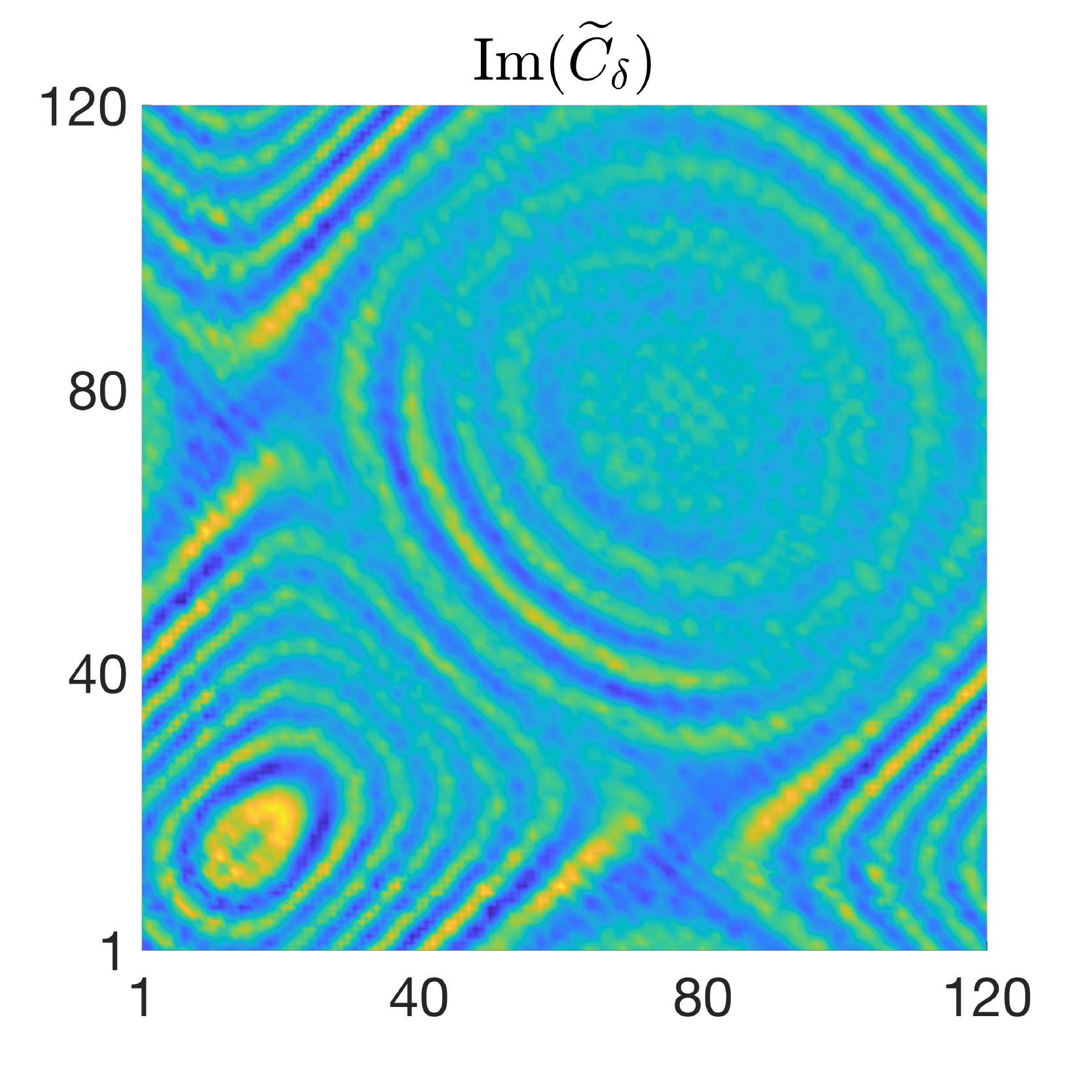}
\includegraphics[scale=\scl]{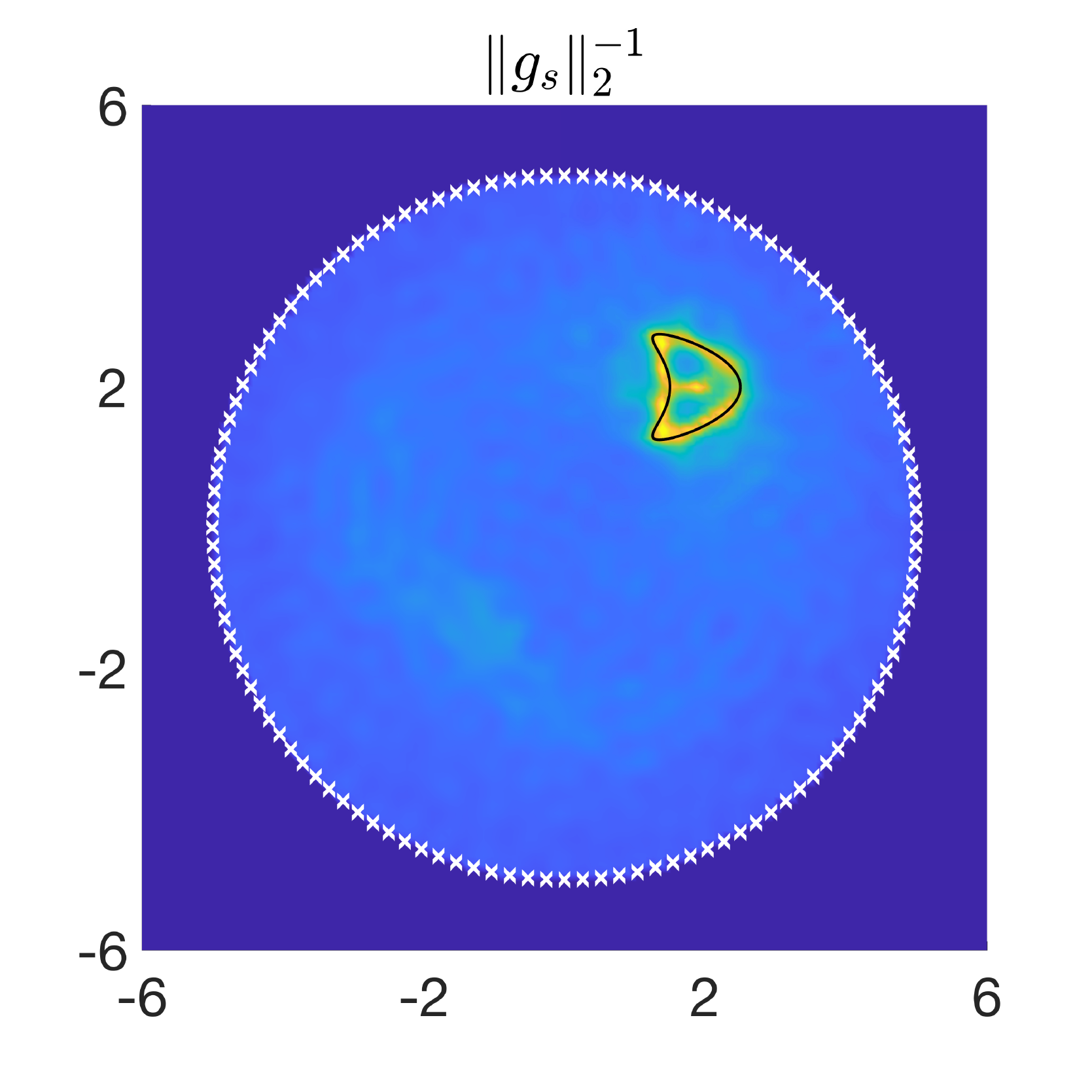}
\caption{In this experiment, too, the defect (a kite) is well identified by our method. This demonstrates that the LSM can be utilized in passive imaging with data generated by a small random scatterer. The setup in this figure is the same as in \cref{fig:elli}.}
\label{fig:kite}
\end{figure}

\begin{figure}
\centering
\def\scl{0.25}
\includegraphics[scale=\scl]{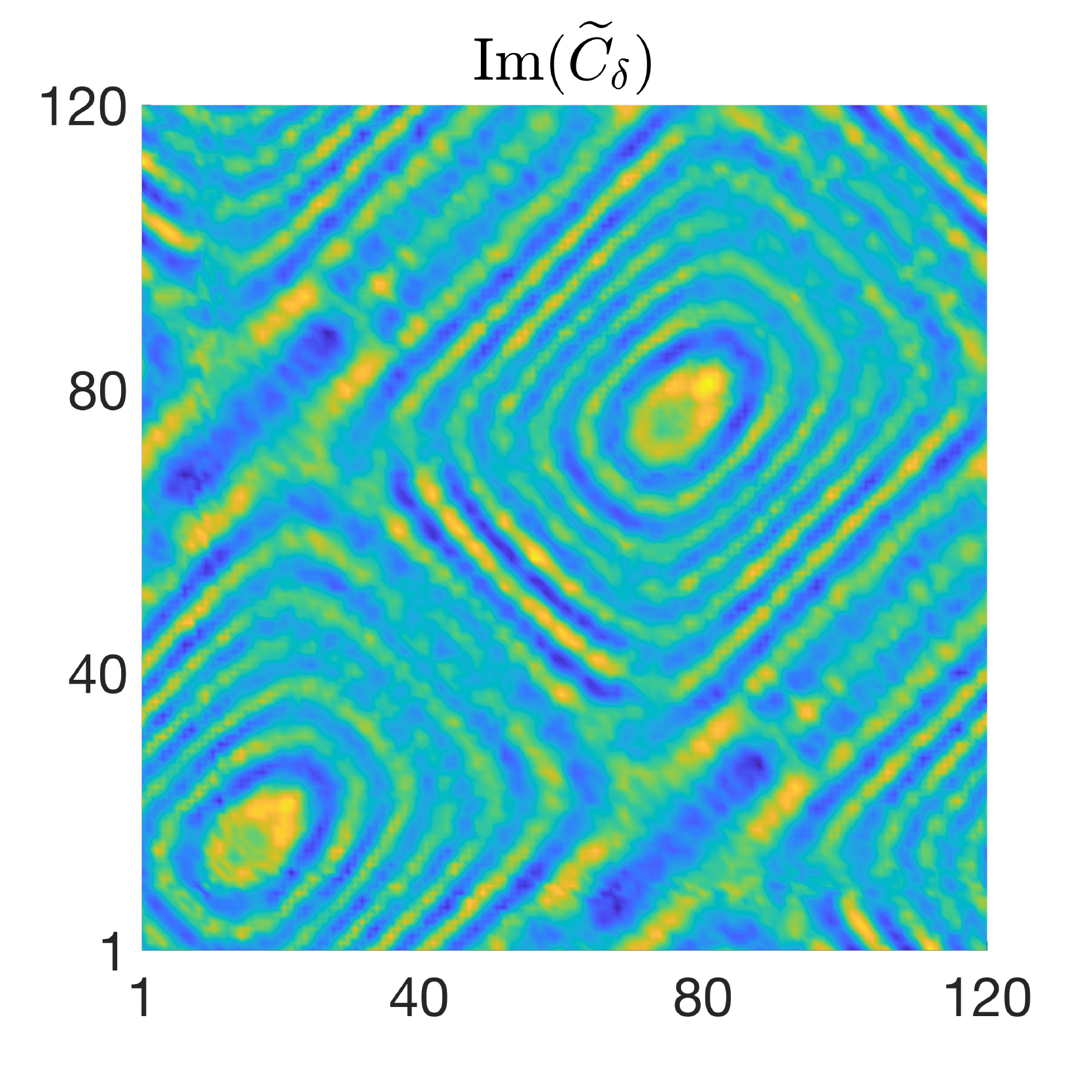}
\includegraphics[scale=\scl]{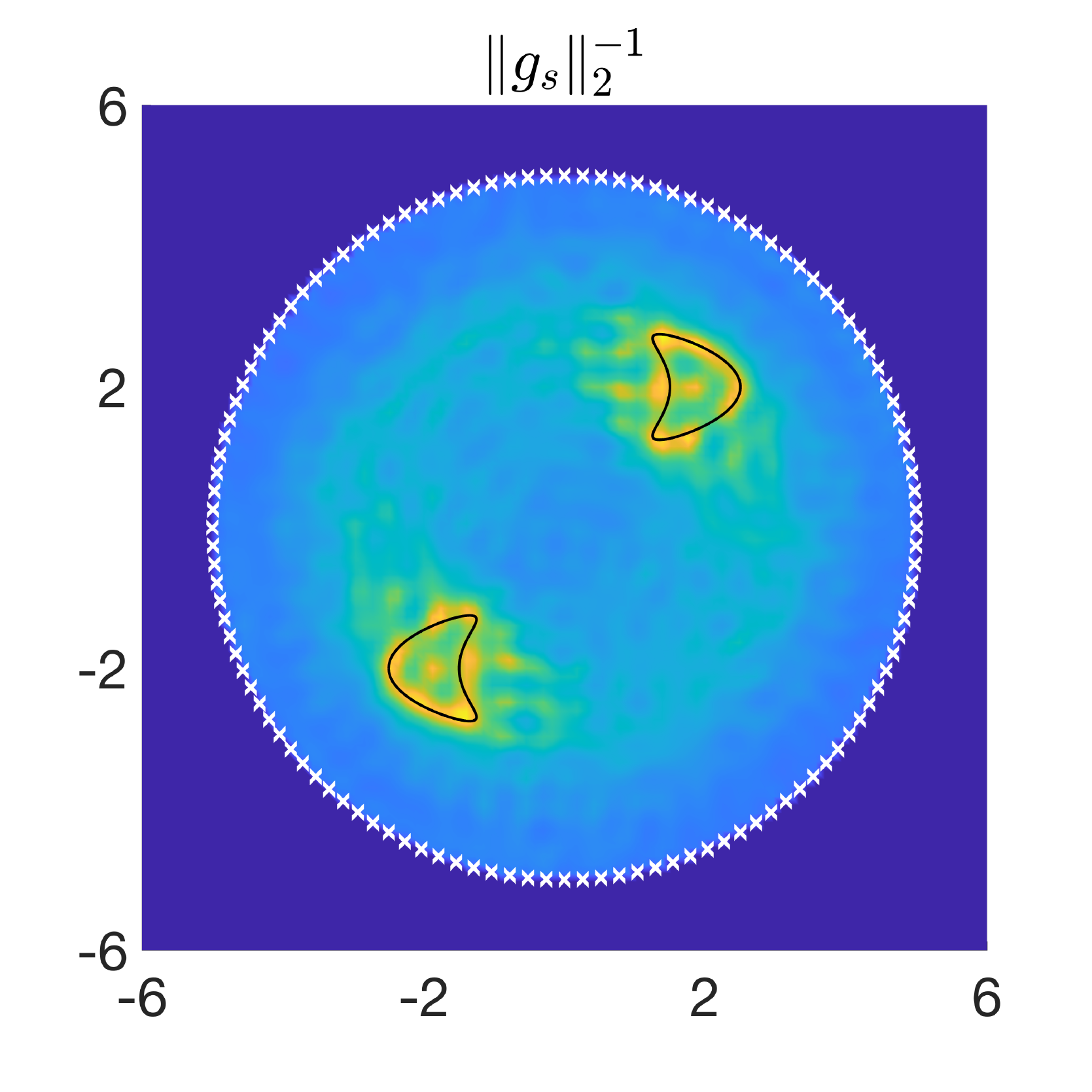}
\caption{Our method can handle several obstacles (here, two kites), as long as they are separated by a few wavelengths. The setup in this figure is the same as in \cref{fig:elli}.}
\label{fig:kite-two}
\end{figure}

We now consider a less favorable but more realistic scenario where the positions of the single small scatterer move according to
\begin{align}\label{eq:scatterer-rand}
\bs{y}_\epsilon^\ell = 100e^{i\theta_y^\ell}, 
\quad 1\leq \ell\leq L,
\end{align}
where the $\theta_y^\ell$'s are independent and identically distributed with the uniform distribution over $(0,2\pi)$. We present the results in \cref{fig:kite-rand} for $J=120$ sensors as described in \cref{eq:sensors}, $L=400$ realizations, and noise level $5\times10^{-3}$; these are not as good as previously. The reason for this is that the sum in \cref{eq:mod-cross-cor-mat} with the quadrature points as given by \cref{eq:scatterer-rand} is a rather poor approximation to the integral in \cref{eq:mod-HK-identity}, the quadrature error being $\OO(1/\sqrt{L})$.

\begin{figure}
\centering
\def\scl{0.25}
\includegraphics[scale=\scl]{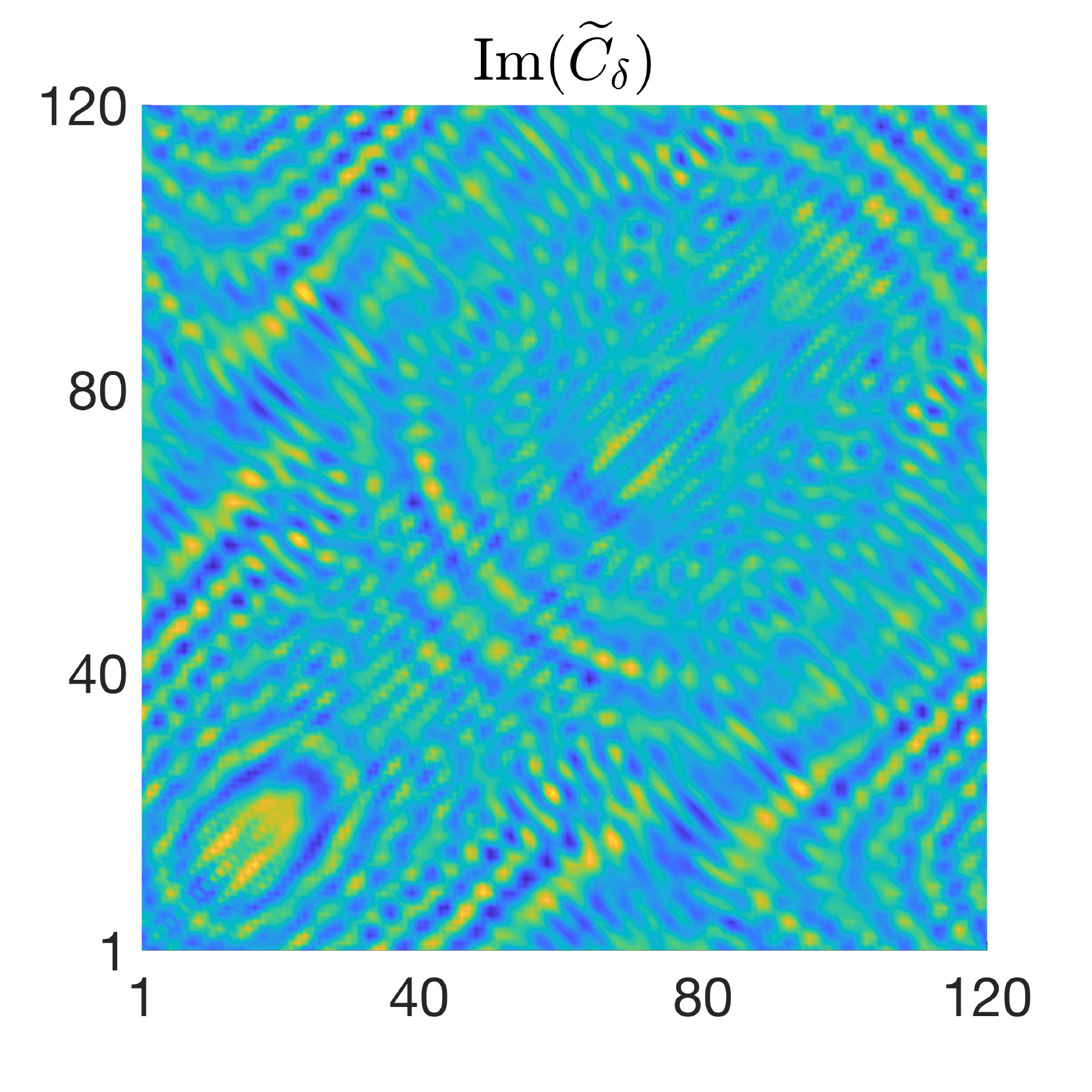}
\includegraphics[scale=\scl]{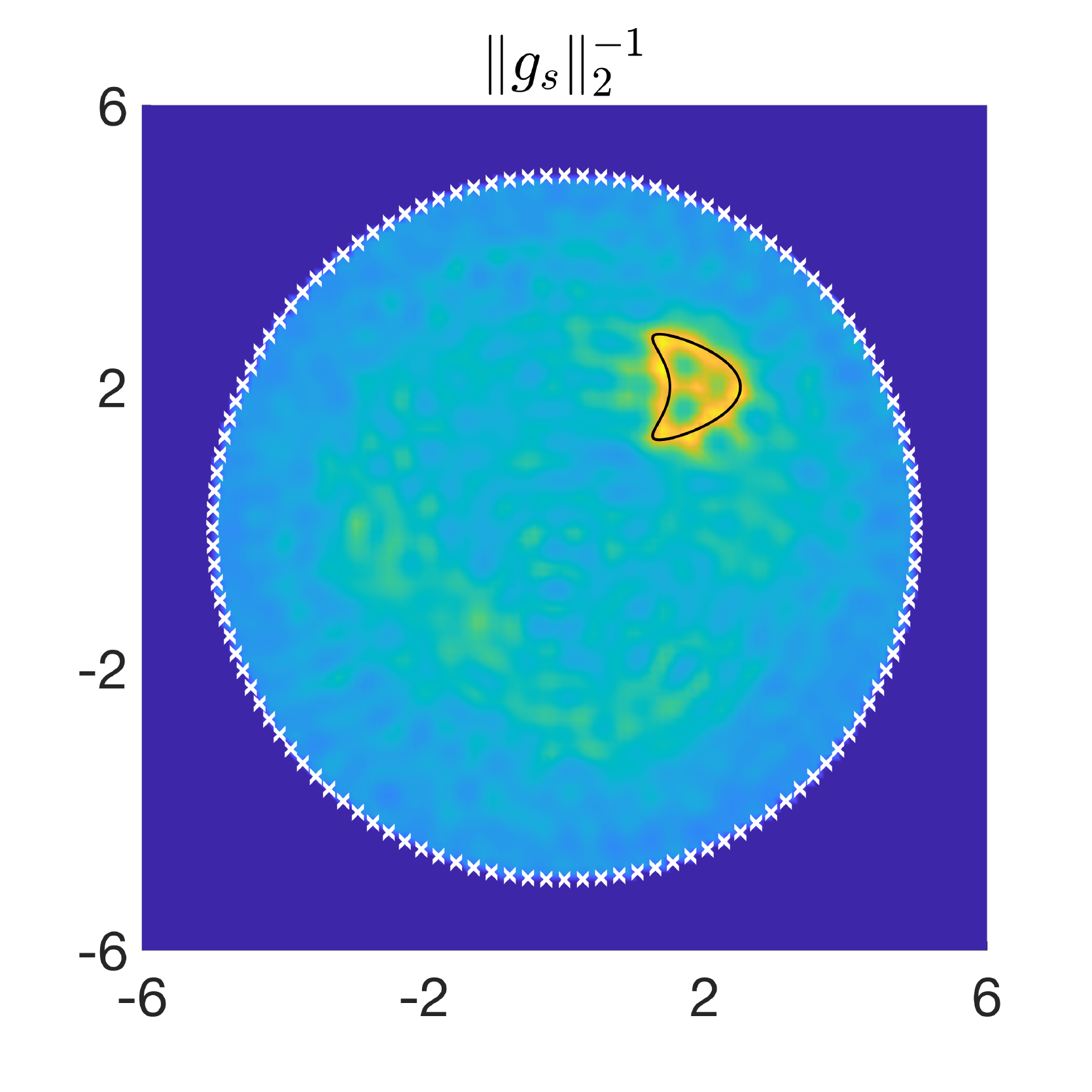}
\caption{When the small scatterer moves randomly around the obstacle according to \cref{eq:scatterer-rand}, the quadrature error associated with our modified cross-correlation matrix is $\OO(1/\sqrt{L})$. It is therefore necessary to increase the number of realizations $L$ from $150$ to $400$ to obtain good numerical results. The entries of the matrix (on the left) are noisy approximations to those of \cref{fig:kite}.}
\label{fig:kite-rand}
\end{figure}

\paragraph{Limited-aperture measurements} In this numerical experiment, we explore measurements with a restricted aperture, as analyzed in \cite{audibert2017}. The outcomes are displayed in \cref{fig:kite-limi} for $J=120$ sensors, $L=150$ different realizations of \cref{eq:scatterer-beta}, and noise amplitude $5\times10^{-3}$. Although this configuration leads to much more challenging reconstructions, our approach based on the modified cross-correlation matrix \cref{eq:mod-cross-cor-mat} and a small random scatterer produces results comparable to those obtained using the cross-correlation matrix \cref{eq:cross-cor-mat} and a random source (see \cite[sect.~5]{montanelli2023}).

\begin{figure}
\centering
\def\scl{0.21}
\includegraphics[scale=\scl]{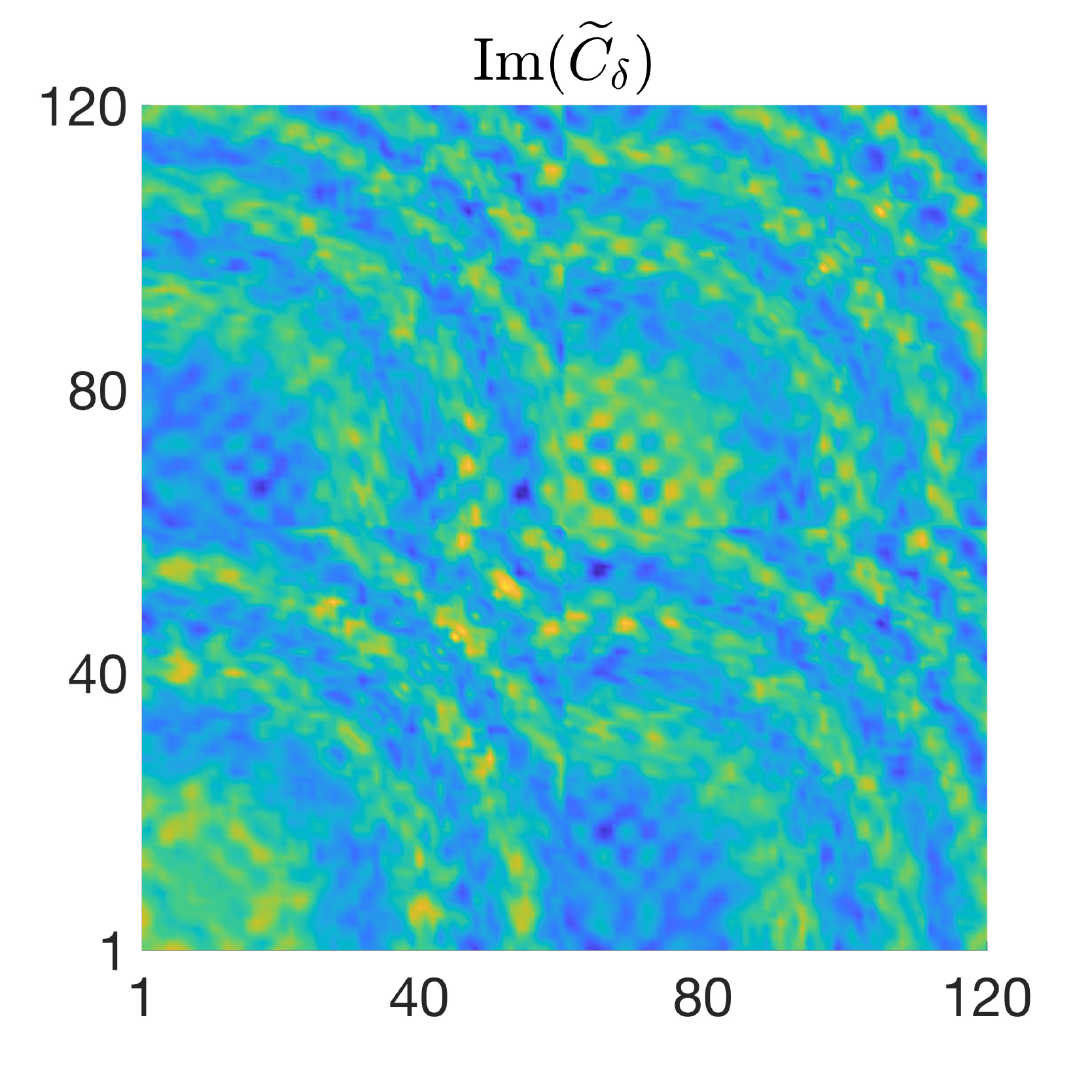}
\includegraphics[scale=\scl]{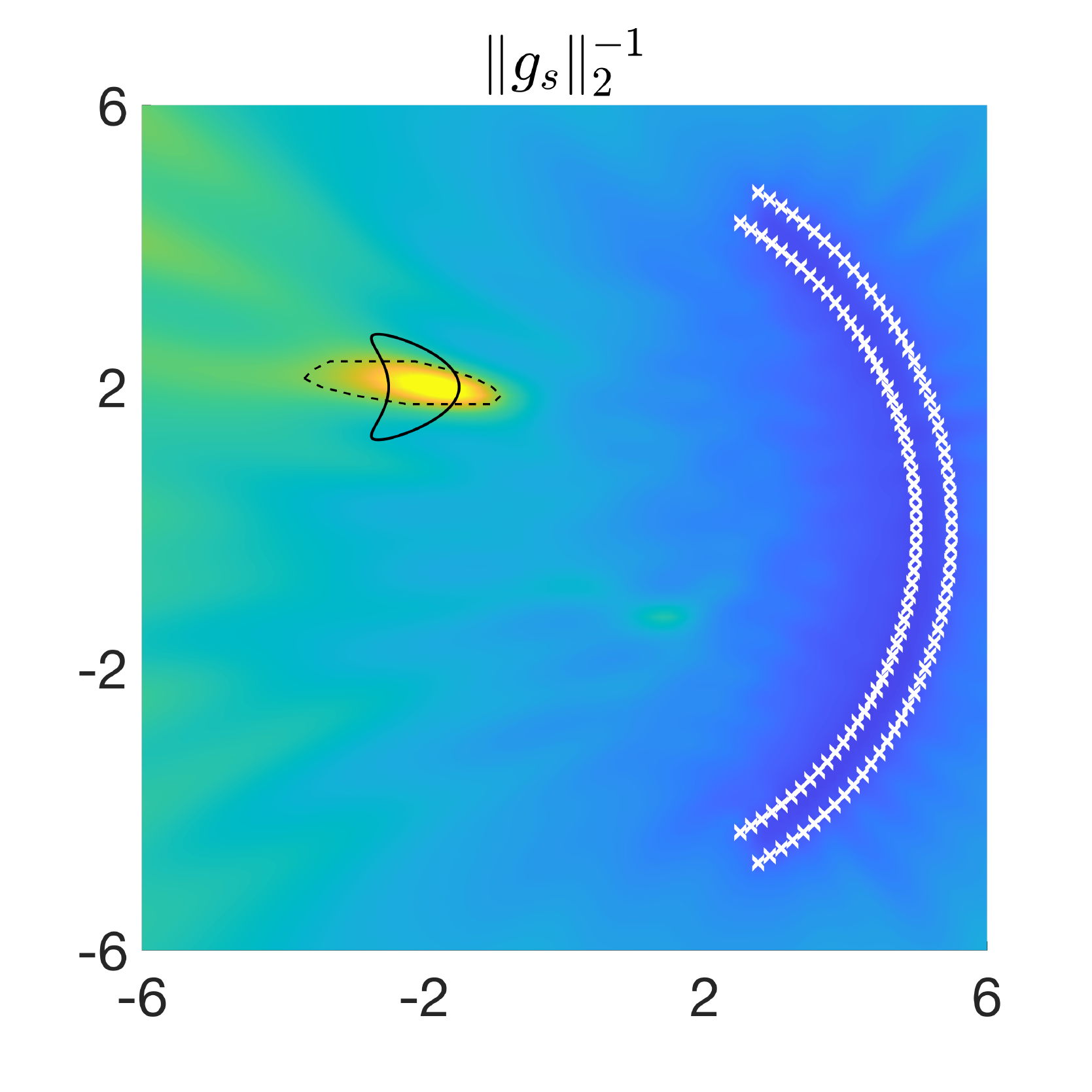}
\includegraphics[scale=\scl]{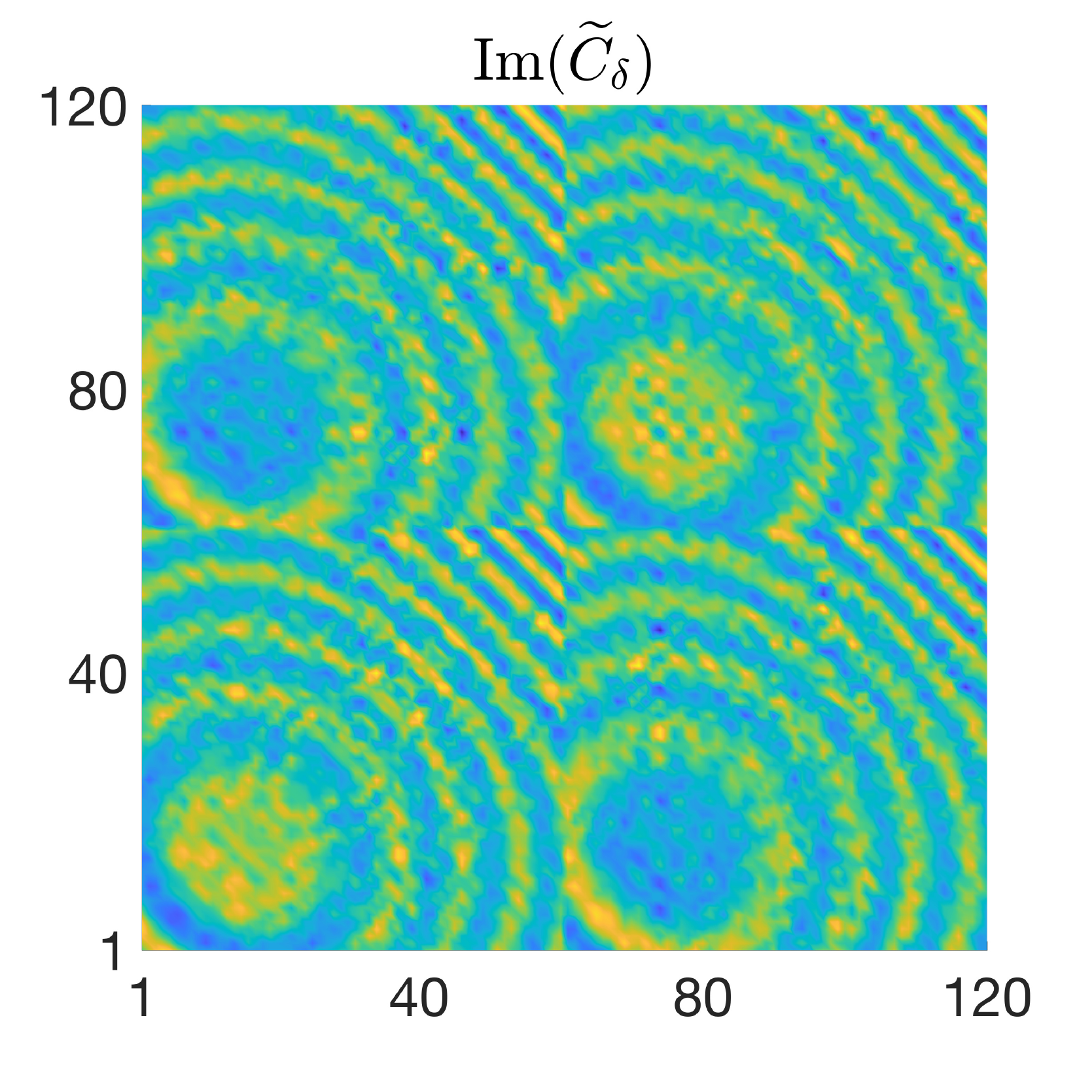}
\includegraphics[scale=\scl]{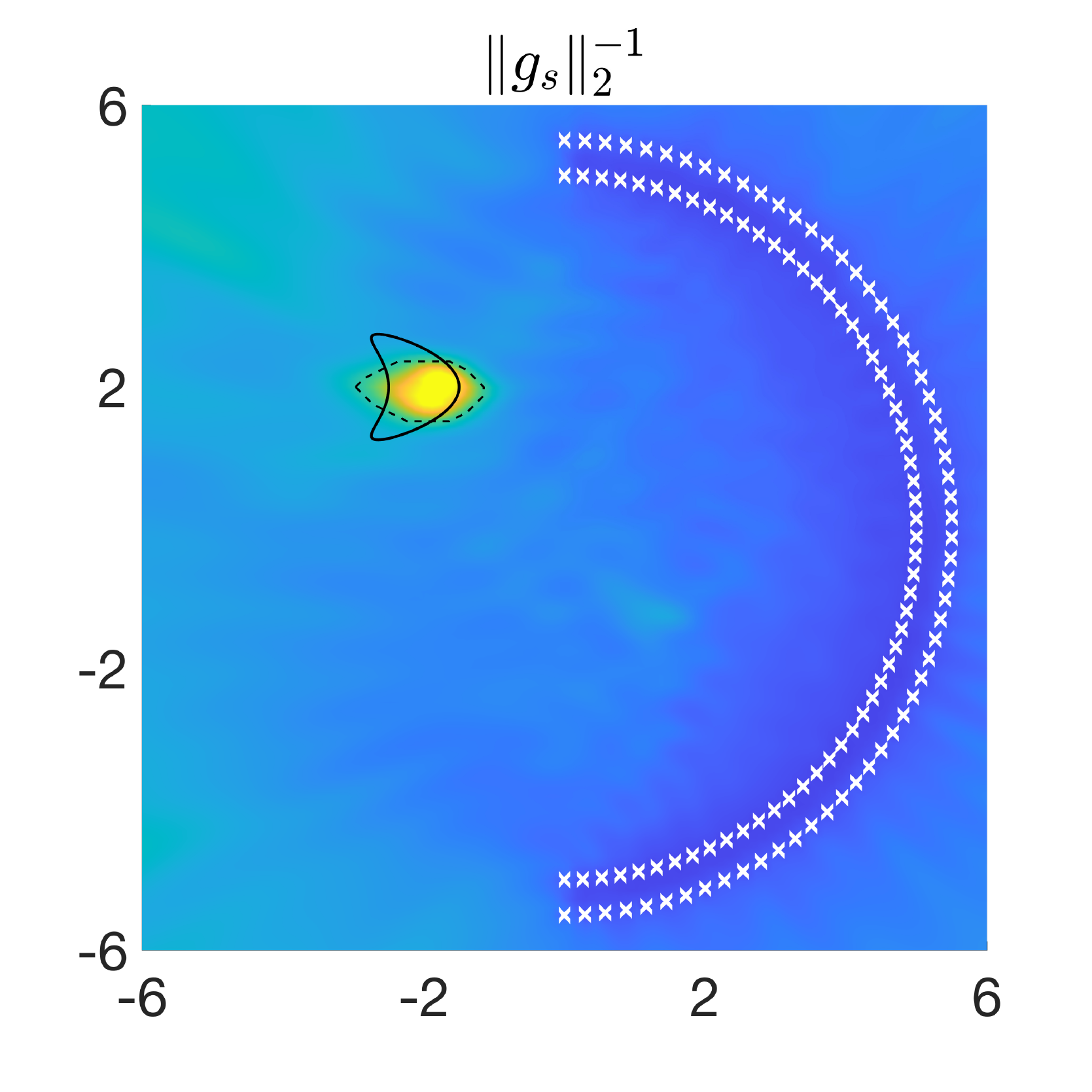}
\includegraphics[scale=\scl]{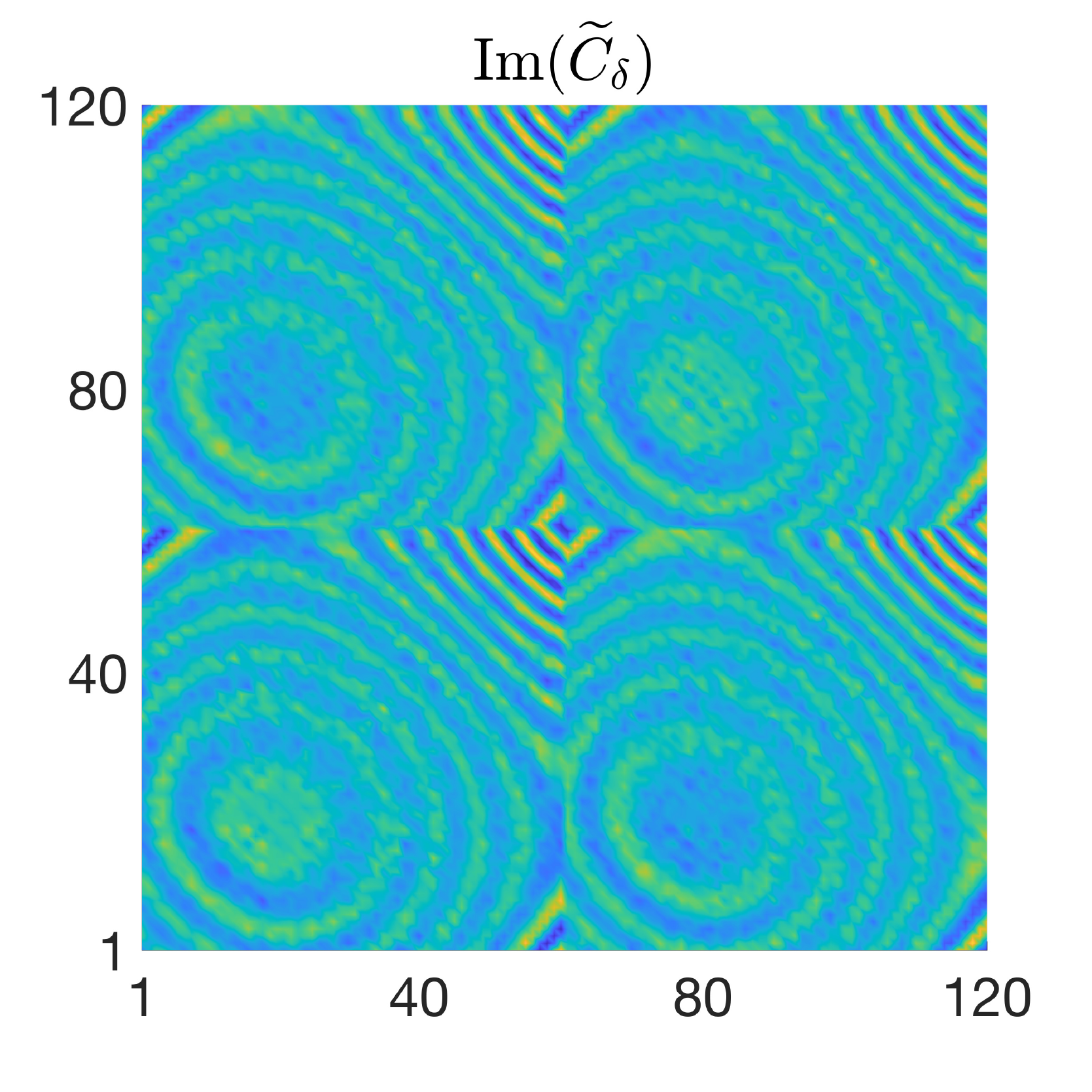}
\includegraphics[scale=\scl]{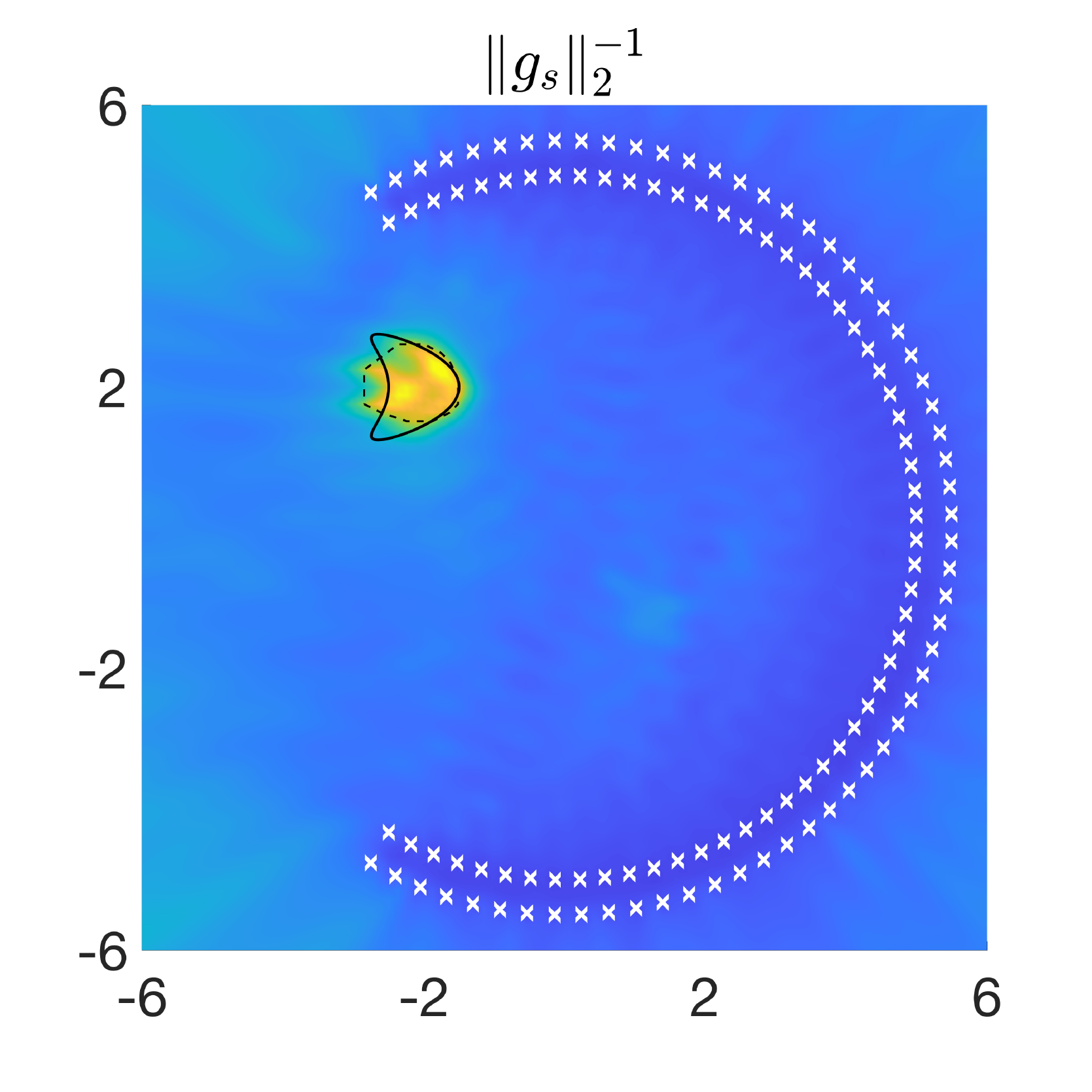}
\caption{In the case of limited-aperture measurements, achieving a perfect reconstruction of the shape is not anticipated. Here, the aperture is $2\pi/3$ in the top row, $\pi$ is the middle row, and $4\pi/3$ in the bottom row. The dotted line corresponds to the convex hull of the level set of the indicator function that minimizes the area error. The error is minimized at $0.54$ (top), $0.45$ (middle), and $0.62$ (bottom) times the maximum of the indicator function, with corresponding errors of $3.51\times10^{-2}$, $2.10\times10^{-2}$, and $3.94\times10^{-3}$, respectively.}
\label{fig:kite-limi}
\end{figure}

\paragraph{Several small scatterers} We now assume that there are $R>1$ several small random scatterers around $D$. At each acquisition $\ell$, the scatterers, indexed by $r$, are located at
\begin{align}
\bs{y}_\epsilon^{\ell,r} = 100e^{i\theta_y^{\ell,r}},
\quad 1\leq \ell\leq L, \quad 1\leq r\leq R,
\end{align}
where the $\theta_y^{\ell,r}$'s are independent and identically distributed with the uniform distribution over $(0,2\pi)$. Let $\bs{Y}_\epsilon^\ell=\{\bs{y}_\epsilon^{\ell,r}\}_{r=1}^R$ denote the positions of all small scatterers at a given acquisition $\ell$ and define
\begin{align}\label{eq:v-several}
\tilde{v}_\epsilon(\bs{x},\bs{Y}_\epsilon^\ell,\bs{z}_\epsilon) = \sum_{r=1}^R\tilde{v}_\epsilon(\bs{x},\bs{y}_\epsilon^{\ell,r},\bs{z}_\epsilon).
\end{align}
In this scenario, the scattered field $w^s_\epsilon$ reads
\begin{align}
w_\epsilon^s(\bs{x},\bs{Y}_\epsilon^\ell,\bs{z}_\epsilon) \approx u^s(\bs{x},\bs{z}_\epsilon) + \tilde{v}_\epsilon(\bs{x},\bs{Y}_\epsilon^\ell,\bs{z}_\epsilon).
\end{align}
We use the same technique to remove $u^s$ and assemble the matrix
\begin{align}\label{eq:mod-cross-cor-mat-R}
\widetilde{C}_{jm}^R = \frac{2ik\vert\Sigma_\epsilon\vert\sigma_\epsilon}{LR}\sum_{\ell=1}^L\overline{\tilde{v}_\epsilon(\bs{x}_j,\bs{Y}_\epsilon^\ell,\bs{z}_\epsilon)}\tilde{v}_\epsilon(\bs{x}_m,\bs{Y}_\epsilon^\ell,\bs{z}_\epsilon) - \left[\phi(\bs{x}_j,\bs{x}_m) - \overline{\phi(\bs{x}_j,\bs{x}_m)}\right].
\end{align}
The results are shown in \cref{fig:kite-several} for $R=5$ (top row) and $R=30$ (bottom row), with $J=120$ sensors described in \cref{eq:sensors}, $L=400$ realizations, and noise levels of $10^{-2}$ (top row) and $5\times10^{-2}$ (bottom row). This increased noise level is manageable because the signal magnitude within \cref{eq:v-several} is boosted by a factor of $R$ compared to previous scenarios. This indicates a significant improvement and opens avenues for future endeavors.

\begin{figure}
\centering
\def\scl{0.25}
\includegraphics[scale=\scl]{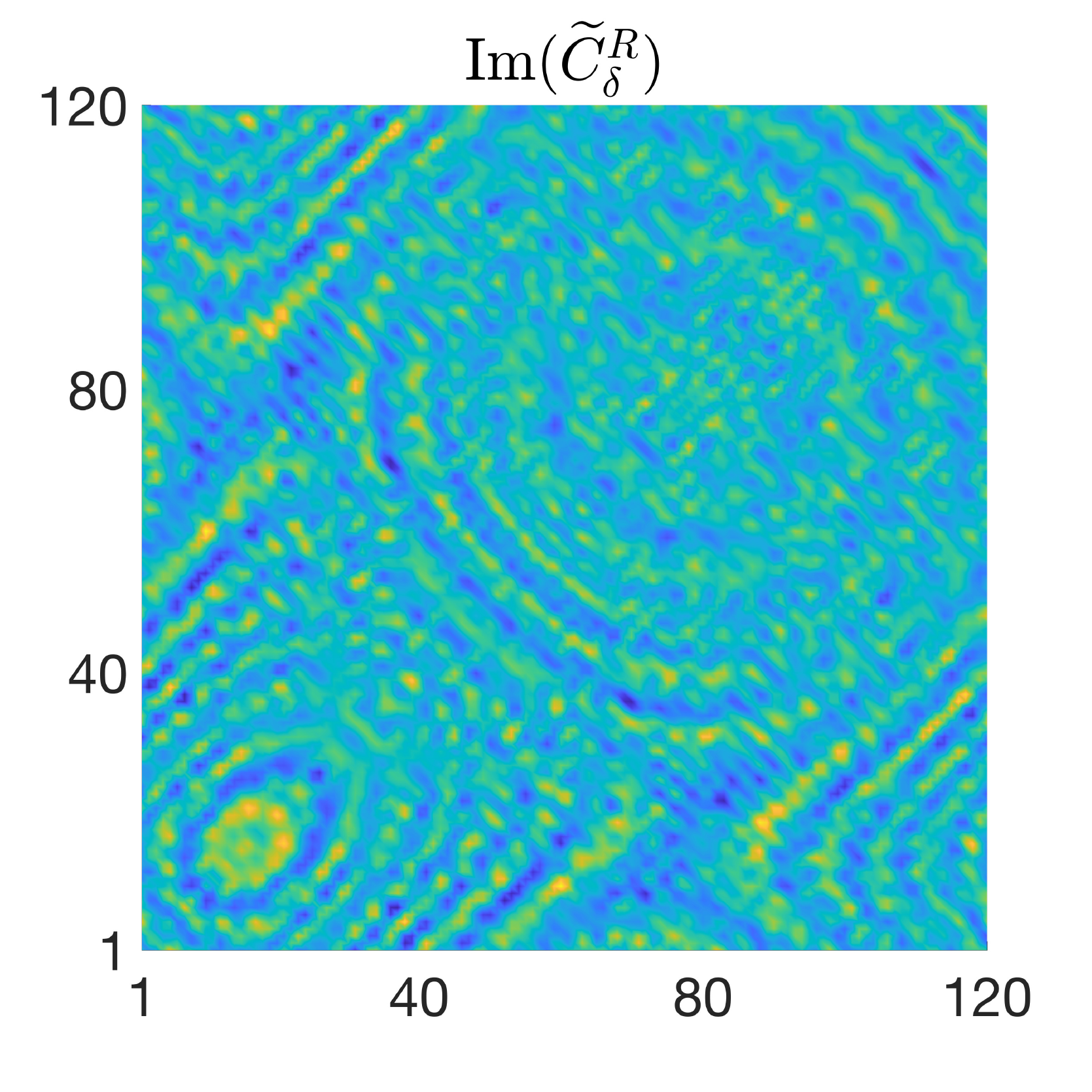}
\includegraphics[scale=\scl]{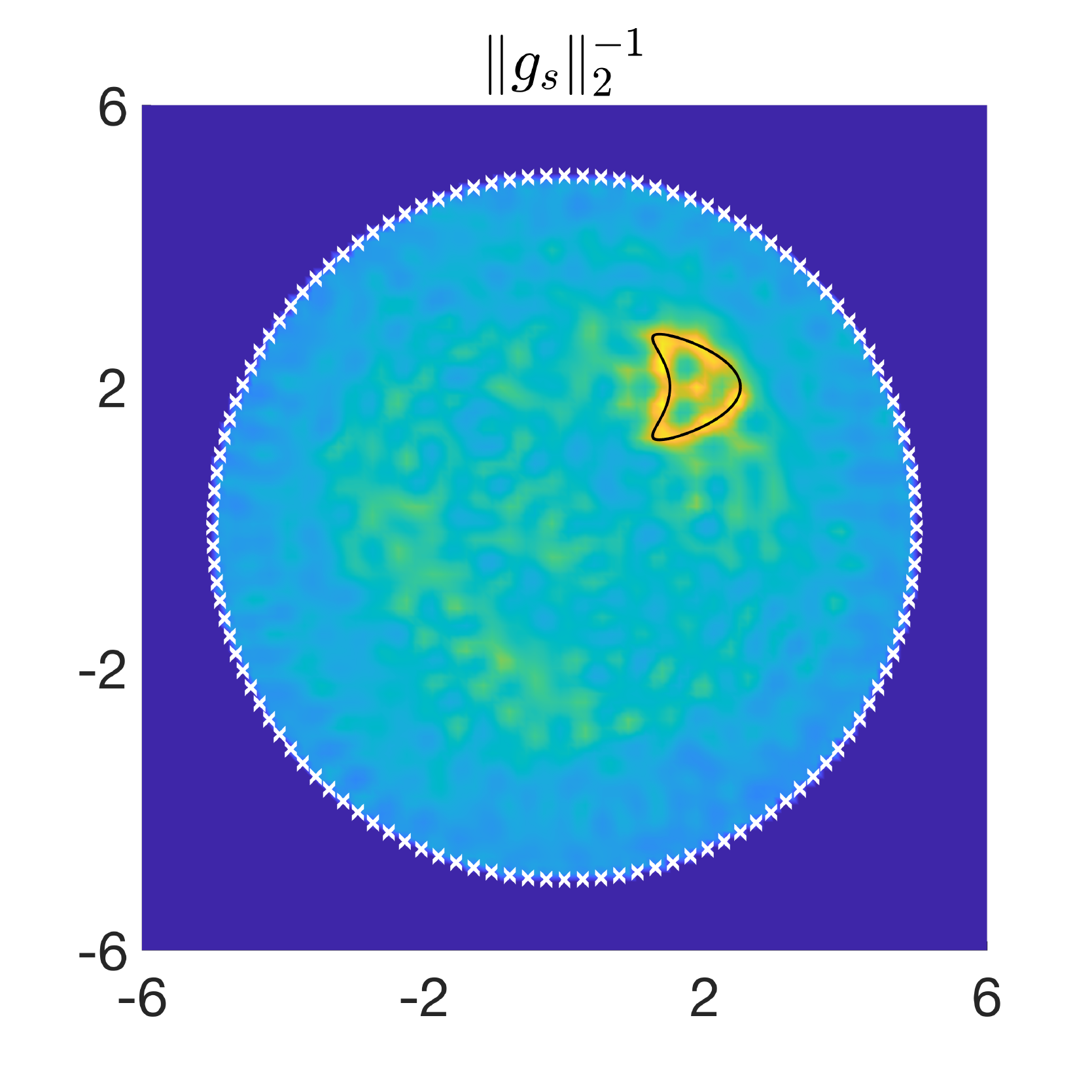}
\includegraphics[scale=\scl]{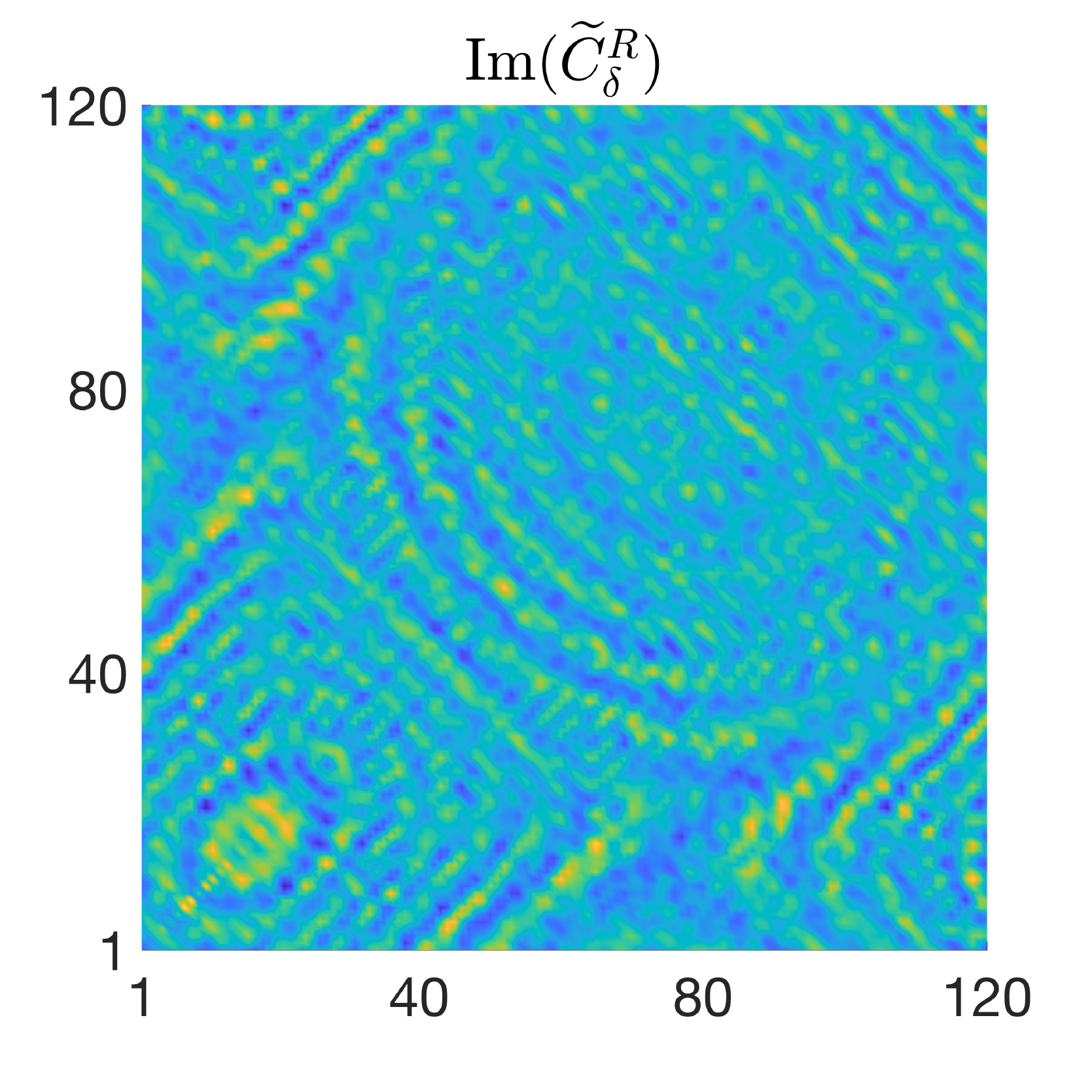}
\includegraphics[scale=\scl]{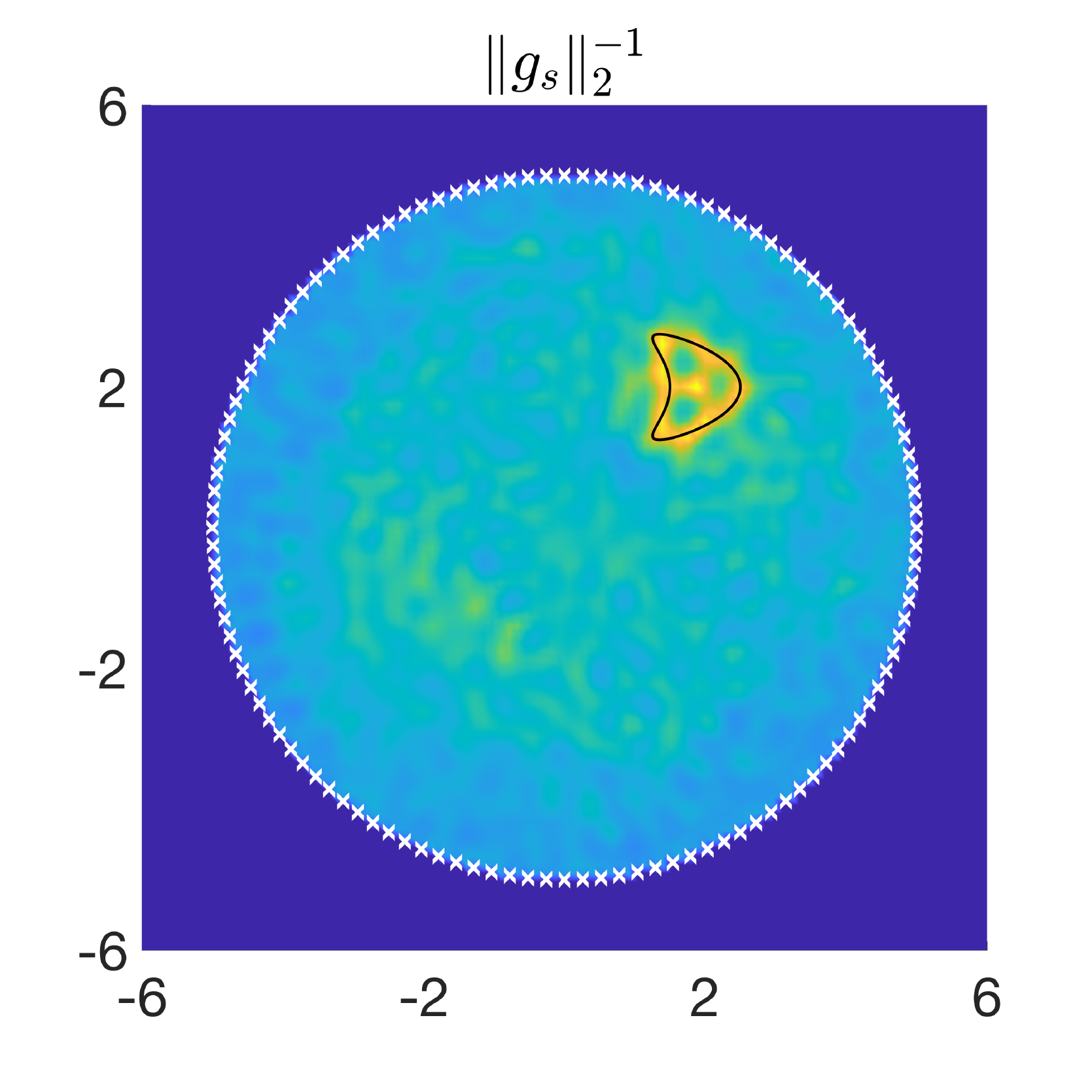}
\caption{When using several small scatterers, the signal is amplified by the number $R$ of scatterers. With $R=5$ random scatterers, we were able to increase the noise level from $5\times10^{-3}$ to $10^{-2}$ (top row). With $R=30$, we managed to raise the noise level to $5\times10^{-2}$ (bottom row).}
\label{fig:kite-several}
\end{figure}

\section{Conclusions}

We have introduced a novel version of the LSM as a powerful tool for addressing the sound-soft inverse scattering problem in two dimensions featuring randomly distributed small scatterers. Our approach is underpinned by a robust theoretical foundation, leveraging a rigorous asymptotic model and a modified Helmholtz--Kirchhoff identity, building upon our prior work on the LSM for random sources \cite{montanelli2023}. The implementation is comprehensive, incorporating essential components such as boundary elements, SVD, Tikhonov regularization, and Morozov's discrepancy principle. Finally, the numerical experiments presented in this paper demonstrate the effectiveness, accuracy, and robustness of our algorithms across various scenarios. The code used for generating the figures in the numerical experiments is available on GitHub at \href{https://github.com/Hadrien-Montanelli/lsmlab}{github.com/Hadrien-Montanelli/lsmlab}.

This study has laid the groundwork for potential extensions and broader applications. The first avenue involves extending the developed framework to three dimensions and exploring its adaptability to different boundary conditions, enhancing the model's versatility. Additionally, expanding the model to encompass a continuous distribution of random scatterers on a surface would move us toward a more realistic configuration. This adjustment might also address the noise issues raised in the numerical experiments, although a detailed analysis of this aspect is beyond the paper's scope. The subsequent progression towards a heightened level of realism involves contemplating random scatterers distributed within a volume. Finally, we would like to extend our sampling method to elastic waves, enabling its application to subsurface imaging.

\section*{Acknowledgments}

We sincerely thank the members of the Inria Idefix and Makutu research teams, in particular Lorenzo Audibert, H\'{e}l\`{e}ne Barucq, and Florian Faucher, for fruitful discussions about inverse scattering problems and subsurface imaging. Florian also provided us with an earlier version of \cref{fig:imaging}. Finally, we extend our gratitude to Maxence Cassier and Christophe Hazard for their insights into the asymptotics for small scatterers.

\appendix
\section{Auxiliary interior problem}\label{sec:appendix-int} Let $D_\epsilon=D_\epsilon(\bs{y}_\epsilon)$ denote the disk of radius $\lambda\epsilon$ centered at $\bs{y}_\epsilon$, as in \cref{sec:asymptotics}. Let $u_\epsilon\in H^1(D_\epsilon)$ be the solution to
\begin{align}
\left\{
\begin{array}{ll}
\Delta u_\epsilon + k^2 u_\epsilon = 0 \quad \text{in $D_\epsilon$}, \\[0.4em]
u_\epsilon = f \quad \text{on $\partial D_\epsilon$},
\end{array}
\right.
\end{align}
for some $f\in H^{1/2}(\partial D_\epsilon)$. The corresponding single layer potential reads
\begin{align}\label{eq:S_eps}
& S_\epsilon:H^{-1/2}(\partial D_\epsilon)\to H^{1/2}(\partial D_\epsilon), \\
& S_\epsilon g(\bs{x}) = \int_{\partial D_\epsilon}\phi(\bs{x},\bs{y})g(\bs{y})ds(\bs{y}), \quad \bs{x}\in \partial D_\epsilon, \nonumber
\end{align}
with $\phi(\bs{x},\bs{y})$ as in \cref{eq:pt-src}. We now write $D_\epsilon = \bs{y}_\epsilon + \lambda\epsilon\widehat{D}$, where $\widehat{D}$ is the unit disk centered at the origin, and define the following change of variables:
\begin{align}
\widehat{\bs{x}} = \frac{\bs{x}-\bs{y}_\epsilon}{\lambda\epsilon}, \quad \widehat{\bs{y}} = \frac{\bs{y}-\bs{y}_\epsilon}{\lambda\epsilon}, \quad \widehat{u}_\epsilon(\widehat{\bs{x}}) = u_\epsilon(\bs{y}_\epsilon+\lambda\epsilon\widehat{\bs{x}}).
\end{align}
Then, $\widehat{u}_\epsilon\in H^1(\widehat{D})$ satisfies
\begin{align}
\left\{
\begin{array}{ll}
\Delta \widehat{u}_\epsilon + 4\pi^2\epsilon^2 \widehat{u}_\epsilon = 0 \quad \text{in $\widehat{D}$}, \\[0.4em]
\widehat{u}_\epsilon = \widehat{f}_\epsilon \quad \text{on $\partial \widehat{D}$},
\end{array}
\right.
\end{align}
with $\widehat{f}_\epsilon(\widehat{\bs{x}}) = f(\bs{y}_\epsilon+\lambda\epsilon\widehat{\bs{x}})\in H^{1/2}(\partial\widehat{D})$. The associated single layer potential reads
\begin{align}\label{eq:Shat_eps}
& \widehat{S}_\epsilon:H^{-1/2}(\partial\widehat{D})\to H^{1/2}(\partial\widehat{D}), \\
& \widehat{S}_\epsilon\widehat{g}(\widehat{\bs{x}}) = \int_{\partial\widehat{D}}\widehat{\phi}_\epsilon(\widehat{\bs{x}},\widehat{\bs{y}})\widehat{g}(\widehat{\bs{y}})ds(\widehat{\bs{y}}), \quad \widehat{\bs{x}}\in \partial\widehat{D}, \nonumber
\end{align}
with 
\begin{align}
\widehat{\phi}_\epsilon(\widehat{\bs{x}},\widehat{\bs{y}}) = \phi(\bs{y}_\epsilon+\lambda\epsilon\widehat{\bs{x}},\bs{y}_\epsilon+\lambda\epsilon\widehat{\bs{y}}) = \frac{i}{4}H_0^{(1)}(2\pi\epsilon\vert\widehat{\bs{x}}-\widehat{\bs{y}}\vert).
\end{align}

On $\partial\widehat{D}$, we can expand any $\widehat{f}\in H^{1/2}(\partial\widehat{D})$ as a Fourier series
\begin{align}
\widehat{f}(\theta) = \frac{1}{\sqrt{2\pi}}\sum_{n=-\infty}^{+\infty}c_n e^{in\theta}, \quad \theta\in[0,2\pi],
\end{align}
with Fourier coefficients
\begin{align}
c_n = \frac{1}{2\pi}\int_0^{2\pi}\widehat{f}(\theta)e^{-in\theta}d\theta,
\end{align}
and define the $1/2$-norm as
\begin{align}
\Vert\widehat{f}\Vert_{H^{1/2}(\partial\widehat{D})}^2 = \sum_{n=-\infty}^{+\infty}\vert c_n\vert^2(1 + n^2)^{1/2}.
\end{align}
Similarly, we can expand any $\widehat{g}\in H^{-1/2}(\partial\widehat{D})$ as a Fourier series
\begin{align}
\widehat{g}(\theta) = \frac{1}{\sqrt{2\pi}}\sum_{n=-\infty}^{+\infty}d_n e^{in\theta}, \quad \theta\in[0,2\pi],
\end{align}
and define the negative $1/2$-norm as
\begin{align}
\Vert\widehat{g}\Vert_{H^{-1/2}(\partial\widehat{D})}^2 = \sum_{n=-\infty}^{+\infty}\vert d_n\vert^2(1 + n^2)^{-1/2}.
\end{align}
Finally, we define the duality pairing between $H^{1/2}(\partial\widehat{D})$ and $H^{-1/2}(\partial\widehat{D})$ via
\begin{align}
\langle\widehat{f},\widehat{g}\rangle = \int_{\partial\widehat{D}}\widehat{f}(\widehat{\bs{y}})\overline{\widehat{g}(\widehat{\bs{y}})}ds(\widehat{\bs{y}}) = \sum_{n=-\infty}^\infty c_nd_n,
\end{align}
for any $\widehat{f}\in H^{1/2}(\partial\widehat{D})$ and $\widehat{g}\in L^2(\partial\widehat{D})$. We extend it to $\widehat{g}\in H^{-1/2}(\partial\widehat{D})$ by density.

On $\partial D_\epsilon$, we define the $1/2$-norm of any $f\in H^{1/2}(\partial D_\epsilon)$ via
\begin{align}
\Vert f\Vert_{H^{1/2}(\partial D_\epsilon)} = \Vert\widehat{f}_\epsilon\Vert_{H^{1/2}(\partial\widehat{D})},
\end{align}
where $\widehat{f}_\epsilon(\widehat{\bs{x}}) = f(\bs{y}_\epsilon+\lambda\epsilon\widehat{\bs{x}})\in H^{1/2}(\partial\widehat{D})$. The duality pairing between the spaces $H^{1/2}(\partial D_\epsilon)$ and $H^{-1/2}(\partial D_\epsilon)$ is then defined via
\begin{align}
\langle f,g\rangle = \int_{\partial D_\epsilon}f(\bs{y})\overline{g(\bs{y})}ds(\bs{y}),
\end{align}
for $f\in H^{1/2}(\partial D_\epsilon)$ and $g\in L^2(\partial D_\epsilon)$. We extend it to $g\in H^{-1/2}(\partial D_\epsilon)$ by density. This allows us to finally define
\begin{align}
\Vert g\Vert_{H^{-1/2}(\partial D_\epsilon)} = \sup_{f\in H^{1/2}(\partial D_\epsilon)\neq0}\frac{\langle f,g\rangle}{\Vert f\Vert_{H^{1/2}(\partial D_\epsilon)}}.
\end{align}
From the definitions above, we note that
\begin{align}
\langle f,g\rangle = \int_{\partial D_\epsilon}f(\bs{y})\overline{g(\bs{y})}ds(\bs{y}) = \lambda\epsilon\int_{\partial\widehat{D}}\widehat{f}_\epsilon(\widehat{\bs{y}})\overline{\widehat{g}_\epsilon(\widehat{\bs{y}})}ds(\widehat{\bs{y}}) = \lambda\epsilon\langle\widehat{f}_\epsilon,\widehat{g}_\epsilon\rangle,
\end{align}
where $\widehat{f}_\epsilon(\widehat{\bs{x}}) = f(\bs{y}_\epsilon+\lambda\epsilon\widehat{\bs{x}})$ and $\widehat{g}_\epsilon(\widehat{\bs{x}}) = g(\bs{y}_\epsilon+\lambda\epsilon\widehat{\bs{x}})$, and hence
\begin{align}
\Vert g\Vert_{H^{-1/2}(\partial D_\epsilon)} = \lambda\epsilon\Vert\widehat{g}_\epsilon\Vert_{H^{-1/2}(\partial\widehat{D})}.
\end{align}

Going back to the potentials, a direct calculation yields, for $g\in H^{-1/2}(\partial D_\epsilon)$,
\begin{align}
\widehat{S_\epsilon g}(\widehat{\bs{x}}) = S_\epsilon g(\bs{y}_\epsilon+\lambda\epsilon\widehat{\bs{x}}) 
& = \int_{\partial D_\epsilon}\phi(\bs{y}_\epsilon+\lambda\epsilon\widehat{\bs{x}},\bs{y})g(\bs{y})ds(\bs{y}), \nonumber \\
& = \lambda\epsilon\int_{\partial\widehat{D}}\phi(\bs{y}_\epsilon+\lambda\epsilon\widehat{\bs{x}},\bs{y}_\epsilon+\lambda\epsilon\widehat{\bs{y}})\widehat{g}_\epsilon(\widehat{\bs{y}})ds(\widehat{\bs{y}}), \\
& = \lambda\epsilon\widehat{S}_\epsilon\widehat{g}_\epsilon(\widehat{\bs{x}}). \nonumber
\end{align}
Finally, suppose that $g_\epsilon=S_\epsilon^{-1}f\in H^{-1/2}(\partial D_\epsilon)$ for some $f\in H^{1/2}(\partial D_\epsilon)$, then
\begin{align}
S_\epsilon g_\epsilon = f \Rightarrow \widehat{S_\epsilon g_\epsilon} = \widehat{f}_\epsilon
\Rightarrow \lambda\epsilon\widehat{S}_\epsilon\widehat{g}_\epsilon = \widehat{f}_\epsilon
\Rightarrow \lambda\epsilon\widehat{g}_\epsilon = \widehat{S}_\epsilon^{-1}\widehat{f}_\epsilon
\Rightarrow \lambda\epsilon\widehat{S_\epsilon^{-1}f} = \widehat{S}_\epsilon^{-1}\widehat{f}_\epsilon.
\end{align}
We summarize all the previous definitions and observations in the following lemma.

\begin{lemma}\label{lem:appendix-int}
For any function $f\in H^{1/2}(\partial D_\epsilon)$ and $g\in H^{-1/2}(\partial D_\epsilon)$, define
\begin{align}
\widehat{f}_\epsilon(\widehat{\bs{x}}) = f(\bs{y}_\epsilon+\lambda\epsilon\widehat{\bs{x}})\in H^{1/2}(\partial\widehat{D}) \quad \text{and} \quad \widehat{g}_\epsilon(\widehat{\bs{x}}) = g(\bs{y}_\epsilon+\lambda\epsilon\widehat{\bs{x}})\in H^{-1/2}(\partial\widehat{D}).
\end{align}
Then
\begin{align}
\Vert f\Vert_{H^{1/2}(\partial D_\epsilon)} = \Vert\widehat{f}_\epsilon\Vert_{H^{1/2}(\partial\widehat{D})} \quad \text{and} \quad \Vert g\Vert_{H^{-1/2}(\partial D_\epsilon)} = \lambda\epsilon\Vert\widehat{g}_\epsilon\Vert_{H^{-1/2}(\partial\widehat{D})},
\end{align}
as well as
\begin{align}
\langle f,g\rangle = \lambda\epsilon\langle\widehat{f}_\epsilon,\widehat{g}_\epsilon\rangle, \quad \widehat{S_\epsilon g} = \lambda\epsilon\widehat{S}_\epsilon\widehat{g}_\epsilon, \quad \text{and} \quad \lambda\epsilon\widehat{S_\epsilon^{-1}f} = \widehat{S}_\epsilon^{-1}\widehat{f}_\epsilon.
\end{align}
\end{lemma}

We now prove a result on the Fourier coefficients of functions defined on $\partial D_\epsilon$.

\begin{lemma}\label{lem:appendix-fourier}
Let $f\in H^{1/2}(\partial D_\epsilon)$. Suppose that $f$ can be smoothly extended to a twice continuously differentiable function in a neighborhood $U_\epsilon$ of $D_\epsilon$ and that there exists $r\geq0$ and a constant $C>0$, independent of $\epsilon$, such that for small enough $\epsilon$,
\begin{align}
\sup_{\bs{x}\in U_\epsilon}\left\vert\frac{\partial^{i+j} f}{\partial x^i\partial y^j}(\bs{x})\right\vert \leq C\epsilon^r,
\end{align}
for all integers $i,j\geq0$ such that $i+j\leq 2$. Then, the Fourier coefficients of $\widehat{f}_\epsilon$ verify, for small enough $\epsilon$,
\begin{align}
c_0(\epsilon) = f(\bs{y}_\epsilon) + \OO(\epsilon^{r+2}) \quad \text{and} \quad \vert c_n(\epsilon)\vert \leq C\frac{\epsilon^{r+1}}{\vert n\vert^2} \quad \forall n\neq 0,
\end{align}
where $C>0$ is a constant independent of $n$ and $\epsilon$.
\end{lemma}
\begin{proof}
Since $f$ is twice continuously differentiable in $U_\epsilon$ and its second derivatives are uniformly bounded by a quantity $\OO(\epsilon^r)$, we have that
\begin{align}
f(\bs{y}_\epsilon + \lambda\epsilon e^{i\theta}) = f(\bs{y}_\epsilon) + \lambda\epsilon\cos\theta\frac{\partial f}{\partial x}(\bs{y}_\epsilon) + \lambda\epsilon\sin\theta\frac{\partial f}{\partial x}(\bs{y}_\epsilon) + \OO(\epsilon^{r+2}),
\end{align}
and hence
\begin{align}
c_0(\epsilon) = \frac{1}{2\pi}\int_0^{2\pi}f(\bs{y}_\epsilon + \lambda\epsilon e^{i\theta})d\theta = f(\bs{y}_\epsilon) + \OO(\epsilon^{r+2}).
\end{align}
Moreover, because $\theta\mapsto\widehat{f}_\epsilon(\theta)=f(\bs{y}_\epsilon + \lambda\epsilon e^{i\theta})$ is twice continuously differentiable, its first derivative has bounded variation 
\begin{align}
V_\epsilon = \int_0^{2\pi}\left\vert\frac{d^2}{d\theta^2}\widehat{f}_\epsilon(\theta)\right\vert d\theta,
\end{align}
and its Fourier coefficients satisfy \cite[Thm.~4.1]{montanelli2015b}
\begin{align}
\vert c_n(\epsilon)\vert \leq \frac{V_\epsilon}{2\pi\vert n\vert^2} \quad \forall n\neq 0.
\end{align}
It suffices to estimate $V_\epsilon$ to conclude. Let
\begin{align}
x_\epsilon(\theta) = \lambda\epsilon^{-p}\cos\theta_y + \lambda\epsilon\cos\theta, \quad y_\epsilon(\theta) = \lambda\epsilon^{-p}\sin\theta_y + \lambda\epsilon\sin\theta, \quad \theta\in[0,2\pi].
\end{align}
Since $\widehat{f}_\epsilon(\theta) = f(x_\epsilon(\theta), y_\epsilon(\theta))$, we have that
\begin{align}
\frac{d}{d\theta}\widehat{f}_\epsilon(\theta) = \lambda\epsilon\left[-\sin\theta\frac{\partial f}{\partial x}(x_\epsilon(\theta), y_\epsilon(\theta)) + \cos\theta\frac{\partial f}{\partial y}(x_\epsilon(\theta), y_\epsilon(\theta))\right]
\end{align}
and
\begin{align}
\frac{d^2}{d\theta^2}\widehat{f}_\epsilon(\theta) = \lambda\epsilon\left[-\cos\theta\frac{\partial f}{\partial x}(x_\epsilon(\theta), y_\epsilon(\theta)) - \sin\theta\frac{\partial f}{\partial y}(x_\epsilon(\theta), y_\epsilon(\theta))\right] + \OO(\epsilon^{r+2}),
\end{align}
since the second derivatives are uniformly bounded by a quantity $\OO(\epsilon^r)$. Therefore, there exists a constant $C>0$, independent of $\epsilon$, such that for small enough $\epsilon$, $V_\epsilon \leq C \epsilon^{r+1}$.
\end{proof}

We note that if we assume that $f$ can be smoothly extended to an infinitely differentiable function and that all of its derivatives are uniformly bounded by a quantity $\OO(\epsilon^r)$, then we can show that $\vert c_n(\epsilon)\vert\leq C\epsilon^{r+1}/\vert n\vert^{1+\nu}$ for all $n\neq0$ and any integer $\nu>0$.

Let $S_\epsilon^{-1}:H^{1/2}(\partial D_\epsilon)\to H^{-1/2}(\partial D_\epsilon)$ and $\widehat{S}_\epsilon^{-1}:H^{1/2}(\partial\widehat{D})\to H^{-1/2}(\partial\widehat{D})$ be the inverse operators of \cref{eq:S_eps} and \cref{eq:Shat_eps}.

\begin{theorem}\label{thm:appendix-int}
If $\widehat{f}$ has Fourier coefficients $c_n$, then $\widehat{S}^{-1}_\epsilon\widehat{f}$ has Fourier coefficients
\begin{align}
d_n = -\frac{2i}{\pi}\frac{c_n}{J_n(2\pi\epsilon)H_n^{(1)}(2\pi\epsilon)}, \quad n\in\Z.
\end{align}
Moreover, there exists a constant $C>0$, independent of $\epsilon$, such that for small enough $\epsilon$,
\begin{align}
\Vert\widehat{S}_\epsilon^{-1}\Vert \leq C \quad \text{and} \quad \Vert S_\epsilon^{-1}\Vert \leq C.
\end{align}
If $\widehat{f}_\epsilon(\widehat{\bs{x}})=f(\bs{y}_\epsilon+\lambda\epsilon\widehat{\bs{x}})$ for some function $f$ that verifies the assumptions of \cref{lem:appendix-fourier}, then there exists a constant $C>0$, independent of $n$ and $\epsilon$, such that for small enough $\epsilon$, 
\begin{align}
\vert d_0(\epsilon)\vert \leq C\vert\log\epsilon\vert^{-1}\vert f(\bs{y}_\epsilon)\vert + \OO(\vert\log\epsilon\vert^{-1}\epsilon^{r+2}) \quad \text{and} \quad \vert d_n(\epsilon)\vert \leq C\frac{\epsilon^{r+1}}{\vert n\vert} \quad \forall n\neq0.
\end{align}
Moreover, there exists a constant $C>0$, independent of $\epsilon$, such that for small enough $\epsilon$,
\begin{align}
\Vert S^{-1}_\epsilon f\Vert_{H^{-1/2}(\partial D_\epsilon)} \leq C\vert\log(\epsilon)\vert^{-1}\vert f(\bs{y}_\epsilon)\vert + \OO(\epsilon^{r+1})
\end{align}
\end{theorem}

\begin{proof}
Let $\widehat{u}_\epsilon\in H^1(\widehat{D})\cup H^1_{\mrm{loc}}(\R^2\setminus\overline{\widehat{D}})$ be the solution to
\begin{align}
\left\{
\begin{array}{ll}
\Delta \widehat{u}_\epsilon + 4\pi^2\epsilon^2\widehat{u}_\epsilon = 0 \quad \text{in $\R^2\setminus\partial\widehat{D}$}, \\[0.4em]
\widehat{u}_\epsilon = \widehat{f} \quad \text{on $\partial \widehat{D}$}, \\[0.4em]
\text{$\widehat{u}_\epsilon$ is radiating},
\end{array}
\right.
\end{align}
for some $\widehat{f}\in H^{1/2}(\partial \widehat{D})$ with
\begin{align}
\widehat{f}(\theta) = \frac{1}{\sqrt{2\pi}}\sum_{n=-\infty}^{+\infty}c_n e^{in\theta}, \quad \theta\in[0,2\pi].
\end{align}
We look for a solution of the form
\begin{align}
& \widehat{u}_\epsilon(r,\theta) = \sum_{n=-\infty}^{+\infty}\alpha_n(\epsilon)J_n(2\pi\epsilon r) e^{in\theta} \quad \text{in $\widehat{D}$},
\end{align}
and
\begin{align}
\widehat{u}_\epsilon(r,\theta) = \sum_{n=-\infty}^{+\infty}\beta_n(\epsilon)H_n^{(1)}(2\pi\epsilon r) e^{in\theta} \quad \text{in $\R^2\setminus\overline{\widehat{D}}$},
\end{align}
with $\widehat{u}_\epsilon=\widehat{S}_\epsilon\widehat{g}_\epsilon$. The boundary condition $\widehat{u}_\epsilon=\widehat{S}_\epsilon\widehat{g}_\epsilon=\widehat{f}$ on $\partial\widehat{D}$ yields
\begin{align}
& \widehat{u}_\epsilon(r,\theta) = \frac{1}{\sqrt{2\pi}}\sum_{n=-\infty}^{+\infty} \frac{c_n}{J_n(2\pi\epsilon)} J_n(2\pi\epsilon r) e^{in\theta} \quad \text{in $\widehat{D}$},
\end{align}
and
\begin{align}
\widehat{u}_\epsilon(r,\theta) = \frac{1}{\sqrt{2\pi}}\sum_{n=-\infty}^{+\infty} \frac{c_n}{H_n^{(1)}(2\pi\epsilon)} H_n^{(1)}(2\pi\epsilon r) e^{in\theta} \quad \text{in $\R^2\setminus\overline{\widehat{D}}$}.
\end{align}
Note that $\widehat{S}^{-1}_\epsilon:\widehat{f}\to\widehat{S}^{-1}_\epsilon\widehat{f}$ with
\begin{align}
\widehat{S}^{-1}_\epsilon\widehat{f}(\theta) = \left[\frac{\partial\widehat{u}_\epsilon}{\partial n}\right](\theta) = \sum_{n=-\infty}^{+\infty}\sqrt{2\pi}\epsilon c_n\left(\frac{J_n'(2\pi\epsilon)}{J_n(2\pi\epsilon)} - \frac{H_n^{(1)'}(2\pi\epsilon)}{H_n^{(1)}(2\pi\epsilon)}\right)e^{in\theta}, \quad \theta\in[0,2\pi],
\end{align}
where $[\partial\widehat{u}_\epsilon/\partial n]$ denotes the jump of the conormal derivative through $\partial\widehat{D}$. We note that
\begin{align}
\frac{J_n'(2\pi\epsilon)}{J_n(2\pi\epsilon)} - \frac{H_n^{(1)'}(2\pi\epsilon)}{H_n^{(1)}(2\pi\epsilon)} = -\frac{i}{\pi^2\epsilon}\frac{1}{J_n(2\pi\epsilon)H_n^{(1)}(2\pi\epsilon)} \quad \forall n,
\end{align}
using the Wronskian (see, e.g., \cite[eq.~(9.1.16)]{abramowitz1964}). This yields
\begin{align}
\widehat{S}^{-1}_\epsilon\widehat{f}(\theta) = -\frac{1}{\sqrt{2\pi}}\frac{2i}{\pi}\sum_{n=-\infty}^{+\infty}\frac{c_n}{J_n(2\pi\epsilon)H_n^{(1)}(2\pi\epsilon)}e^{in\theta}, \quad \theta\in[0,2\pi],
\end{align}
which shows that the Fourier coefficients of $\widehat{S}^{-1}_\epsilon\widehat{f}$ are given by
\begin{align}
d_n = -\frac{2i}{\pi}\frac{c_n}{J_n(2\pi\epsilon)H_n^{(1)}(2\pi\epsilon)}, \quad n\in\Z.
\end{align}
From the coefficients $d_n$, it is clear that the squared negative $1/2$-norm reads:
\begin{align}
\Vert\widehat{S}^{-1}_\epsilon\widehat{f}\Vert_{H^{-1/2}(\partial\widehat{D})}^2 = \frac{4}{\pi^2}\sum_{n=-\infty}^{+\infty}\frac{\vert c_n\vert^2}{\vert J_n(2\pi\epsilon)H_n^{(1)}(2\pi\epsilon)\vert^2}(1 + n^2)^{-1/2}.
\end{align}
We choose an integer $N$ such that the uniform estimates of \cref{lem:appendix-large-n} hold and split the norm as follows, 
\begin{align}
\Vert\widehat{S}^{-1}_\epsilon\widehat{f}\Vert_{H^{-1/2}(\partial\widehat{D})}^2 = & \; \frac{4}{\pi^2}\sum_{\vert n\vert < N}\frac{\vert c_n\vert^2}{\vert J_n(2\pi\epsilon)H_n^{(1)}(2\pi\epsilon)\vert^2}(1 + n^2)^{-1/2} \\
& \; + \frac{4}{\pi^2}\sum_{\vert n\vert\geq N}\frac{\vert c_n\vert^2}{\vert J_n(2\pi\epsilon)H_n^{(1)}(2\pi\epsilon)\vert^2}(1 + n^2)^{-1/2}. \nonumber
\end{align}
We use the asymptotics for small arguments of \cref{lem:appendix-small-x} for $n=0$,
\begin{align}
\frac{1}{\vert J_0(2\pi\epsilon)H_0^{(1)}(2\pi\epsilon)\vert} \leq C \vert\log \epsilon\vert^{-1} \leq 1,
\end{align}
with a constant $C>0$ independent of $\epsilon$, as well as for $1\leq\vert n\vert<N$,
\begin{align}
\frac{1}{\vert J_n(2\pi\epsilon)H_n^{(1)}(2\pi\epsilon)\vert} \leq C n,
\end{align}
with a constant $C>0$ independent of $n$ and $\epsilon$ (it is the maximum over $1\leq\vert n\vert<N$ of the constants appearing for each $n$). For $\vert n\vert\geq N$, we use the uniform estimates of \cref{lem:appendix-large-n},
\begin{align}
\frac{1}{\vert J_n(2\pi\epsilon)H_n^{(1)}(2\pi\epsilon)\vert} \leq Cn,
\end{align}
with a constant $C>0$ independent of $n$ and $\epsilon$. Putting the pieces together yields
\begin{align}
\Vert\widehat{S}^{-1}_\epsilon\widehat{f}\Vert_{H^{-1/2}(\partial\widehat{D})}^2 \leq C\sum_{n=-\infty}^{+\infty}\vert c_n\vert^2(1 + n^2)^{1/2} = C\Vert f\Vert_{H^{1/2}(\partial\widehat{D})}^2,
\end{align}
that is, $\Vert\widehat{S}_\epsilon^{-1}\Vert\leq C$. For $S_\epsilon^{-1}$, from \cref{lem:appendix-int}, we observe that, for any $f\in H^{1/2}(\partial D_\epsilon)$,
\begin{align}
\Vert S_\epsilon^{-1} f\Vert_{H^{-1/2}(\partial D_\epsilon)} = \lambda\epsilon\Vert \widehat{S_\epsilon^{-1} f}\Vert_{H^{-1/2}(\partial\widehat{D})} = \Vert\widehat{S}_\epsilon^{-1}\widehat{f}_\epsilon\Vert_{H^{-1/2}(\partial\widehat{D})},
\end{align}
with $\widehat{f}_\epsilon(\widehat{\bs{x}})=f(\bs{y}_\epsilon+\lambda\epsilon\widehat{\bs{x}})\in H^{1/2}(\partial\widehat{D})$. This yields 
\begin{align}
\Vert S_\epsilon^{-1} f\Vert_{H^{-1/2}(\partial D_\epsilon)} \leq C\Vert\widehat{f}_\epsilon\Vert_{H^{1/2}(\partial\widehat{D})} = C\Vert f\Vert_{H^{1/2}(\partial D_\epsilon)}.
\end{align}

If $\widehat{f}_\epsilon(\widehat{\bs{x}})=f(\bs{y}_\epsilon+\lambda\epsilon\widehat{\bs{x}})$ for some function $f$ that verifies the assumptions of \cref{lem:appendix-fourier}, then there exists a constant $C>0$, independent of $n$ and $\epsilon$, such that for small enough $\epsilon$, the Fourier coefficients of $\widehat{f}_\epsilon$ satisfy
\begin{align}
c_0(\epsilon) = f(\bs{y}_\epsilon) + \OO(\epsilon^{r+2}) \quad \text{and} \quad \vert c_n(\epsilon)\vert \leq C\frac{\epsilon^{r+1}}{\vert n\vert^2} \quad \forall n\neq 0.
\end{align}
We choose, again, an integer $N$ such that the uniform estimates of \cref{lem:appendix-large-n} hold, and then combine the asymptotics for small arguments of \cref{lem:appendix-small-x} for $n=0$ and $0\neq\vert n\vert\leq N$ with the uniform estimates for $\vert n\vert\geq N$. This yields, for small enough $\epsilon$,
\begin{align}
\vert d_0(\epsilon)\vert \leq C\vert\log\epsilon\vert^{-1}\vert f(\bs{y}_\epsilon)\vert + \OO(\vert\log\epsilon\vert^{-1}\epsilon^{r+2}) \quad \text{and} \quad \vert d_n(\epsilon)\vert \leq C\frac{\epsilon^{r+1}}{\vert n\vert} \quad \forall n\neq0,
\end{align}
with a constant $C>0$ independent of $n$ and $\epsilon$. This immediately implies that
\begin{align}
\Vert S^{-1}_\epsilon f\Vert_{H^{-1/2}(\partial D_\epsilon)} = \Vert\widehat{S}^{-1}_\epsilon\widehat{f}_\epsilon\Vert_{H^{-1/2}(\partial\widehat{D})} \leq C\vert\log(\epsilon)\vert^{-1}\vert f(\bs{y}_\epsilon)\vert + \OO(\epsilon^{r+1}).
\end{align}
\end{proof}

\section{Auxiliary exterior problem}\label{sec:appendix-ext} Let $D_\epsilon=D_\epsilon(\bs{y}_\epsilon)$ and $D=D(\bs{0})$ satisfy the assumptions of \cref{sec:asymptotics}. Let $w_\epsilon\in H^1_\mrm{loc}(\R^s\setminus\{D\cup D_\epsilon\})$ be the solution to
\begin{align}\label{eq:w}
\left\{
\begin{array}{ll}
\Delta w_\epsilon + k^2 w_\epsilon = 0 \quad \text{in $\R^2\setminus\{\overline{D\cup D_\epsilon}\}$}, \\[0.4em]
w_\epsilon = f \quad \text{on $\partial D$}, \\[0.4em]
w_\epsilon = 0 \quad \text{on $\partial D_\epsilon$}, \\[0.4em]
\text{$w_\epsilon$ is radiating},
\end{array}
\right.
\end{align}
for some $f\in H^{1/2}(\partial D)$. We seek to evaluate the solution $w_\epsilon$ in the domain $B$, characterized in \cref{sec:asymptotics}, via $w_\epsilon = \mathscr{S}g_\epsilon + \mathscr{S}_\epsilon h_\epsilon$ with operators
\begin{align}\label{eq:SS}
& \mathscr{S}:H^{-1/2}(\partial D)\to H^1(B), \\
& \mathscr{S}g(\bs{x}) = \int_{\partial D}\phi(\bs{x},\bs{y})g(\bs{y})ds(\bs{y}), \quad \bs{x}\in B, \nonumber
\end{align}
and
\begin{align}\label{eq:SSe}
& \mathscr{S}_\epsilon:H^{-1/2}(\partial D_\epsilon)\to H^1(B), \\
& \mathscr{S}_\epsilon g(\bs{x}) = \int_{\partial D_\epsilon}\phi(\bs{x},\bs{y})g(\bs{y})ds(\bs{y}), \quad \bs{x}\in B. \nonumber
\end{align}
To solve \cref{eq:w}, we must have
\begin{align}
Sg_\epsilon + T_\epsilon h_\epsilon = f \quad \text{on $\partial D$} \quad \text{and} \quad S_\epsilon h_\epsilon +T_\epsilon^Tg_\epsilon = 0 \quad \text{on $\partial D_\epsilon$},
\end{align}
with operators
\begin{align}\label{eq:S}
& S:H^{-1/2}(\partial D)\to H^{1/2}(\partial D), \\
& Sg(\bs{x}) = \int_{\partial D}\phi(\bs{x},\bs{y})g(\bs{y})ds(\bs{y}), \quad \bs{x}\in \partial D, \nonumber
\end{align}
\begin{align}\label{eq:Se}
& S_\epsilon:H^{-1/2}(\partial D_\epsilon)\to H^{1/2}(\partial D_\epsilon), \\
& S_\epsilon g(\bs{x}) = \int_{\partial D_\epsilon}\phi(\bs{x},\bs{y})g(\bs{y})ds(\bs{y}), \quad \bs{x}\in \partial D_\epsilon, \nonumber
\end{align}
\begin{align}\label{eq:Te}
& T_\epsilon:H^{-1/2}(\partial D_\epsilon)\to H^{1/2}(\partial D), \\
& T_\epsilon g(\bs{x}) = \int_{\partial D_\epsilon}g(\bs{x},\bs{y})g(\bs{y})ds(\bs{y}), \quad \bs{x}\in \partial D, \nonumber
\end{align}
and
\begin{align}\label{eq:Tet}
& T_\epsilon^T:H^{-1/2}(\partial D)\to H^{1/2}(\partial D_\epsilon), \\
& T_\epsilon^Tg(\bs{x}) = \int_{\partial D}\phi(\bs{x},\bs{y})g(\bs{y})ds(\bs{y}), \quad \bs{x}\in \partial D_\epsilon. \nonumber
\end{align}
This yields
\begin{align}\label{eq:psi}
g_\epsilon = (I - S^{-1}T_\epsilon S_\epsilon^{-1}T_\epsilon^T)^{-1}S^{-1}f \quad \text{on $\partial D$} \quad \text{and} \quad h_\epsilon=-S_\epsilon^{-1}T_\epsilon^T g_\epsilon \quad \text{on $\partial D_\epsilon$}.
\end{align}
Both $S^{-1}$ and $S_\epsilon^{-1}$ are continuous operators (see \cref{thm:appendix-int} for the latter). Let
\begin{align}\label{eq:Ae}
A_\epsilon=S^{-1}T_\epsilon S_\epsilon^{-1}T_\epsilon^T:H^{-1/2}(\partial D)\to H^{-1/2}(\partial D).
\end{align} 
We characterize, below, the norms of $T_\epsilon$, $T_\epsilon^T$, $A_\epsilon$, $\mathscr{S}$, and $\mathscr{S}_\epsilon$.

\begin{lemma}\label{lem:appendix-ext}
There exists a constant $C>0$, independent of $\epsilon$, such that for small enough $\epsilon$,
\begin{align}
\Vert T_\epsilon\Vert \leq C \epsilon^{p/2}, \;\; \Vert T_\epsilon^T\Vert \leq C \epsilon^{p/2}, \;\; \Vert A_\epsilon\Vert \leq C \epsilon^{p}, \;\; \Vert\mathscr{S}\Vert \leq C, \;\; \text{and} \;\; \Vert \mathscr{S}_\epsilon\Vert \leq C \epsilon^{p/2}.
\end{align}
If $g_\epsilon=S_\epsilon^{-1}f$ for some function $f$ that verifies the assumptions of \cref{lem:appendix-fourier}, then there exists a constant $C>0$, independent of $\epsilon$, such that for small enough $\epsilon$,
\begin{align}
\Vert T_\epsilon g_\epsilon\Vert_{H^{1/2}(\partial D)} \leq C\vert\log\epsilon\vert^{-1}\epsilon^{p/2}\vert f(\bs{y}_\epsilon)\vert + \OO(\epsilon^{p/2}\epsilon^r\epsilon^2)
\end{align}
and
\begin{align}
\Vert\mathscr{S}_\epsilon g_\epsilon\Vert_{H^1(B)} \leq C\vert\log\epsilon\vert^{-1}\epsilon^{p/2}\vert f(\bs{y}_\epsilon)\vert + \OO(\epsilon^{p/2}\epsilon^r\epsilon^2).
\end{align}
\end{lemma}

\begin{proof}
To show that $\Vert T_\epsilon\Vert$ is small, we observe that, via \cref{lem:appendix-large-x}, for $\bs{x}\in\partial D$,
\begin{align}
\vert T_\epsilon g(\bs{x})\vert \leq \sup_{\bs{y}\in\partial D_\epsilon}\vert\phi(\bs{x},\bs{y})\vert\int_{\partial D_\epsilon}\vert g(\bs{y})\vert ds(\bs{y}) \leq C\epsilon^{p/2}\Vert g\Vert_{H^{-1/2}(\partial D_\epsilon)},
\end{align}
which leads to $\Vert T_\epsilon g\Vert_{L^2(\partial D)} \leq C \epsilon^{p/2}$. Similarly, we can write, for $\bs{x}\in\partial D$,
\begin{align}
\vert\nabla(T_\epsilon g)(\bs{x})\vert \leq \sup_{\bs{y}\in\partial D_\epsilon}\vert \nabla_{\bs{x}} \phi(\bs{x},\bs{y})\vert\int_{\partial D_\epsilon}\vert g(\bs{y})\vert ds(\bs{y}) \leq C\epsilon^{p/2}\Vert g\Vert_{H^{-1/2}(\partial D_\epsilon)},
\end{align}
which leads to $\Vert\nabla(T_\epsilon g)\Vert_{L^2(\partial D)} \leq C \epsilon^{p/2}$. This yields, since $T_\epsilon g\in H^s(\partial D)$ for any $s\geq0$,
\begin{align}
\Vert T_\epsilon g\Vert_{H^{1/2}(\partial D)}^2 \leq C \Vert T_\epsilon g\Vert_{H^1(\partial D)}^2 = C \{\Vert T_\epsilon g\Vert_{L^2(\partial D)}^2 + \Vert\nabla(T_\epsilon g)\Vert_{L^2(\partial D)}^2\} \leq C\epsilon^p.
\end{align}
The same reasoning can be applied to $T_\epsilon^T$. This immediately gives a bound on $\Vert A_\epsilon\Vert$.

For $\mathscr{S}$ and $\mathscr{S}_\epsilon$, we obtain the following estimates for $x\in B$,
\begin{align}
\vert \mathscr{S}g(\bs{x})\vert \leq \sup_{\bs{y}\in\partial D}\vert \phi(\bs{x},\bs{y})\vert\int_{\partial D}\vert g(\bs{y})\vert ds(\bs{y}) \leq C\Vert g\Vert_{H^{-1/2}(\partial D)},
\end{align}
and
\begin{align}
\vert \mathscr{S}_\epsilon g(\bs{x})\vert \leq \sup_{\bs{y}\in\partial D_\epsilon}\vert \phi(\bs{x},\bs{y})\vert\int_{\partial D_\epsilon}\vert g(\bs{y})\vert ds(\bs{y}) \leq C\epsilon^{p/2}\Vert g\Vert_{H^{-1/2}(\partial D_\epsilon)},
\end{align}
and similar ones for their gradients.

Let $g_\epsilon=S_\epsilon^{-1}f$ for some function $f$ that verifies the assumptions of \cref{lem:appendix-fourier} and, for $\bs{x}\in \partial D$, let us write $T_\epsilon g_\epsilon(\bs{x})$ as a duality pairing
\begin{align}
T_\epsilon g_\epsilon(\bs{x})=\langle \phi(\bs{x},\cdot),S_\epsilon^{-1}f\rangle
\end{align}
between $\phi(\bs{x},\cdot)\in H^{1/2}(\partial D_\epsilon)$ and $S_\epsilon^{-1}f\in H^{-1/2}(\partial D_\epsilon)$. Then, via \cref{lem:appendix-int},
\begin{align}
T_\epsilon g_\epsilon(\bs{x}) = \lambda\epsilon\langle\widehat{\phi}_\epsilon(\bs{x},\cdot),\widehat{S_\epsilon^{-1}f}\rangle = \langle\widehat{\phi}_\epsilon(\bs{x},\cdot),\widehat{S}_\epsilon^{-1}\widehat{f}_\epsilon\rangle = \sum_{n=-\infty}^\infty c_n\overline{d_n},
\end{align}
where $c_n$ and $d_n$ are the Fourier coefficients of $\widehat{\phi}_\epsilon(\bs{x},\widehat{\bs{y}})=\phi(\bs{x},\bs{y}_\epsilon+\lambda\epsilon\widehat{\bs{y}})$ and $\widehat{S}_\epsilon^{-1}\widehat{f}_\epsilon$. Using \cref{lem:appendix-fourier} for $c_n$ and \cref{thm:appendix-int} and $d_n$, there exists a constant $C>0$, independent of $n$ and $\epsilon$, such that for small enough $\epsilon$,
\begin{align}
& c_0(\epsilon) = \phi(\bs{x},\bs{y}_\epsilon) + \OO(\epsilon^{p/2+2}) = \OO(\epsilon^{p/2}) \quad \text{and} \quad \vert c_n(\epsilon)\vert \leq C\frac{\epsilon^{p/2+1}}{\vert n\vert^2} \quad \forall n\neq 0, \\
& \vert d_0(\epsilon)\vert \leq C\vert\log\epsilon\vert^{-1}\vert f(\bs{y}_\epsilon)\vert + \OO(\vert\log\epsilon\vert^{-1}\epsilon^{r+2}) \quad \text{and} \quad \vert d_n(\epsilon)\vert \leq C\frac{\epsilon^{r+1}}{\vert n\vert} \quad \forall n\neq0. \nonumber
\end{align}
Therefore, we have
\begin{align}
\vert T_\epsilon g_\epsilon(\bs{x})\vert \leq C\vert\log\epsilon\vert^{-1}\epsilon^{p/2}\vert f(\bs{y}_\epsilon)\vert + \OO(\epsilon^{p/2}\epsilon^r\epsilon^2).
\end{align}
We have the same estimate for the gradient and hence
\begin{align}
\Vert T_\epsilon g_\epsilon\Vert_{H^{1/2}(\partial D)} \leq C\vert\log\epsilon\vert^{-1}\epsilon^{p/2}\vert f(\bs{y}_\epsilon)\vert + \OO(\epsilon^{p/2}\epsilon^r\epsilon^2).
\end{align}
The proof for $\Vert\mathscr{S}_\epsilon g_\epsilon\Vert_{H^1(B)}$ is (almost) identical.
\end{proof}

\begin{theorem}\label{thm:appendix-ext}
There exists a constant $C>0$, independent of $\epsilon$, such that for small enough $\epsilon$,
\begin{align}
\Vert w_\epsilon\Vert_{H^1(B)} \leq C \Vert f\Vert_{H^{1/2}(\partial D)}.
\end{align}
\end{theorem}

\begin{proof}
We first write
\begin{align}
\Vert w_\epsilon\Vert_{H^1(B)} \leq \Vert\mathscr{S}g_\epsilon\Vert_{H^1(B)} + \Vert\mathscr{S}_\epsilon h_\epsilon\Vert_{H^1(B)},
\end{align}
and use \cref{lem:appendix-ext} to get
\begin{align}
\Vert w_\epsilon\Vert_{H^1(B)} \leq C\Vert g_\epsilon\Vert_{H^{-1/2}(\partial D)} + C\epsilon^{p/2}\Vert h_\epsilon\Vert_{H^{-1/2}(\partial D_\epsilon)},
\end{align}
with a constant $C>0$ independent of $\epsilon$. From \cref{eq:psi} and \cref{lem:appendix-ext}, we observe that
\begin{align}
\Vert g_\epsilon\Vert_{H^{-1/2}(\partial D)} \leq C\Vert(I - A_\epsilon)^{-1}\Vert\,\Vert S^{-1}\Vert\,\Vert f\Vert_{H^{1/2}(\partial D)} \leq C \Vert f\Vert_{H^{1/2}(\partial D)},
\end{align}
since $\Vert A_\epsilon\Vert\leq C\epsilon^{p/2}\leq1/2$ and $\Vert (I-A_\epsilon)^{-1}\Vert\leq(1-\Vert A_\epsilon\Vert)^{-1}\leq2$ for small enough $\epsilon$, and
\begin{align}
\Vert h_\epsilon\Vert_{H^{-1/2}(\partial D_\epsilon)} \leq C \Vert S_\epsilon^{-1}\Vert \,\Vert T_\epsilon^T\Vert\,\Vert g_\epsilon\Vert_{H^{-1/2}(\partial D)} \leq C\epsilon^{p/2}\Vert f\Vert_{H^{1/2}(\partial D)}.
\end{align}
This yields
\begin{align}
\Vert w_\epsilon\Vert_{H^1(B)} & \leq C\Vert f\Vert_{H^{1/2}(\partial D)} + C\epsilon^p\Vert f\Vert_{H^{1/2}(\partial D)} \leq C\Vert f\Vert_{H^{1/2}(\partial D)}.
\end{align}
\end{proof}

\section{Asymptotic formulas}\label{sec:appendix-asymptotics}

Here, we revisit several helpful asymptotic formulas. Due to the symmetry relations that follow, our focus can be solely on positive integers $n\in\N$.

\begin{lemma}[Symmetry relations]\label{lem:appendix-sym}
For all $n\in\N$,
\begin{align}
J_{-n}(x) = (-1)^nJ_n(x) \quad \forall x\in\R \quad \text{and} \quad H_{-n}^{(1)}(x) = (-1)^nH_n^{(1)}(x) \quad \forall x\in(0,+\infty).
\end{align}
\end{lemma}

\begin{lemma}[Asymptotics for small argument]\label{lem:appendix-small-x} 
When $n\in\N$ is fixed and $x\to0^+$:
\begin{align}
& J_n(x) = \frac{1}{n!}\left(\frac{x}{2}\right)^n + \OO(x^{n+1}) \quad \forall n\neq 0, \quad J_0(x) = 1 + \OO(x^2), \label{eq:small-x-Jn} \\
& H_n(x) = \frac{(n-1)!}{i\pi}\left(\frac{2}{x}\right)^{n} + \OO\left(\frac{1}{x^{n-1}}\right) \quad \forall n\neq 0, \quad H_0(x) = \frac{2i}{\pi}\log\left(\frac{x}{2}\right) + \OO(1). \label{eq:small-x-Hn}
\end{align}
\end{lemma}

\begin{proof}
See, e.g., \cite[eqs.~(9.1.7)--(9.1.9)]{abramowitz1964}.
\end{proof}

Let us emphasize that the estimates in \cref{lem:appendix-small-x} are not uniform in $n$, that is, the constants in the $\OO$ are independent of $x\to0^+$ but do depend on $n$.

\begin{lemma}[Asymptotics for large argument]\label{lem:appendix-large-x}  
When $n\in\N$ is fixed and $x\to+\infty$:
\begin{align}
& J_n(x) = \sqrt{\frac{2}{\pi x}}\cos(x - n\pi/2 - \pi/4)\left(1 + \OO\left(\frac{1}{x}\right)\right), \label{eq:large-x-Jn} \\
& H_n^{(1)}(x) = \sqrt{\frac{2}{\pi x}}e^{i(x - n\pi/2 - \pi/4)}\left(1 + \OO\left(\frac{1}{x}\right)\right). \label{eq:large-x-Hn}
\end{align}
\end{lemma}

\begin{proof}
See, e.g., \cite[eqs.~(9.2.1)--(9.2.3)]{abramowitz1964}.
\end{proof}

The estimates in \cref{lem:appendix-large-x} are not uniform in $n$ either.

\begin{lemma}[Asymptotics for large order]\label{lem:appendix-large-n}  
When $n\in\N\to\infty$:
\begin{align}\label{eq:large-n-Jn}
J_n(x) = \frac{1}{n!}\left(\frac{x}{2}\right)^{n}\left(1 + \OO\left(\frac{1}{n}\right)\right), 
\end{align}
uniformly on compact subsets of $[0,+\infty)$,
\begin{align}\label{eq:large-n-Hn}
H_n^{(1)}(x) = \frac{(n-1)!}{i\pi}\left(\frac{2}{x}\right)^{n}\left(1 + \OO\left(\frac{1}{n}\right)\right), 
\end{align}
uniformly on compact subsets of $(0,+\infty)$, and
\begin{align}\label{eq:large-n-JnHn}
J_n(x)H_n^{(1)}(x) = \frac{n}{i\pi}\left(1 + \OO\left(\frac{1}{n}\right)\right), 
\end{align}
uniformly on compact subsets of $[0,+\infty)$. 
\end{lemma}

\begin{proof}
It can be obtained from the series representation of the functions $J_n$ and $H_n^{(1)}$; see, e.g., \cite[eqs.~(3.97)--(3.98)]{colton2019}.
\end{proof}

We note that the estimates in \cref{lem:appendix-large-n} are uniform on compact subsets of $[0,+\infty)$ for \cref{eq:large-n-Jn} and \cref{eq:large-n-JnHn}, and $(0,+\infty)$ for \cref{eq:large-n-Hn}, i.e., the constants in the $\OO$ are independent of both $n\to+\infty$ and $x$ in that subset.

For the remaining lemmas, we will assume that all variables are defined as in \cref{sec:asymptotics} and \cref{sec:mod-HK-identity}. In particular,
\begin{align}
\bs{y}_\epsilon = \lambda\epsilon^{-p}e^{i\theta_y}, \quad \bs{z}_\epsilon = \lambda\epsilon^{-q}e^{i\theta_z}, \quad 0<p<q.
\end{align}

\begin{lemma}\label{lem:appendix-dist}  
When $\epsilon\to0^+$:
\begin{align}
\vert\bs{y}_\epsilon-\bs{z}_\epsilon\vert = \lambda\epsilon^{-q}\left(1 - \epsilon^{q-p}\cos(\theta_y - \theta_z) + \OO\left(\epsilon^{2(q-p)}\right)\right).
\end{align}
For any $\bs{x}=\lambda c_xe^{i\theta_x}$ with $c_x=\OO(1)$, when $\epsilon\to0^+$:
\begin{align}
\vert\bs{x}-\bs{y}_\epsilon\vert = \lambda\epsilon^{-p}\left(1 - c_x\epsilon^{p}\cos(\theta_x - \theta_y) + \OO\left(\epsilon^{2p}\right)\right).
\end{align}
\end{lemma}

\begin{proof} 
We write
\begin{align}
\vert\bs{y}_\epsilon-\bs{z}_\epsilon\vert = \lambda\epsilon^{-q}\left(1 - 2\epsilon^{q-p}\cos(\theta_y - \theta_z) + \epsilon^{2(q-p)}\right)^{1/2}.
\end{align}
Then, for small enough $\epsilon$,
\begin{align}
\left(1 - 2\epsilon^{q-p}\cos(\theta_y - \theta_z) + \epsilon^{2(q-p)}\right)^{1/2} = 1 - \epsilon^{q-p}\cos(\theta_y - \theta_z) + \OO\left(\epsilon^{2(q-p)}\right).
\end{align}
We conclude by multiplying by $\lambda\epsilon^{-q}$. We obtain the result for $\vert\bs{x}-\bs{y}_\epsilon\vert$ by writing
\begin{align}
\vert\bs{x}_\epsilon-\bs{y}_\epsilon\vert = \lambda\epsilon^{-p}\left(1 - 2c_x\epsilon^{p}\cos(\theta_x - \theta_y) + \epsilon^{2p}\right)^{1/2}.
\end{align}
\end{proof}

\begin{lemma}\label{lem:appendix-hank}
When $\epsilon\to0^+$:
\begin{align}
H_0^{(1)}(k\vert\bs{y}_\epsilon-\bs{z}_\epsilon\vert) = \frac{\epsilon^{q/2}}{\pi}e^{i\left(2\pi\epsilon^{-q}\left[1 - \epsilon^{q-p}\cos(\theta_y - \theta_z) + \OO\left(\epsilon^{2(q-p)}\right)\right] - \pi/4\right)}\left(1 + \OO\left(\epsilon^{q-p}\right)\right).
\end{align}
For any $\bs{x}=\lambda c_xe^{i\theta_x}$ with $c_x=\OO(1)$, when $\epsilon\to0^+$:
\end{lemma}
\begin{align}
H_0^{(1)}(k\vert\bs{x}-\bs{y}_\epsilon\vert) = \frac{\epsilon^{p/2}}{\pi}e^{i\left(2\pi\epsilon^{-p}\left[1 - c_x\epsilon^{p}\cos(\theta_x - \theta_y)\right] - \pi/4\right)}\left(1 + \OO\left(\epsilon^{p}\right)\right).
\end{align}
\begin{proof}
This is obtained by combining \cref{lem:appendix-dist} with \cref{lem:appendix-large-x}.
\end{proof}

\begin{lemma}\label{lem:appendix-integral}
For any $c_x=\OO(1)$, when $\epsilon\to0^+$:
\begin{align}
\left\langle e^{-2i\pi\left(\epsilon^{-p}\cos(\theta_y - \theta_z) + c_x\cos(\theta_x - \theta_y)\right)}\right\rangle = & \; \frac{\epsilon^{p/2}}{\pi}\cos\left(2\pi\epsilon^{-p}\left[1 + c_x\epsilon^{p}\cos(\theta_x - \theta_z)\right] - \pi/4\right) \\
& \; \times \left(1 + \OO\left(\epsilon^p\right)\right). \nonumber
\end{align}
\end{lemma}

\begin{proof}
The integral on the left-hand side equals $J_0(2\pi h)$ with
\begin{align}
h = \sqrt{(\epsilon^{-p}\cos\theta_z + c_x\cos\theta_x)^2 + (\epsilon^{-p}\sin\theta_z + c_x\sin\theta_x)^2}.
\end{align}
We then obtain an expansion for $h$ and utilize \cref{lem:appendix-large-x} for $J_0$.
\end{proof}

\begin{lemma}\label{lem:appendix-average}
For any $\bs{x}=c_x\lambda e^{i\theta_x}$ with $c_x=\OO(1)$, when $\epsilon\to0^+$:
\begin{align}
\tilde{v}^i_\epsilon(\bs{x},\bs{y}_\epsilon,\bs{z}_\epsilon) = & -\frac{\epsilon^{p/2}\epsilon^{q/2}}{4\pi^2H_0^{(1)}(2\pi\epsilon)}e^{2i\pi\left(\epsilon^{-q} + \epsilon^{-p} - \epsilon^{-p}\cos(\theta_y-\theta_z) - c_x\cos(\theta_x-\theta_y) + \OO\left(\epsilon^{q-2p}\right)\right)} \\
& \times \left(1 + \OO\left(\epsilon^{\min(p,\,q-p)}\right)\right) \nonumber
\end{align}
and
\begin{align}
\langle\tilde{v}^i_\epsilon\rangle(\bs{x},\bs{z}_\epsilon) = & -\frac{\epsilon^{p}\epsilon^{q/2}}{4\pi^3H_0^{(1)}(2\pi\epsilon)}\cos\left(2\pi\epsilon^{-p}\left[1 + c_x\epsilon^{p}\cos(\theta_x - \theta_z) \right] - \pi/4\right) \\
& \times e^{2i\pi\left(\epsilon^{-q} + \epsilon^{-p} + \OO\left(\epsilon^{q-2p}\right)\right)}\left(1 + \OO\left(\epsilon^{\min(p,\,q-p)}\right)\right). \nonumber
\end{align}
\end{lemma}

\begin{proof}
We recall that
\begin{align}
\tilde{v}^i_\epsilon(\bs{x},\bs{y}_\epsilon,\bs{z}_\epsilon) = -\frac{i}{4}\frac{H_0^{(1)}(k\vert\bs{y}_\epsilon-\bs{z}_\epsilon\vert)}{H_0^{(1)}(2\pi\epsilon)}H_0^{(1)}(k\vert\bs{x}-\bs{z}_\epsilon\vert).
\end{align}
The estimates are direct consequences of \cref{lem:appendix-hank} and \cref{lem:appendix-integral}.
\end{proof}

\begin{lemma}\label{lem:appendix-scal}
When $\epsilon\to0^+$:
\begin{align}
\vert\mu_\epsilon(\bs{y}_\epsilon,\bs{z}_\epsilon)^{-1}\vert^2 = \pi^2\vert H_0^{(1)}(2\pi\epsilon)\vert^2\epsilon^{-q}\left(1 + \OO\left(\epsilon^{q-p}\right)\right).
\end{align}
\end{lemma}

\begin{proof}
We recall that $\mu_\epsilon(\bs{y}_\epsilon,\bs{z}_\epsilon) = -H_0^{(1)}(k\vert\bs{y}_\epsilon-\bs{z}_\epsilon\vert)/H_0^{(1)}(2\pi\epsilon)$. Therefore,
\begin{align}
\vert\mu_\epsilon(\bs{y}_\epsilon,\bs{z}_\epsilon)^{-1}\vert^2 = \vert H_0^{(1)}(2\pi\epsilon)\vert^2 \vert H_0^{(1)}(k\vert\bs{y}_\epsilon-\bs{z}_\epsilon\vert)^{-1}\vert^2.
\end{align}
We use \cref{lem:appendix-hank} to get an expansion for $H_0^{(1)}(k\vert\bs{y}_\epsilon-\bs{z}_\epsilon\vert)^{-1}$ and its complex conjugate, and utilize the fact that $\vert z\vert^2=z\overline{z}$. 
\end{proof}

\bibliographystyle{siamplain}
\bibliography{/Users/montanelli/Dropbox/HM/WORK/ACADEMIA/BIBLIOGRAPHY/references.bib}

\begin{thebibliography}{10}

\bibitem{abramowitz1964}
{\sc M.~Abramowitz and I.~A. Stegun}, {\em Handbook of Mathematical Functions
  with Formulas, Graphs, and Mathematical Tables}, Dover Publications, 1964.

\bibitem{alouges2018}
{\sc F.~Alouges and M.~Aussal}, {\em {FEM and BEM simulations with the Gypsilab
  framework}}, SMAI J. Comput. Math., 4 (2018), pp.~297--318.

\bibitem{audibert2014}
{\sc L.~Audibert and H.~Haddar}, {\em A generalized formulation of the linear
  sampling method with exact characterization of targets in terms of farfield
  measurements}, Inverse Probl., 30 (2014), p.~035011.

\bibitem{audibert2017}
{\sc L.~Audibert and H.~Haddar}, {\em The generalized linear sampling method
  for limited aperture measurements}, SIAM J. Imaging Sci., 10 (2017),
  pp.~845--870.

\bibitem{austin2017b}
{\sc A.~P. Austin and L.~N. Trefethen}, {\em Trigonometric interpolation and
  quadrature in perturbed points}, SIAM J. Numer. Anal., 55 (2017),
  pp.~2113--2122.

\bibitem{bao2007}
{\sc G.~Bao, S.~Hou, and P.~Li}, {\em Inverse scattering by a continuation
  method with initial guesses from a direct imaging algorithm}, J. Comput.
  Phys., 227 (2007), pp.~755--762.

\bibitem{cakoni2014}
{\sc F.~Cakoni and D.~Colton}, {\em A Qualitative Approach to Inverse
  Scattering Theory}, Applied Mathematical Sciences, Springer, New York, 2014.

\bibitem{cakoni2016}
{\sc F.~Cakoni, D.~Colton, and H.~Haddar}, {\em Inverse Scattering Theory and
  Transmission Eigenvalues}, CBMS-NSF Regional Conference Series on
  Mathematics, SIAM, Philadelphia, 2016.

\bibitem{cassier2013}
{\sc M.~Cassier and C.~Hazard}, {\em Multiple scattering of acoustic waves by
  small sound-soft obstacles in two dimensions: {Mathematical justification of
  the Foldy--Lax model}}, Wave Motion, 50 (2013), pp.~18--28.

\bibitem{colton2003}
{\sc D.~Colton, H.~Haddar, and M.~Piana}, {\em The linear sampling method in
  inverse electromagnetic scattering theory}, Inverse Probl., 19 (2003),
  pp.~S105--S137.

\bibitem{colton1996}
{\sc D.~Colton and A.~Kirsch}, {\em A simple method for solving inverse
  scattering problems in the resonance region}, Inverse Probl., 12 (1996),
  pp.~383--393.

\bibitem{colton2018}
{\sc D.~Colton and R.~Kress}, {\em Looking back on inverse scattering theory},
  SIAM Rev., 60 (2018), pp.~779--807.

\bibitem{colton2019}
{\sc D.~Colton and R.~Kress}, {\em Inverse Acoustic and Electromagnetic
  Scattering Theory}, Springer, New York, 4th~ed., 2019.

\bibitem{colton1997}
{\sc D.~Colton, M.~Piana, and R.~Potthast}, {\em A simple method using
  {Morozov's} discrepancy principle for solving inverse scattering problems},
  Inverse Probl., 13 (1997), pp.~1477--1493.

\bibitem{duroux2010}
{\sc A.~Duroux, K.~Sabra, J.~Ayers, and M.~Ruzzene}, {\em Using
  cross-correlations of elastic diffuse fields for attenuation tomography of
  structural damage}, J. Acoust. Soc. Am., 127 (2010), pp.~3311--3314.

\bibitem{evans2010}
{\sc L.~C. Evans}, {\em Partial Differential Equations}, American Mathematical
  Society, Providence, second~ed., 2010.

\bibitem{gallot2011}
{\sc T.~Gallot, S.~Catheline, P.~Roux, J.~Brum, N.~Benech, and C.~Negreira},
  {\em Passive elastography: shear-wave tomography from physiological-noise
  correlation in soft tissues}, IEEE Trans. Ultrason. Ferroelectr. Freq.
  Control, 58 (2011), pp.~1122--1126.

\bibitem{montanelli2023}
{\sc J.~Garnier, H.~Haddar, and H.~Montanelli}, {\em The linear sampling method
  for random sources}, SIAM J. Imag. Sci., 16 (2023), pp.~1572--1593.

\bibitem{garnier2016}
{\sc J.~Garnier and G.~Papanicolaou}, {\em Passive Imaging with Ambient Noise},
  Cambridge University Press, Cambridge, 2016.

\bibitem{godin2010}
{\sc O.~A. Godin, N.~A. Zabotin, and V.~V. Goncharov}, {\em Ocean tomography
  with acoustic daylight}, Geophys. Res. Lett., 37 (2010), p.~L13605.

\bibitem{gouedard2008}
{\sc P.~Gou\'edard, L.~Stehly, F.~Brenguier, M.~Campillo, Y.~Colin~de
  Verdi\`ere, E.~Larose, L.~Margerin, P.~Roux, F.~J. Sanchez-Sesma, N.~M.
  Shapiro, and R.~L. Weaver}, {\em Cross-correlation of random fields:
  {M}athematical approach and applications}, Geophys. Prospect., 56 (2008),
  pp.~375--393.

\bibitem{koulakov2014}
{\sc I.~Koulakov and N.~Shapiro}, {\em Seismic Tomography of Volcanoes},
  Springer, Berlin, 2014, pp.~1--18.

\bibitem{lebedev1972}
{\sc N.~Lebedev and R.~Silverman}, {\em Special Functions and Their
  Applications}, Dover Publications, 1972.

\bibitem{montanelli2017phd}
{\sc H.~Montanelli}, {\em Numerical Algorithms for Differential Equations with
  Periodicity}, PhD thesis, University of Oxford, 2017.

\bibitem{montanelli2022}
{\sc H.~Montanelli, M.~Aussal, and H.~Haddar}, {\em Computing weakly singular
  and near-singular integrals over curved boundary elements}, SIAM J. Sci.
  Comput., 44 (2022), pp.~A3728--A3753.

\bibitem{montanelli2024a}
{\sc H.~Montanelli, F.~Collino, and H.~Haddar}, {\em Computing singular and
  near-singular integrals over curved boundary elements: {The strongly singular
  case}}, SIAM J. Sci. Comput., to appear (2024).

\bibitem{sabra2011}
{\sc K.~G. Sabra and S.~Huston}, {\em Passive structural health monitoring of a
  high-speed naval ship from ambient vibrations}, J. Acoust. Soc. Am., 129
  (2011), pp.~2991--2999.

\bibitem{shapiro2005}
{\sc N.~Shapiro, M.~Campillo, L.~Stehly, and M.~H. Ritzwoller}, {\em
  High-resolution surface-wave tomography from ambient seismic noise}, Science,
  307 (2005), pp.~1615--1618.

\bibitem{siderius2010}
{\sc M.~Siderius, H.~Song, P.~Gerstoft, W.~S. Hodgkiss, P.~Hursky, and C.~H.
  Harrison}, {\em Adaptive passive fathometer processing}, J. Acoust. Soc. Am.,
  127 (2010), pp.~2193--2200.

\bibitem{spence2015}
{\sc E.~A. Spence, I.~V. Kamotski, and V.~P. Smyshlyaev}, {\em Coercivity of
  combined boundary integral equations in high-frequency scattering}, Comm.
  Pure Appl. Math., 68 (2015).

\bibitem{trefethen2019}
{\sc L.~N. Trefethen}, {\em Approximation Theory and Approximation Practice},
  SIAM, Philadelphia, extended~ed., 2019.

\bibitem{woolfe2015}
{\sc K.~F. Woolfe, S.~Lani, K.~G. Sabra, and W.~A. Kuperman}, {\em Monitoring
  deep-ocean temperatures using acoustic ambient noise}, Geophys. Res. Lett.,
  42 (2015), pp.~2878--2884.

\bibitem{montanelli2015b}
{\sc G.~B. Wright, M.~Javed, H.~Montanelli, and L.~N. Trefethen}, {\em
  Extension of {C}hebfun to periodic functions}, SIAM J. Sci. Comput., 37
  (2015), pp.~C554--C573.

\end{thebibliography}

\end{document}